\documentclass[11pt,english]{article}

\usepackage{a4wide}
\usepackage{amssymb,amsthm}
\usepackage{amsmath}
\usepackage{amsfonts}
\usepackage{latexsym}
\usepackage{amsxtra}
\usepackage{graphicx}
\usepackage{wrapfig}
\usepackage{mathrsfs}
\usepackage{epsfig}
\usepackage{times}
\usepackage{makeidx}
\usepackage[titles]{tocloft}
\usepackage{cite}

\DeclareMathAlphabet{\mathpzc}{OT1}{pzc}{m}{it}

\newcommand{\detail}[1]{}
\newcommand{\heikodetail}[1]{}
\renewcommand{\proof}{\noindent{\sc Proof:}\hskip 1.1em}
\renewcommand{\qed}{\hfill\mbox{$\Box$}\\}
\def\yy{\mbox{$\spadesuit$}}
\newcommand{\invisible}[1]{\par ($\spadesuit$ \emph{hidden comments in \TeX{}
file\/})} 
\newcommand{\reallyinvisible}[1]{}      

\newfont{\thickmath}{msbm10 scaled \magstephalf}%
\newfont{\smallthickmath}{msbm7 scaled \magstephalf}%
\newfont{\footnotethickmath}{msbm8}%
\newfont{\footnotesmallthickmath}{msbm6}%

\newcommand{\cA}{{\mathcal{A}}}

\newcommand{\cC}{{\mathcal{C}}}

\newcommand{\cR}{{\mathcal{R}}}

\newcommand{\area}{{\textnormal{area}}} %domain
\newcommand{\sym}{{\textnormal{sym}}} %domain
\newcommand{\vol}{{\textnormal{vol}}} %domain

\newcommand{\hR}{{\widehat{\mathcal{R}}}}

\newtheorem{lemma}{\bf Lemma}[section]
\newtheorem{theorem}[lemma]{\bf Theorem}
\newtheorem{proposition}[lemma]{\bf Proposition}
\newtheorem{corollary}[lemma]{\bf Corollary}

\newtheorem{definition}[lemma]{\bf Definition}

\newcommand{\Fo}{\,\,\,\text{for }\,\,}
\newcommand{\Foa}{\,\,\,\text{for all }\,\,}

\newcommand{\AND}{\,\,\,\text{and }\,\,}

\newcommand\Reals{{\mathbb R}}
\newcommand\R{{\mathbb R}}
\newcommand\Z{{\mathbb Z}}

\newcommand\N{{\mathbb N}}
\renewcommand\S{{\mathbb S}}

\newcommand{\bbbr}{\Reals}

\newcommand\trace{\mathop{\rm trace}\nolimits}
\newcommand\rank{\mathop{\rm rank}\nolimits}
\renewcommand\span{\mathop{\rm span}\nolimits}

\newcommand\Id{{{\rm Id}}}

\newcommand{\xx}{\mbox{$\clubsuit$}}

\def\mbbbs{\mbox{\smallthickmath S}}

\renewcommand{\R}{\bbbr}

\newcommand{\eps}{\varepsilon}

\begin{document}

\renewcommand{\thefigure}{\arabic{figure}}

%%%%%%%%%%%%%%%%%%%%%%%%%%%%%%%%%%%%%%%%%%%%%%%%%%%%
%%%%%%%%%%%   title   %%%%%%%%%%%%%%%%%%%%%%%%%%%
%%%%%%%%%%%%%%%%%%%%%%%%%%%%%%%%%%%%%%%%%%%%%%%%
\title{\large\bf 
Plateau's problem in Finsler $\mathbf{3}$-space
}

\author{\normalsize Patrick Overath, Heiko von der Mosel}

%\date{\normalsize version of \today}

\maketitle

%%%%%%%%%%%%%%%%%%%%%%%%%%%%%%%%%%%%%%%%%%%%%%
%%%%%%%%%%%% abstract  %%%%%%%%%%%%%%%%%%%%%%%
%%%%%%%%%%%%%%%%%%%%%%%%%%%%%%%%%%%%%%%%%%%%%

\frenchspacing

\begin{abstract}
We explore a connection between the Finslerian area functional based
on the Busemann-Hausdorff-volume form, and
well-investigated Cartan functionals to solve Plateau's problem in
Finsler $3$-space, and prove higher regularity of solutions. 
Free and semi-free geometric boundary value problems, 
as well as the Douglas problem
in Finsler space can be dealt with in the same way. We also provide
a simple isoperimetric inequality for minimal surfaces in Finsler spaces.

\vspace{2mm}

\centering{Mathematics Subject Classification (2000): 44A12, 49Q05, 
49Q10, 53A35, 53B40, 53C60}

\end{abstract}

%\setcounter{tocdepth}{1}
%\tableofcontents
      
%\bigskip 

% Can remove table of contents any time we wish.
% Added to show the present structure of our draft

\renewcommand\theequation{{\thesection{}.\arabic{equation}}}
\def\setnumbers{\setcounter{equation}{0}}

%--------- NEW INTRO HERE

\section{Introduction}\label{sec:1}
The classic Plateau problem in Euclidean $3$-space is concerned with finding a 
minimal surface, i.e., a surface with vanishing  mean curvature, spanned in a
given closed Jordan curve $\Gamma\subset\R^3.$ A particularly successful approach
to this problem is to minimize the area functional
\begin{equation}\label{area_classic}
\area_{B}(X):=\int_B |(X_{u^1}\wedge X_{u^2})(u)|\,du
\end{equation}
in the class 
$$
\cC(\Gamma):=\{X\in W^{1,2}(B,\R^3):\,\textnormal{$X|_{\partial B}$ is a continuous
and weakly monotonic\footnotemark
parametrization of $\Gamma$}\},
$$
%\footnotetext{For a precise definition see \cite[pp. 231--232]{DHKW1}.}
where $B:=\{u=(u^1,u^2)\in\R^2:|u|=\sqrt{(u^1)^2+(u^2)^2}<1\}$ denotes the open unit disk and $W^{1,2}(B,\R^3)$
the class of Sobolev mappings from $B$ to $\R^3$ with square integrable first weak derivatives.
\footnotetext{See \cite[pp. 231--232]{DHKW1} for the notion of weakly monotonic
mappings on the boundary.}
There are various ways to obtain area minimizing surfaces. Courant \cite{courant_book}, e.g., minimized
the Dirichlet energy
\begin{equation}\label{dirichlet}
\mathscr{D}(X):=\frac 12 \int_B|\nabla X(u)|^2\,du
\end{equation}
as the natural and particularly simple {\it dominance functional} of $\area$. Outer variations of $\mathscr{D}$
establish harmonicity and therefore smoothness of the minimizer's coordinate functions $X^1, X^2, X^3$
on $B$, and inner variations
yield the conformality relations 
\begin{equation}\label{conformal}
|X_{u^1}|^2=|X_{u^2}|^2\quad\AND \quad X_{u^1}\cdot X_{u^2}=0\quad\textnormal{on $B$.}
\end{equation}
The combination of these properties leads to a simultaneous minimization of $\area$ and to
vanishing mean curvature of the minimizing surface. 
%which, however, may have
%branch points, i.e., parameters $\bar{u}\in B$ wih
%$|(X_{u^1}\wedge X_{u^2})(\bar{u})|=0.$ 
There is a huge amount of literature dealing with the classic Plateau problem and related
geometric boundary value problems in Euclidean space and also in Riemannian manifolds; 
see, e.g., the monographs \cite{nitsche1,nitsche2}, \cite{osserman}, \cite{DHKW1,DHKW2}, \cite{DHS1,DHT3,DHT2}, 
and the numerous references therein. 

Interestingly, nothing seems to be known about the Plateau problem for minimal surfaces in Finsler manifolds,
not even in Finsler spaces, which may have to do with the by far more complicated expression for the
Finsler-area functional that does not seem to allow a straightforward generalization of Courant's method
via minimization of appropriately chosen dominance functionals. It is the purpose of this note to attack
Plateau's problem in Finsler $3$-space by an alternative variational approach directly minimizing Finsler area.

\medskip

For the precise definition of Finsler area let
$\mathscr{N}=\mathscr{N}^n$ be an $n$-dimensional smooth
manifold with tangent bundle $T\mathscr{N}:=\bigcup_{x\in\mathscr{N}}T_x\mathscr{N}$ and its zero-section $o:=\{(x,0)\in T\mathscr{N}\}$.
A non-negative function $F\in C^\infty(T\mathscr{N}\setminus o)$ is called
a {\it Finsler metric} on $\mathscr{N}$ (so that $(\mathscr{N},F)$ becomes
a {\it Finsler manifold}) if $F$ satisfies the conditions
\begin{enumerate}
\item[\rm (F1)] $F(x,ty)=tF(x,y)$ for all $t>0$ and all $(x,y)\in T\mathscr{N}$ (homogeneity);
\item[\rm (F2)] $g_{ij}(x,y):=\big(F^2/2)_{y^iy^j}(x,y)$ form the coefficients
of a positive definite matrix, the {\it fundamental tensor,}
for all $(x,y)\in
T\mathscr{N}\setminus o $, where  for  given local coordinates $x^1,\ldots,x^n$ about
$x\in \mathscr{N}$, the $y^i,$ $i=1,\ldots,n$, denote the corresponding
bundle coordinates via $y=y^i\frac{\partial}{\partial  x^i}|_x\in 
T_x\mathscr{N}.$ Here we sum over repeated Latin indices from $1$ to $n$
according to the Einstein summation convention, and
$F(x,y)$ is written as $F(x^1,\ldots,x^n,y^1,\ldots,y^n)$.
\end{enumerate}
If $F(x,y)=F(x,-y)$ for all $(x,y)\in T\mathscr{N}$ then $F$ is
called a {\it reversible} Finsler metric, and if there are coordinates such that
$F$ depends only on $y$, then $F$ is called a {\it Minkowski metric}.

Any $C^2$-immersion $X:\mathscr{M}^m\hookrightarrow \mathscr{N}^n$ from a
smooth 
$m$-dimensional manifold $\mathscr{M}=\mathscr{M}^m$
into
$\mathscr{N}$ induces a {\it pulled-back Finsler metric} $X^*F$ on $\mathscr{M}$
via
$$
(X^*F)(u,v):=F(X(u),dX|_u(v))\quad\textnormal{for $(u,v)\in T\mathscr{M}.$}
$$
Following Busemann \cite{busemann-1947} and Shen \cite{shen98} we define
the {\it Busemann-Hausdorff volume form} as the volume ratio of the
Euclidean and the Finslerian unit ball, i.e.,
$$
d\textnormal{vol}_{X^*F}(u):=\sigma_{X^*F}(u)du^1\wedge \ldots\wedge du^m
\quad\textnormal{on $\mathscr{M}$,}
$$
where
\begin{equation}\label{volumefactor}
\sigma_{X^*F}(u):=\frac{\mathscr{H}^m(B_1^m(0))}{\mathscr{H}^m(
\{v=(v^1,\ldots, v^m)\in\R^m:X^*F(u,v^\delta\frac{\partial}{\partial u^\delta|_u})\le 1\}},
\end{equation}
with a summation over  Greek indices from $1$ to $m$ in the denominator.
Here $\mathscr{H}^m$ denotes the $m$-dimensional Hausdorff-measure.
The {\it Busemann-Hausdorff area} or in short {\it Finsler area}\footnote{Notice that the alternative Holmes-Thompson volume form (see \cite{alvarez-berck-wrong-hausdorff-2006}) leads to a different notion of Finslerian minimal surfaces that we do not
address here.} of the immersion $X:\mathscr{M}\to\mathscr{N}$ is then given
by
\begin{equation}\label{area}
\area^F_\Omega(X):=\int_{u\in\Omega}\,d\vol_{X^*F}(u)
\end{equation}
for a measurable subset $\Omega\subset\mathscr{M}.$
Shen \cite[Theorem 1.2]{shen98}  derived
the first variation of this functional which leads to the definition
of Finsler-mean curvature, and critical immersions for $\area^F_\Omega$
are therefore {\it Finsler-minimal immersions}, or simply minimal surfaces
in $(\mathscr{N},F).$

As mentioned before, to the best of our knowledge, there is no contribution
to solving the Finslerian Plateau problem or any  other related geometric boundary
value problems for Finsler-minimal surfaces, 
such as the Douglas problem (with boundary contours with at least two components),
free, or semi-free problems (prescribing a supporting set for part of
the boundary values). What little
is known about Finsler-minimal graphs, or rotationally symmetric Finsler-minimal surfaces
for very specific Finsler structures, 
will be briefly described at the end of this introduction when we discuss how sharp our 
additional assumptions on a general Finsler metric are. 

To describe our variational approach to Finsler-minimal surfaces let us focus
on {\it Finsler spaces} and on co-dimension one, that is, $\mathscr{N}:=
\R^{m+1}.$ 

The  key observation -- in its original form  due to 
H. Busemann \cite[Section 7]{busemann-1947}  
in his search
for
explicit volume formulas for intersection bodies in convex analysis --
is, that one can rewrite the integrand \eqref{volumefactor} of Finsler area
in the following way.
\begin{theorem}[Cartan area integrand]\label{thm:1.1}
If $\mathscr{N}=\R^{m+1}$ and $F=F(x,y)$ is a Finsler metric on $\R^{m+1}$, and
$X\in C^1(\mathscr{M},\R^{m+1})$ is an immersion from a smooth $m$-dimensional
manifold $\mathscr{M}$ 
into $\R^{m+1}$, then we obtain for the Finsler area of an open subset $\Omega\subset\mathscr{M}$
with local coordinates $(u^1,\ldots,u^m):\Omega\to\tilde{\Omega}\subset\R^m$ the 
expression
\begin{equation}\label{area_space}
\area^F_\Omega(X)=\int_{\tilde{\Omega}}\cA^F(X(u),\big(\frac{\partial X}{\partial u^1}
\wedge
\ldots\wedge\frac{\partial X}{\partial u^m}\big)(u))\,du^1\wedge\ldots\wedge du^m,
\end{equation}
where 
\begin{equation}\label{AF_space}
\cA^F(x,Z)=\frac{|Z|\mathscr{H}^m(B_1^m(0))}{\mathscr{H}^m(\{T\in Z^\perp\subset\R^{m+1}:
F(x,T)\le 1\})}\quad\Fo (x,Z)\in\R^{m+1}\times (\R^{m+1}\setminus\{0\}).
\end{equation}
\end{theorem}
The integrand in \eqref{area_space} depends on the position vector $X(u)$ and the
normal direction $(X_{u^1}\wedge\ldots\wedge X_{u^m})(u)$, and from the specific
form \eqref{AF_space} one immediately deduces the positive homogeneity
in its second argument:
\begin{equation}\label{H}
\cA^F(x,tZ)=t\cA^F(x,Z)\quad\Foa (x,Z)\in\R^{m+1}\times(\R^{m+1}\setminus\{0\}), \,t>0.
\tag{H}
\end{equation}
These properties identify $\cA^F$ as a {\it Cartan integrand}; see \cite[p. 2]{HilvdM-parma}.
Notice that if the Finsler metric $F$ equals the Euclidean metric
$E$, that is, $F(x,y)=E(y):=|y|$ for $y\in\R^{m+1}$, then
the expression $\cA^F=\cA^E$ reduces to the classic area integrand
for hypersurfaces in $\R^{m+1}$:
$$
\cA^F(x,Z)=\cA^E(x,Z)=|Z| \Foa (x,Z)\in \R^{m+1}\times (\R^{m+1}\setminus
\{0\}).
$$
%Now we proceed as in the classic case of minimal surfaces in Euclidean
%spaces: to apply variational methods we use the observation that the
%integral in \eqref{area_space} makes perfect sense also for surfaces
%of Sobolev class, and we can look at the following {\it Plateau
%problem in Finsler spaces}:
%
%formulate (P)
Geometric boundary value  
problems for Cartan functionals on two-dimensional surfaces
have been investigated under two additional 
%\footnote{Check how much ellipticity
%Brian White actually needs for his smooth existence in extreme boundary contours!}
conditions; see 
\cite{HilvdM-calcvar,HilvdM-courant,HilvdM-partially,HilvdM-parma,
HilvdM-douglas}: 
a mild linear growth condition,
which in the present situation can be guaranteed by  relatively harmless 
$L^\infty$-bounds 
on the underlying Finsler structure $F$ (see condition \eqref{growth} in
our existence result, Theorem \ref{thm:plateau} below), and, more
importantly, convexity in the second argument. So,
the question arises: Is there a chance to find a sufficiently large and interesting
class of Finsler structures $F$ so that the Cartan area integrand $\cA^F$
is convex in its second argument? 
It turns out that in the  the co-dimension one case and for
{\it  reversible} Finsler metrics, there
is a result also due to Busemann \cite[Theorem II, p. 28]{busemann-convex-brunn-minkowski-1949}
(see Theorem \ref{thm:convexity}), establishing this convexity. 
Thus, for general non-reversible Finsler metrics $F$  one 
is lead to think about some sort of  symmetrization of $F$ in its second
argument in order to have a chance to apply Busemann's result at some stage.
Rewriting the integrand $\cA^F$ by means of the area formula and using
polar coordinates (see Lemma \ref{lem:polar}) motivates the following 
particular kind of  symmetrization, the
{\it $m$-harmonic 
symmetrization\footnote{A possible
connection to the harmonic symmetrization of weak Finsler structures in \cite{papadopoulos-troyanov}
remains to be investigated.} $F_\textnormal{sym}$} of the Finsler structure $F$
defined as
\begin{equation}\label{m-harmonic}
F_{\textnormal{sym}}(x,y):= \left[\frac{2}{\frac{1}{F^m(x,y)}+\frac{1}{F^m(x,-y)}}\right]^{\frac{1}{m}}\quad\Fo (x,y)\in T\mathscr{N}\setminus o,
\end{equation}
which by definition and by (F1) is an even  and 
 positively $1$-homogeneous function of the $y$-variable,  and thus continuously
extendible by zero to all of $T\mathscr{N}$. Moreover, one can check that $F_{\textnormal{sym}}$ leads to the same expression of the Cartan integrand, i.e., 
$\cA^F=\cA^{F_{\textnormal{sym}}}$ (see Lemma \ref{lem:AF=AFsym}). However,
in general $F_{\textnormal{sym}}$ is {\it not} 
a Finsler structure.

This motivates our {\bf General Assumption: }
\begin{enumerate}
\item[\rm\bf (GA)]
{\it Let $F(x,y)$ be a Finsler metric on $\mathscr{N}=\R^{m+1}$ such that its
$m$-harmonic symmetrization
$F_\sym(x,y)$ is also a Finsler metric on $\R^{m+1}$.}
\end{enumerate}
Notice that a reversible Finsler metric $F$ automatically coincides with
its $m$-harmonic symmetrization $F_\sym$ so that our general assumption
(GA) is superfluous in reversible Finsler spaces.

\medskip

%The proof of this crucial lemma relies on the interpretation of $\cA^F$ as
%the image of the Finsler structure $F$ under the {\it spherical Radon transform},
%a detailed analysis of which is presented in Section \ref{sec:3}.
This leads to the following existence result.
\medskip

\begin{theorem}[Plateau problem for Finsler area]\label{thm:plateau}
Let 
$F=F(x,y)$ be a Finsler metric on $\R^3$ satisfying {\rm (GA)},
and assume in addition that
\begin{equation}\label{growth}
\tag{D*}
0<m_F:=\inf_{\R^3\times \S^2} F(\cdot,\cdot)\le \sup_{\R^3\times \S^2} F(\cdot,\cdot)=:M_F<\infty.
\end{equation}
Then for any given rectifiable Jordan curve $\Gamma\subset\R^3$ there exists a
surface $X\in\mathcal{C}(\Gamma)$, such that
$$
\area_B^F(X)=\inf_{\mathcal{C}(\Gamma)}\area_B^F(\cdot).
$$
In addition, one has the conformality relations
\begin{equation}\label{conf}
|X_{u^1}|^2=
|X_{u^2}|^2\quad\textnormal{and}\quad
X_{u^1}\cdot 
X_{u^1} = 0\quad\textnormal{$\mathscr{L}^2$-a.e. on $B$,}
\end{equation}
and $X$ is of class $C^{0,\sigma}(B,\R^3)\cap C^0(\bar{B},\R^3)\cap
W^{1,q}(B,\R^3)$ for some $q>2$, and for $\sigma:=(m_F/M_F)^2\in (0,1].$
\end{theorem}

\detail{

\bigskip

Achtung,  der Fall $\sigma=1$ f\"uhrt tats\"achlich  auf Lipschitzstetigkeit:
Zumindest die Saetze in \cite[p.298 \& p.165]{gilbarg-trudinger_1998} scheinen anwendbar. Aus der Wachstumsbedingung ergibt sich
\begin{equation}
 \int_{B_R(z)} |Du|^2\mathrm{d}x \le C R^{2\gamma}
\end{equation}
gerade mit Lemma 12.2 in \cite{gilbarg-trudinger_1998}
\begin{equation}
 \int_{B_R(z)} |Du|\mathrm{d}x \le (\int_{B_R(z)} |Du|^2\mathrm{d}x)^{\frac{1}{2}} (\mathscr{H}^2(B_R(z)))^{\frac{1}{2}} \le C R^{\gamma +1} = C R^{2-1 +(\gamma -1+1)} = C R^{1 + \gamma}
\end{equation}
und damit nach \cite[Theorem 7.19]{gilbarg-trudinger_1998}
$u\in C^{0,\gamma}_{\mathrm{loc}}(B)$, wobei fuer jeden Ball $B_R(z)\subset B$ folgt
\begin{equation}
 \mathrm{osc}_{B_R(z)} u \le C R^\gamma.
\end{equation}
$\gamma\le 1$ ist dabei zulaessig. Hierin sind alle Konstanten als generisch zu betrachten.
}

A simple comparison argument leads to the following
isoperimetric inequality for area-minimizing surfaces in Finsler space:
\begin{corollary}[Isoperimetric Inequality]\label{cor:isop}
Let $F(x,y)$ be a Finsler metric on $\R^3$ satisfying the growth condition
\eqref{growth} in Theorem \ref{thm:plateau}. Then
any minimizer $X\in\cC(\Gamma)$ of Finsler area $\area^F_B$ satisfies 
the simple isoperimetric inequality
\begin{equation}\label{isop}
\area^F_B(X)\le \frac{M_F^2}{4\pi m_F^2}\Big(\mathscr{L}^F(\Gamma)\Big)^2,
\end{equation}
where $\mathscr{L}^F:=\int F(\Gamma,\dot{\Gamma})$ denotes the
Finslerian length of $\Gamma$. 
\end{corollary}

{\bf Remarks.}\,
1.\,
Assumption \eqref{growth} in Theorem \ref{thm:plateau} is automatically satisfied in case of
a Minkowski metric $F=F(y)$ since the defining properties (F1) and (F2)
guarantee that any Finsler metric is positive away from the zero-section
of the tangent bundle, so that the positive minimum and maximum of any
Minkowski metric on the unit sphere is attained.

2.\,
If $\Gamma$ satisfies a chord-arc condition  with respect to a three-point
condition
(see, e.g. \cite[Theorem 5.1]{HilvdM-parma})
one can establish H\"older continuity of the minimizer in Theorem \ref{thm:plateau}
up to the boundary in form of an a priori estimate, 
a fact that  is well known for classic minimal surfaces in Euclidean space.

3.\,
One can use the bridge between the Finsler world and Cartan functionals 
established here in the same
way
to prove existence of Finsler-minimal surfaces solving other
geometric  boundary value problems like {\it free}, or {\it semi-free} problems, 
where the boundary
or parts of the boundary are prescribed to be mapped to a given {\it supporting set}, such
as a given torus, possibly with additional topological constraints (like spanning the
hole of the torus). For the solution of such geometric boundary value problems for
Cartan functionals
see \cite{HilvdM-partially}, or \cite{dahmen-diplomathesis}. One can also prescribe
more than one boundary curve and control the topological connectedness of Finsler minimal
surfaces spanning these more complicated boundary contours
under  the so-called {\it Douglas condition} in the famous
{\it Dougals problem}; see \cite{Kurzke-vdM}, \cite{HilvdM-douglas} for details in
the context of Cartan functionals.

4.\,
Using the full strength of \cite[Theorem II]{busemann-convex-brunn-minkowski-1949}
one can extend the existence result, Theorem \ref{thm:plateau} to continuous {\it weak Finsler metrics}
as defined in \cite{papadopoulos-troyanov},
\cite{papadopoulos-troyanov-2009} or \cite{matveev-rademacher-troyanov-zeghib-2009}, 
 assuming also in our
general assumption (GA)  that $F_\sym$ is merely a continuous weak Finsler metric
which is only convex in its second entry.

\medskip

Notice that the Finsler-area minimizing surfaces
$X$ obtained in Theorem \ref{thm:plateau}
are in general not immersed. {\it Branch points}, i.e.,
parameters $\bar{u}\in B $ with $(X_{u^1}\wedge X_{u^2})(\bar{u})=0$ may
occur, and it is an open question under what circumstances area-minimizing
surfaces in Finsler space (that are not graphs) are  immersed. This has
to do with the fact that the area-minimizers in Theorem \ref{thm:plateau}
are obtained as solutions of the Plateau problem for the corresponding
Cartan functional. 
   General Cartan functionals, however,
 do not possess nice Euler-Lagrange equations, in contrast
  to the elliptic pde-systems in diagonal form obtained in the classic cases of
  minimal surfaces or surfaces of prescribed mean curvature in Euclidean space,
  or even in Riemannian manifolds. Due to this lack of accessible variational 
  equations it is by
  no means obvious how to exclude branch points.
Intimately connected to this is the issue of possible higher regularity 
of 
   Finsler-area minimizing
  surfaces. This is a delicate problem and,
  in view of the current state of research
  depends on whether the corresponding Cartan functional possesses a so-called
  {\it perfect dominance function}. Such a function is, roughly speaking,
  a Lagrangian $G(z,P)$, that is positively $2$-homogeneous, $C^2$-smooth, and strictly
  convex in   
  $P\in\R^{3\times 2}\setminus\{0\}$, and  that
 dominates  the Cartan integrand, and coincides with
 it on conformal entries; see Definition 
  \ref{def:perfectdom}. 
  It was shown in \cite{HilvdM-courant,HilvdM-partially},
  and \cite{HilvdM-crelle} that minimizers of Cartan functionals with
  a perfect
  dominance function are of class $W^{2,2}$ and $C^{1,\alpha}$
  up to the boundary.  According to \cite[Theorem 1.3]{HilvdM-dominance} 
  there is a fairly large
  class of Cartan integrands with a perfect dominance 
  function,  and we are going to exploit this quantitative result in the 
  present context to prove the following
  theorem about higher regularity of Finsler-area minimizing surfaces. 

  For the precise statement we introduce for $k=0,1,2,\ldots $ and functions
  $g\in C^k(\R^{3}\setminus\{0\})$ the
  semi-norms 
  \begin{equation}\label{rho_k}
  \rho_k(g):=\max\{|D^\alpha g(\xi)|:\xi\in\S^2,|\alpha|\le k\}.
  \end{equation}
%  and set 
% \begin{equation}\label{rho_hat} 
%  \hat{\rho}_k(g):=\max\{1,\rho_k(g)\}.
% \end{equation} 
\begin{theorem}[Higher regularity]\label{thm:higher_reg}
There is a universal constant $\delta_0\in (0,1)$ 
such that any Finsler-area minimizing
and conformally parametrized (see \eqref{conf}) surface $X\in\cC(\Gamma)$ is of
class $W^{2,2}_\textnormal{loc}(B,\R^3)\cap 
C^{1,\alpha}({B},\R^3)$ if the Finsler-structure $F=F(x,y)$ satisfies
\begin{equation}\label{suff_reg_condition}
\rho_2(F(x,\cdot)-|\cdot|)<\delta_0\quad\Foa x\in \R^{3}.
\end{equation}
Moreover, if in addition 
the boundary contour $\Gamma $ is of class $C^4$ one obtains
$X\in W^{2,2}(B,\R^3)\cap C^{1,\alpha}(\bar{B},\R^3)$ and a constant $c=c(\Gamma)$
depending only 
%on the Finsler structure $F$ and 
on $\Gamma$ such that
$$
\|X\|_{W^{2,2}(B,\R^3)}
+
\|X\|_{C^{1,\alpha}(\bar{B},\R^3)}\le c(\Gamma).
$$ 
\end{theorem}
Notice that condition \eqref{suff_reg_condition} may be relaxed for the minimizers $X$
of Finsler area obtained in Theorem \ref{thm:plateau}: If $\Gamma$ satisfies a
chord-arc condition, we obtain a priori estimates on the H\"older norm of $X$ on
$\bar{B}$, and therefore uniform $L^\infty$-bounds $\|X\|_{L^\infty(B,\R^3)}\le R_0$,
so that it is sufficient to assume the inequality in \eqref{suff_reg_condition}
only for all $x\in \overline{B_{R_0}(0)}\subset\R^3.$

\medskip

Let us finally discuss our crucial general assumption (GA). Is it a natural
assumption, and how restrictive is it? Since generalized, i.e., possibly
branched Finsler-minimal surfaces have apparently not been treated
in the literature so far, we return to Finsler-minimal immersions, for which the
connection between Finsler area and Cartan integrals
turns out to be very useful to obtain a whole set of new global results such as 
Bernstein theorems, enclosure results, uniqueness results, removability of singularities,
and new isoperimetric inequalities; see \cite{overath-vdM-2012b}.  Also these
results require  
(GA) as the only essential assumption, and they extend the few results
in the literature about Finsler-minimal graphs, that had been established   so far
only
in very specific Finsler spaces.  Souza, Spruck and  Tenenblat considered
the three-dimensional 
{\it Minkowski-Randers space} 
$(\R^3,F)$, where $F$ has the special form
$F(y):=|y|+b_iy^i$ for some constant
 vector $b\in \R^3$, and  they 
 used pde-methods in 
\cite{sst} to prove that any Finsler-minimal graph over a plane in that space
is a plane if and only if
$0\le |b|<1/\sqrt{3}.$ This upper bound on the linear perturbation 
$|b|$ is indeed sharp, since for $|b|\in (1/\sqrt{3},1)$, where 
$(\R^3,F)$ is still a Finsler space (see
e.g. \cite[p. 4]{chern-shen-2005}), Souza and Tenenblat have
presented a  Finsler-minimal cone with a point singularity\footnote{Technically,
the bound $1/\sqrt{3}$ for $|b|$ is the threshold beyond which the underlying
pde ceases to be an elliptic equation; see \cite[p. 300]{sst}}.%Examples, (34), p. 300,  
This Bernstein theorem was later generalized by Cui and Shen \cite{cui-shen-2009}
to the more general setting of $(\alpha,\beta)$-Minkowski spaces 
$(\R^{m+1},F) $ with $F(y):=\alpha(y)\phi\big[\beta(y)/\alpha(y)\big]$
with $\alpha(y):=|y|$ and  the linear perturbation term
$\beta(y):=b_iy^i$, and a positive smooth scalar function
$\phi$ satisfying a particular differential equation to guarantee
that $F$ is at least a Finsler metric; see e.g. \cite[Lemma 1.1.2]{chern-shen-2005}.
Cui and Shen present fairly complicated additional and more restrictive
conditions on $\phi$ (see condition 
(1) in \cite[Theorem 1.1]{cui-shen-2009} or condition (4) of
\cite[Theorem 1.2]{cui-shen-2009}) that could be verified only for
a few 
specific choices of $(\alpha,\beta)$-metrics, and only in dimension $m=2$: 
for the Minkowski-Randers case with
$\phi(s)=1+s$ if $|b|<1/\sqrt{3}$ (reproducing \cite[Theorem 6]{sst}),
 for the {\it two-order metric} with $\phi(s)=(1+s)^2$ 
under the condition $|b|<1/\sqrt{10}$, or for the {\it Matsumoto metric}
where $\phi(s)=(1-s)^{-1}$ if $|b|<1/2.$
\detail{

\bigskip

{\tt In \cite{cui-shen-2009} $|b|<1/2$ is set for the Matsumoto metric.
In \cite{cui-shen-2009} is only metioned, that the bound on $b$ regarding ellipticity, as in the application of their theorems to the Matsumoto metric, is $|b|<1/2$!}, due to the fact that the Matsumoto metric is only Finslerian for $|b|<\frac{1}{2}$, but their ellipticity condition is fulfilled even for $|b|<1$ as it is also the case for the symmetrization of the Matsumoto metric, which is Finslerian for $|b|<1$.

\bigskip

} 
By direct calculation one can 
check that the threshold values for $|b|$
in these specific $(\alpha,\beta)$-spaces are {\it exactly} those under which 
our general assumption (GA) is automatically satisfied -- (GA) does not hold
if $|b|$ is larger. Moreover,
beyond these threshold values\footnote{For the Matsumoto metric
the threshold value for $|b|$, beyond which the pde fails to be elliptic
and $F_\sym$ is no longer Finsler, is actually $1$, but the metric itself is only 
Finsler for $|b|<1/2.$} 
Cui and Shen have established the existence of
Finsler-minimal 
cones with a point singularity in the respective $(\alpha,\beta)$-spaces, 
which indicates that
 our assumption (GA) for general Finsler metrics
 is not only  natural but also sharp.
In addition, Cui and Shen present an example of an $(\alpha,\beta)$-metric $F$
allowing a Bernstein result, where $\phi$ is of the form $(1+h(s))^{-1/m}$
with an arbitrary odd smooth function $h$ with $|h |<1$, but also in this
case (GA) is trivially satisfied, since one can check that $F_{\textnormal{sym}}(y)
=|y|$. 

For general Finsler metrics $F(x,y)$ our general assumption (GA) may not
be verified easily. Therefore we conclude 
with a sufficient condition that guarantees that this assumption holds. 
This condition involves the {\it arithmetic symmetrization} (with respect to $y$)
$
F_s(x,y):=\frac 12 (F(x,y)+F(x,-y))
$
with its fundamental tensor
$$
(g_{F_s})_{ij}:=(F_s^2/2)_{y^iy^j},
%\quad\textnormal{and its inverse $(g_{F_s})^{ij}$},
$$
and  {\it antisymmetric part} $F_a$ of $F$ given by
$
F_a(x,y):=\frac 12 (F(x,y)-F(x,-y)).
$
\begin{theorem}\label{thm:sufficient}
If the Finsler metric $F=F(x,y)$ on $\R^{m+1}$ 
with its arithmetic symmetrization $F_s$
and its antisymmetric part $F_a$ satisfies the inequality
\begin{equation}\label{suff_ineq}
%\|\nabla F_a(x,y)\|_{F_s}:=\sqrt{(g_{F_s})^{ij}(x,y)(F_a)_{y^i}(x,y)
%(F_a)_{y^j}(x,y)} < \frac{1}{\sqrt{m+1}}
((F_a)_{y^l}(x,y)w^l)^2<\frac{1}{m+1}(g_{F_s})_{ij}(x,y)w^iw^j\Foa w\in\R^{m+1},
\end{equation}
and if the matrix $(F_a(x,y)(F_a)_{y^iy^j}(x,y))$  is negative semi-definite 
for all $x\in\R^{m+1}$ and $y\in
\R^{m+1}\setminus\{0\}$, then $F$ satisfies assumption {\rm (GA)}.
\end{theorem}
Notice that the second condition is, of course, satisfied if the antisymmetric
part $F_a$ is linear in $y$, which is, e.g., the case for the Minkowski-Randers 
metric $F(y)=|y|+b_iy^i.$
In that case, inequality \eqref{suff_ineq}  yields exactly the bound $1/\sqrt{3}$ in
dimension $m=2$ that was also necessary
to deduce (GA) directly, and this bound is sharp in the sense discussed before. The same
holds true for the two-order metric $F(y)=\alpha(y)\phi
\big[\beta(y)/\alpha(y)\big]$, $\beta(y)=
b_iy^i$ for
$\phi(s)=(1+s)^2$, if $|b|\in [0,1/\sqrt{10})$,
but our sufficient condition, on the other hand, does not include the Matsumoto
metric $\phi(s)=(1-s)^{-1}$, although we can directly verify (GA) for that metric
if $|b|<1/2.$ \detail{\tt AGAIN: MATSUMOTO NEEDS TO BE CHECKED \yy See remark on this at p.8\yy \xx\xx WAS KOMMT DENN F\"UR MATSUMOTO AUS UNSERER
BEDINGUNG HERAUS, VIELLEICHT DER WERT $1/2$?\xx\xx \yy Da in (GA) $F$ als Finslersch vorausgesetzt wird, kommt da $|b|<1/2$ raus. Wenn wir fuer $F$ selbst nur $F>0$ auf $\R^{m+1}\backslash \{ 0\}$, $1$-homogen und stetig auf $\R^{m+1}$ fordern, sowie $F_{\sym}$ Finslersch, so erhielten wir $|b|<1$. \yy}
Theorem \ref{thm:sufficient} does, however, allow for more general
Finsler structures because it permits an $x$-dependence, e.g., $F(x,y):=F_r(x,y)+b_iy^i,$
where $F_r$ is a reversible Finsler metric. Even without the $x$-dependence our result
is valid for more general Minkowski metrics than treated before, 
for instance the
{\it perturbed quartic metric} (see \cite[p. 15]{bao-chern-shen_2000})
$$
F_r(y):=\sqrt{\sqrt{\sum_{i=1}^{m+1} (y^i)^4} + 
\varepsilon \sum_{i=1}^{m+1} (y^i)^2}\quad\Fo\varepsilon>0.
$$

\detail{

\bigskip

About permissible $x$-dependence only in $F_s$ for our 
sufficient condition: $F(x,y):=F_s(x,y) + \langle b, y\rangle$ is a possible choice, 
where $F_s(x,y)$ is a reversible (possibly Minkwoskian) Finsler metric. 
Choose $F_s(y)$, for instance, to be a perturbed quartic metric 
$F_s(y) := \sqrt{\sqrt{\sum_{i=1}^{m+1} (y^i)^4} + 
\varepsilon \sum_{i=1}^{m+1} (y^i)^2}$ for $\varepsilon>0$.  

\bigskip

}

\medskip

The present paper is structured as follows.
In Section \ref{sec:2.1} 
we explore the connection between Finsler area and Cartan functionals and prove 
Theorem
\ref{thm:1.1}. In addition, we represent the Cartan integrand $\cA^F$
with an integral formula (Lemma \ref{lem:polar}) which motivates
the $m$-harmonic symmetrization. That $F$ and $F_{\textnormal{sym}}$ possess
the same Cartan area integrand is shown in Lemma \ref{lem:AF=AFsym}. Some
quantitative $L^\infty$-estimates and Busemann's convexity result (Theorem 
\ref{thm:convexity}) lead 
 to the solution of  Plateau's  problem, i.e. to the proof of Theorem
\ref{thm:plateau} in Section \ref{sec:2.2}. 
In Section \ref{sec:3.1} we introduce and analyze the (spherical)
Radon transform
since one may express the Cartan area integrand $\cA^F$  in terms
of this transformation (see Lemma \ref{cor:cartan-radon}). The material
of this section, however, will also be useful for our investigation on 
Finsler-minimal immersions; see \cite{overath-vdM-2012b}. In Section 
\ref{sec:3.2} we  compare   
 $\cA^F$ and its derivatives up to second order with those of
 the classic area integrand  in order
to apply the regularity theory for minimizers of Cartan functionals that is
based on the concept of perfect dominance functions. Towards the end
of Section \ref{sec:3.2} we prove Theorem
\ref{thm:higher_reg};  the lengthy calculation for the proof of Theorem
\ref{thm:sufficient} is deferred to Section \ref{sec:4}.
%The results of Sections \ref{sec:2} and \ref{sec:3} 
%constitute the foundation for the proofs of all theorems stated in this introduction; see
%Section \ref{sec:4}. For more  refined results involving, e.g.,  $x$-dependence of the Finsler 
%structure or more general target manifolds $\mathscr{N}$
%we refer to the Ph.D. thesis \cite{overath-phd_2012} of the first author. 

\medskip

{\bf Acknowledgment.}\, 
Part of this 
work was completed during the second author's stay at Tohoku University at Sendai, Japan, in the Spring of 2011, 
and we would like to express our deepest gratitude to Professor Seiki Nishikawa for his
hospitality and interest. Some of the results are contained in the 
first author's thesis \cite{overath-phd_2012}, who was
partially supported by DFG grant Mo 966/3-1,2.

\setnumbers
\section{Existence of Finsler-minimal surfaces}
\label{sec:2}
\subsection{Representing Finsler area as a Cartan functional}\label{sec:2.1}
The Finsler area \eqref{area} is by definition a parameter invariant integral, which
implies by virtue of a general result of Morrey \cite[Ch. 9.1]{morrey66}
 in co-dimension one, i.e., for $n=m+1$, that the Finsler-area integrand 
 \eqref{volumefactor} has special structure. In our context we are interested
 in the explicit form of that structure. To deduce that structure we take for any point
 $\xi\in\mathscr{N}$ an open neighbourhood $W_\xi\subset\mathscr{N}$ containing
 $\xi$ such that there is a smooth basis section $\{b_i\}_{i=1}^{m+1}$ in
 the tangent bundle $TW_\xi\subset T\mathscr{N}$ and then a local coordinate chart
 $u^1,\ldots,u^m$ on a suitable open neighbourhood  $\Omega_\xi\subset\mathscr{M}$ such
 that the given immersion $X:\mathscr{M}\to\mathscr{N}$ satisfies $X(\Omega_\xi)\subset W_\xi$.
 Now we can express its differential $dX:T\mathscr{M}\to T\mathscr{N}$ locally as
 $$
 dX=X^i_\delta du^\delta\otimes b_i,
 $$
 and we set $\nabla X(u):=(X^i_\delta(u))\in\R^{(m+1)\times m}$ for any $u\in\Omega_\xi,$
 so that we obtain from \eqref{volumefactor}
 $$
 \sigma_{X^*F}(u)=\frac{\mathscr{H}^m(B_1^m(0))}{
 \mathscr{H}^m(\{v\in\R^m:F(X(u),v^\delta X^i_\delta b_i|_{X(u)})\le 1\})}=:a_\xi^F(X(u),\nabla X(u)),
 $$
 where  for $x\in W_\xi$ and $P=(P^i_\delta)\in\R^{(m+1)\times m}$ we have set
 \begin{equation}\label{aF_definition}
 a_\xi^F(x,P):=\begin{cases}\frac{\mathscr{H}^m(B_1^m(0))}{
 \mathscr{H}^m(\{v\in\R^m:F(x,v^\delta P^i_\delta b_i|_x)\le 1\})}  &\quad \textnormal{if
 $\rank P=m$,}\\
 0 & \quad\textnormal{if
  $\rank P<m$.}
  \end{cases}
 \end{equation}
 The following result was probably first shown by Busemann \cite{busemann-1947}; cf. \cite[Chapter 7, p. 229]{thompson-1996}.
 \begin{proposition}\label{prop:AF}
 Let $(\mathscr{N},F)$ be a Finsler manifold of dimension $n=m+1$. For $\xi\in\mathscr{N}$,
 $W_\xi\subset\mathscr{N}$, and
 a basis section $\{b_i\}_{i=1}^{m+1}$ on $TW_\xi\subset T\mathscr{N}$ chosen as above one can
 write
 $$
 a_\xi^F(x,P)=
 \cA_\xi^F(x,P_1\wedge \ldots \wedge P_m)\quad\textnormal{for $x\in W_\xi$ and $P=(P_1|P_2|\ldots
 |P_m)\in\R^{(m+1)\times m}$},
 $$
 where $P_\delta=(P^i_\delta)_{i=1}^{m+1}\subset\R^{m+1}$ for $\delta=1,\ldots,m,$ 
 denote the column vectors of the matrix
 $P$, and the wedge product is given as usual by
 $$
 P_1\wedge\ldots\wedge P_m:=\sum_{i=1}^{m+1}\det(e_i|P_1|\ldots|P_m)e_i
 $$
\detail{

\bigskip

 {\tt ACHTUNG: IST DIESE DEFINITION DES DACHPRODUKTS PASSEND, ODER
 MUSS DER $e_i$-EINTRAG IN DIE ERSTE SPALTE?}{\tt  $e_i$ Eintraege bitte nach vorn, vgl. [Gerd Fischer Lineare Algebra, 2010, p. 285] und in \cite[p. 558]{shen98}, wo gerade in der Definition von $X\rfloor dV_F$ $X$ in das erste Argument von $dV_F$ eingesetzt wird (Konsistenz).}

\bigskip

}
 for the standard basis $\{e_i\}_{i=1}^{m+1}$ of $\R^{m+1}$.
The function $\cA_\xi^F:W_\xi\times \R^{m+1}\to [0,\infty)$
 is defined  by 
\begin{equation}\label{AF}
\cA_\xi^F(x,Z):=
\begin{cases}\frac{|Z|\mathscr{H}^m(B_1^m(0))}{\mathscr{H}^m(\{T=(T^1,\ldots,T^{m+1})\in Z^\perp:
F(x,T^ib_i|_x)\le 1\})}, & \quad\textnormal{for $x\in W_\xi$ and $Z\not= 0$}\\
0 & \quad\textnormal{for $x\in W_\xi$ and $Z=0,$}
\end{cases}
\end{equation}
where $Z^\perp:=\{T\in\R^{m+1}:T\cdot Z=0\}$ denotes the 
orthogonal complement of the $m$-dimensional subspace spanned by
$Z$,
and we have the homogeneity relation
\begin{equation}\label{hom}
\cA_\xi^F(x,tZ)=t\cA_\xi^F(x,Z)\quad\textnormal{for all $x\in W_\xi,$ \,\,$t>0$, \,\,
$Z\in\R^{m+1}$.}
\end{equation}
\end{proposition}
Using the (globally defined) standard basis $\{e_1,\ldots,e_{m+1}\}$ of $\R^{m+1}$ we 
immediately deduce the 

\noindent
{\sc Proof of Theorem \ref{thm:1.1}.}\,
If $\mathscr{N}=\R^{m+1}$ and $F=F(x,y)$ is a Finsler metric on $\R^{m+1}$, and
$X\in C^1(\mathscr{M},\R^{m+1})$ is an immersion from a smooth  $m$-dimensional
manifold $\mathscr{M}$ 
into $\R^{m+1}$, then we obtain for the Finsler area of an open subset $\Omega\subset\mathscr{M}$
formula \eqref{area_space} in 
local coordinates $(u^1,\ldots,u^m):\Omega\to\tilde{\Omega}\subset\R^m,$
with the explicit expression \eqref{AF_space} for the integrand $\cA^F$ as stated in
Theorem \ref{thm:1.1}. 
\qed

\noindent
{\sc Proof of Proposition \ref{prop:AF}.}\,
It suffices to consider matrices $P=(P_1|\ldots|P_m)\in \R^{(m+1)\times m}$ of 
full rank $m$. Then the linear mapping $\ell:\R^m\to\R^{m+1}$ given by
$$
\ell(v):=v^\delta P_\delta\quad\textnormal{for $v=(v^1,\ldots,v^m)\in\R^m,$}
$$
has rank $m$, i.e., $\ell$ is injective.

For given $x\in W_\xi$ we set
\begin{eqnarray*}
V_x&:=&\{v\in\R^m:F(x,v^\delta P^i_\delta b_i|_x)\le 1\}\qquad\textnormal{and}\\
\Upsilon_x & :=& \{T\in\R^{m+1}:F(x,T^ib_i|_x)\le 1\,\,\&\,\, (P_1\wedge\ldots\wedge P_m)\cdot T=0\},
\end{eqnarray*}
and claim that $\ell(V_x)=\Upsilon_x$.

Indeed, for $T\in \ell(V_x)$ we find $v=(v^1,\ldots,v^m)\in\R^m$ such that $T=v^\delta P_\delta$ and 
$$
1\ge F(x,v^\delta P^i_\delta b_i|_x)=F(x,T^i b_i|_x),
$$
so that 
$$
(P_1\wedge\ldots\wedge P_m)\cdot T=v^\delta (P_1\wedge\ldots\wedge P_m)\cdot P_\delta =0,
$$
i.e., $T\in\Upsilon_x$. On the other hand, for $T\in \Upsilon_x$ we find
$$
0=(P_1\wedge\ldots\wedge P_m)\cdot T=\det(T|P_1|\ldots|P_m)
$$
so that $T$ is a linear combination of the $P_\delta$, $\delta=1,\ldots,m,$ (since
$P=(P_1|\ldots|P_m)$ was assumed to have full rank $m$), i.e., there are $v^\delta\in\R$,
$\delta=1,\ldots,m,$ such that $T=v^\delta P_\delta.$ Hence
$$
F(x,v^\delta P^i_\delta b_i|_x)=F(x,T^ib_i|_x)\le 1,
$$
since we assumed $T\in \Upsilon_x$. This implies $v\in V_x$ and therefore $T=\ell(v)\in\ell(V_x).$

Next we claim that for arbitrary $x\in W_\xi$
\begin{equation}\label{counting}
\mathscr{H}^0(V_x\cap\ell^{-1}(T))=\chi_{\Upsilon_x}(T)\quad\textnormal{for all $T\in\R^{m+1}$,}
\end{equation}
where $\chi_A$ denotes the characteristic function of a set $A\subset\R^{m+1}.$

To prove this we observe that for $T\not\in\span\{P_1,\ldots,P_m\}$ we know that 
$T\not\in\Upsilon_x$, since
$$
(P_1\wedge \ldots \wedge P_m)\cdot T=\det(T|P_1|\cdots|P_m)\not=0.
$$
This implies  \eqref{counting} for such $T$. 
For $T\in\span\{P_1,\ldots,P_m\}$ , i.e., $T=v^\delta P_\delta$ for some 
$v=(v^1,\ldots,v^m)\in\R^m$, we distinguish two cases: If $T\in\Upsilon_x$, in addition,
we find
$$
F(x,v^\delta P^i_\delta b_i|_x)=F(x,T^ib_i|_x)\le 1,
$$
hence $\mathscr{H}^0(V_x\cap\ell^{-1}(T))\ge 1.$ We even have 
$$
\mathscr{H}^0(V_x\cap\ell^{-1}(T))=1=\chi_{\Upsilon_x}(T),
$$
since $\ell$ is injective. If, finally, $T\not\in\Upsilon_x$ but still in the span
of the $P_\delta$, $\delta=1,\ldots,m,$ we know by $\ell(V_x)=\Upsilon_x$
that $V_x\cap\ell^{-1}(T)=\emptyset$ and therefore the identity \eqref{counting}
holds also in this case.

We apply now the area formula (see. e.g. \cite[Theorem 3.2.3]{federer}) to the linear
mapping $\ell$ and use \eqref{counting} to deduce for arbitrary $x\in W_\xi$
\begin{eqnarray*}
\mathscr{H}^m(V_x) & = & \int_{v\in V_x}\,d\mathscr{L}^m(v)
 = 
\frac{1}{|P_1\wedge\ldots\wedge P_m|}\int_{v\in V_x}|P_1\wedge\ldots\wedge P_m|\,d\mathscr{L}^m(v)\\
& = &
\frac{1}{|P_1\wedge\ldots\wedge P_m|}\int_{v\in V_x}\Big|
\frac{\partial\ell}{\partial v^1}\wedge\ldots\wedge
\frac{\partial\ell}{\partial v^m}\Big|\,d\mathscr{L}^m(v)\\
& = &
\frac{1}{|P_1\wedge\ldots\wedge P_m|}\int_{T\in\R^{m+1}}\mathscr{H}^0(V_x\cap
\ell^{-1}(T))\,d\mathscr{H}^m(T)\\
& \overset{\eqref{counting}}{=} &
\frac{1}{|P_1\wedge\ldots\wedge P_m|}\int_{T\in\R^{m+1}}\chi_{\Upsilon_x}(T)\,d\mathscr{H}^m(T)
 =  \frac{1}{|P_1\wedge\ldots\wedge P_m|}\mathscr{H}^m(\Upsilon_x),
\end{eqnarray*}
which proves the proposition, since the homogeneity relation \eqref{hom} follows immediately
from \eqref{AF}.
\qed

Notice that the expression $\cA^F$ in \eqref{AF_space} is well-defined on
$\R^{m+1}\times (\R^{m+1}\setminus\{0\})$ and can be continuously extended by zero
to all of $\R^{m+1}\times \R^{m+1}$ by virtue of the homogeneity relation
\eqref{H}, as long as $F$ is {\it any} continuous positively $1$-homogeneous
function on $\R^{m+1}\times\R^{m+1}.$ The following alternative representation
of $\cA^F$ for such an $F$ will turn out to be quite useful to transfer pointwise bounds from $F$
to $\cA^F$ (see Corollary \ref{cor:pointwise})
and to quantify the convexity of $A_F$ in the second variable; see Section \ref{sec:3.2}.
\begin{lemma}\label{lem:polar}
Let $F\in C^0(\R^{m+1}\times \R^{m+1})$ satisfy $F(x,y)>0$ for $y\not= 0$, and
\begin{equation}\label{H-Finsler}
F(x,ty)=tF(x,y)\Foa t>0,\, (x,y)\in\R^{m+1}\times\R^{m+1}.
\end{equation}
Then the expression $\cA^F$ defined in \eqref{AF_space} can be rewritten as
\begin{equation}\label{polar_rep}
\cA^F(x,Z)=\frac{|Z|\mathscr{H}^{m-1}(\mathbb{S}^{m-1}(0))}{\sqrt{\det(f_\delta\cdot f_\sigma)}\int_{\S^{m-1}}
\frac{1}{F^m(x,\theta^\kappa f_\kappa)}\,d\mathscr{H}^{m-1}(\theta)}\Foa (x,Z)\in\R^{m+1}\times
(\R^{m+1}\setminus\{0\}).
\end{equation}
for any choice of basis $\{f_1,\ldots,f_m\}$ of the $m$-dimensional subspace
$Z^\perp\subset\R^{m+1}.$
\end{lemma}
\detail{

\bigskip

Alternative form:
\begin{equation}\label{polar_rep_alternative}
\cA^F(x,Z)=\frac{|Z|\mathscr{H}^m(B_1^m(0))}{\sqrt{\det(f_\delta\cdot f_\sigma)}\int_{\S^{m-1}}
\frac{1}{mF^m(x,\theta^\kappa f_\kappa)}\,d\mathscr{H}^{m-1}(\theta)}\Foa (x,Z)\in\R^{m+1}\times
(\R^{m+1}\setminus\{0\})
\end{equation}

\bigskip

}

\proof
Applying the area formula \cite[Theorem 3.2.3]{federer} to the linear mapping
$g:\R^m\to\R^{m+1}$ given by $g(t):=t^\kappa f_\kappa$ for $t=(t^1,\ldots,t^m)\in\R^m$
with Jacobian determinant
$\sqrt{\det\big[Dg(t)^{\textnormal{T}}Dg(t)\big]}=
\sqrt{\det(f_\delta\cdot f_\sigma)}$ independent of $t$, one calculates for
the denominator of $\cA^F$ in \eqref{AF_space}
\begin{eqnarray*}
\mathscr{H}^m(\{T\in Z^\perp:F(x,T)\le 1\}) & = &
\int_{\R^{m+1}}\chi_{\{T\in\R^{m+1}:F(x,T)\le 1, \,T\cdot Z=0\}}(z)\,d\mathscr{H}^m(z)\\
& &\hspace{-3cm}=\sqrt{\det(f_\delta\cdot f_\sigma)}\int_{\R^m}\chi_{\{t\in\R^m:F(x,t^\kappa f_\kappa)\le 1\}}
(\zeta)\,d\mathscr{L}^m(\zeta)\\
& &\hspace{-3cm}=
\sqrt{\det(f_\delta\cdot f_\sigma)}\int_{\S^{m-1}}\int_0^\infty \chi_{\{r\theta\in\R^m:
F(x,(r\theta)^\kappa f_\kappa)\le 1\}}(s\theta)s^{m-1}\,dsd\mathscr{H}^{m-1}(\theta)\\
& &\hspace{-3cm}=
\sqrt{\det(f_\delta\cdot f_\sigma)}\int_{\S^{m-1}}\int_0^{1/F(x,\theta^\kappa f_\kappa)}s^{m-1}\,ds
d\mathscr{H}^{m-1}(\theta)\\
& &\hspace{-3cm}=
\sqrt{\det(f_\delta\cdot f_\sigma)}\int_{\S^{m-1}}\frac{1}{mF^m(x,\theta^\kappa f_\kappa)}
\,d\mathscr{H}^{m-1}(\theta),
\end{eqnarray*}
where we have transformed to polar coordinates $\zeta=s\theta$ for $\theta=\zeta/|\zeta|\in
\S^{m-1}$ with $d\mathscr{L}^m(\zeta)=s^{m-1}d\mathscr{H}^{m-1}(\theta)$, and we used
\eqref{H-Finsler} to write $rF(x,\theta^\kappa f_\kappa)=F(x,(r\theta)^\kappa f_\kappa)\le 1$ in the 
defining set of the characteristic function $\chi$, and the identity $\mathscr{H}^m(\S^{m-1})=m\mathscr{H}^m(B_1^m(0)).$
\qed

\medskip

The $m$-symmetrization $F_\sym$ of a Finsler metric $F$ leads to the same expression 
$\cA^F$ as can be seen in the following lemma.
\begin{lemma}\label{lem:AF=AFsym}
Let $F=F(x,y)\in C^0(\R^{m+1}\times \R^{m+1})$ be  strictly positive as long as $y\not= 0$ and assume
that \eqref{H-Finsler} holds true. Then
\begin{equation}\label{AF=AFsym}
\cA^F(x,Z)=A^{F_\sym}(x,Z)\quad\Foa (x,Z)\in\R^{m+1}\times\R^{m+1}.
\end{equation}
\end{lemma}
\proof
By inspection of the definition \eqref{m-harmonic} 
of $F_\sym$ one observes that  the homogeneity condition \eqref{H-Finsler} on $F$
implies that also $F_\sym$ is positively $1$-homogeneous in $y$ and thus extendible
by zero to all of $\R^{m+1}\times\R^{m+1}$, so that also $\cA^{F_\sym}$ is
well-defined (replacing $F$ by $F_\sym$ in \eqref{AF_space}) and positively $1$-homogeneous
on $\R^{m+1}\times(\R^{m+1}\setminus\{0\})$. Hence $\cA^{F_\sym}$ is also
extendible by zero onto
$\R^{m+1}\times\R^{m+1}$. Thus it suffices to prove \eqref{AF=AFsym} for $Z\not= 0$,
so that $Z^\perp\subset\R^{m+1}$ is an $m$-dimensional subspace of $\R^{m+1}$.
If $\{f_\delta\}_{\delta=1}^m$ is a basis of $Z^\perp$ then so is $\{(-f_\delta)\}_{\delta=1}^m$.
Moreover $f_\delta\cdot f_\sigma=(-f_\delta)\cdot (-f_\sigma)$ for all $\delta,
\sigma=1,\ldots,m,$ so that $\sqrt{\det(f_\delta\cdot f_\sigma)}=
\sqrt{\det((-f_\delta)\cdot (-f_\sigma))}$; hence we can use Lemma \ref{lem:polar} twice
to 
compute
\begin{eqnarray*}
\cA^F(x,Z) & = & \frac{|Z|\mathscr{H}^m(\S^{m-1})}{\sqrt{\det(f_\delta\cdot f_\sigma)}
\int_{\S^{m-1}}\frac{1}{ F^m(\theta^\kappa f_\kappa)}\,d\mathscr{H}^{m-1}(\theta)}\\
& = &
\frac{|Z|\mathscr{H}^m(\S^{m-1})}{\sqrt{\det(f_\delta\cdot f_\sigma)}
\int_{\S^{m-1}}\frac{d\mathscr{H}^{m-1}(\theta)}{2 F^m(\theta^\kappa f_\kappa)}
+
\sqrt{\det((-f_\delta)\cdot (-f_\sigma))}\int_{\S^{m-1}}\frac{d\mathscr{H}^{m-1}(\theta)}{2 F^m(-\theta^\kappa f_\kappa)}}\\
& = &
\frac{|Z|\mathscr{H}^m(\S^{m-1})}{\sqrt{\det(f_\delta\cdot f_\sigma)}
\int_{\S^{m-1}}\Big[\frac{1}{2 F^m(\theta^\kappa f_\kappa)}+
\frac{1}{2 F^m(-\theta^\kappa f_\kappa)}\Big]\,d\mathscr{H}^{m-1}(\theta)}\\
& = &
\frac{|Z|\mathscr{H}^m(\S^{m-1})}{\sqrt{\det(f_\delta\cdot f_\sigma)}
\int_{\S^{m-1}}\Big[\frac{2^{1/m}}{(\frac{1}{F^m(\theta^\kappa f_\kappa)}+
\frac{1}{F^m(-\theta^\kappa f_\kappa)})^{1/m}}\Big]^{-m}\,d\mathscr{H}^{m-1}(\theta)}\\
& = &
\frac{|Z|\mathscr{H}^m(\S^{m-1})}{\sqrt{\det(f_\delta\cdot f_\sigma)}
\int_{\S^{m-1}}\frac{1}{ F^m_\sym(\theta^\kappa f_\kappa)}\,d\mathscr{H}^{m-1}(\theta)}
 \quad =\quad  \cA^{F_\sym}(x,Z).
\end{eqnarray*}
\qed

\subsection{Solving the Plateau problem}\label{sec:2.2}
The existence of minimizing solutions for two-dimensional
geometric boundary value problems 
including the Plateau problem has been established for general
{\it Cartan functionals} in \cite{HilvdM-calcvar,HilvdM-courant,HilvdM-partially}. 
These functionals are double integrals of the form
\begin{equation}\label{cartan-functionals}
\int\int_B\cC(X(u),(X_{u^1}\wedge X_{u^2})(u))\,du
\end{equation}
defined  on mappings $X:B\subset\R^2\to\R^n$ for $n\ge 2$, where the
{\it Cartan integrand} (or {\it parametric integrand})
$\cC\in C^0(\R^n\times\R^N)$ is characterized by the homogeneity
condition
\begin{equation}\label{homo-cartan}
\tag{H}
\cC(x,tZ)=t\cC(x,Z)\quad\Foa (x,Z)\in\R^n\times\R^N,\, t>0.
\end{equation}
(Here, $N=n(n-1)/2$ denotes the dimension of the space of bivectors
$\xi\wedge \eta$ for $\xi,\eta\in\R^n.$)

For the existence theory one requires, in addition, the existence of
constants $0<m_1<m_2$ such that
\begin{equation}\label{D}
\tag{D}
m_1|Z|\le\cC(x,Z)\le m_2|Z|\quad\Foa (x,Z)\in\R^n\times\R^N,
\end{equation}
and
\begin{equation}\label{C}
\tag{C}
Z\mapsto \cC(x,Z)\quad\textnormal{is convex for all $x\in\R^n.$}
\end{equation}
We have already observed in the introduction
that \eqref{homo-cartan} holds for the Finsler-area
integrand $\cA^F$, so it suffices to
prove \eqref{D} and \eqref{C} for $\cA^F$.
\begin{lemma}\label{lem:pointwise}
Let $F_1,F_2\in C^0(\R^{m+1}\times\R^{m+1})$ be strictly positive on
$\R^{m+1}\times(\R^{m+1}\setminus\{0\})$, each satisfying the homogeneity relation
\eqref{hom}. If for $x\in\R^{m+1}$ there exist numbers $0<c_1(x)\le c_2(x)$
with
$$
c_1(x)F_1(x,y)\le F_2(x,y)\le c_2(x)F_1(x,y)\quad\Foa y\in\R^{m+1},
$$
then 
\begin{equation}\label{pointwise}
m_1(x)\cA^{F_1}(x,Z)\le \cA^{F_2}(x,Z)\le m_2(x)\cA^{F_1}(x,Z)\quad\Foa Z\in\R^{m+1},
\end{equation}
where $m_1(x):=c_1^m(x)$ and 
$m_2(x):=c_2^m(x)$. 
\end{lemma}
\proof
The statement is obvious for $Z=0$ since then all terms in \eqref{pointwise} vanish.
For $Z\not= 0$ we choose a basis of the subspace $Z^\perp\subset\R^{m+1}$ and use
the representation \eqref{polar_rep} of Lemma \ref{lem:polar} to compute
\begin{multline*}
m_1(x)\cA^{F_1}(x,Z)\overset{\eqref{polar_rep}}{=}
\frac{m_1(x)|Z|\mathscr{H}^m(\S^{m-1})}{\sqrt{\det(f_\delta\cdot f_\sigma)}
\int_{\S^{m-1}}\frac{1}{ F_1^m(x,\theta^\kappa f_\kappa)}\,d\mathscr{H}^{m-1}(\theta)}\\
= \frac{|Z|\mathscr{H}^m(\S^{m-1})}{\sqrt{\det(f_\delta\cdot f_\sigma)}
\int_{\S^{m-1}}\frac{d\mathscr{H}^{m-1}(\theta)}{ (c_1(x)F_1(x,\theta^\kappa f_\kappa))^m}}\le
\frac{|Z|\mathscr{H}^m(\S^{m-1})}{\sqrt{\det(f_\delta\cdot f_\sigma)}
\int_{\S^{m-1}}\frac{d\mathscr{H}^{m-1}(\theta)}{ F_2^m(x,\theta^\kappa f_\kappa)}}
\overset{\eqref{polar_rep}}{=}\cA^{F_2}(x,Z).
\end{multline*}
The second inequality in \eqref{pointwise} can be established in the same way.
\qed
\begin{corollary}\label{cor:pointwise}
Let $F$ be a Finsler metric on $\R^{m+1}$ with
\begin{equation}\label{D*}
0<c_1:=\inf_{\R^{m+1}\times\S^m}F(\cdot,\cdot)\le
\sup_{\R^{m+1}\times\S^m}F(\cdot,\cdot)=\|F\|_{L^\infty(\R^{m+1}\times \S^m)}=:c_2<\infty.
\end{equation}
Then
\begin{equation}\label{AFdefinite}
m_1|Z|\le \cA^F(x,Z)\le m_2|Z|\quad\Foa (x,Z)\in\R^{m+1}\times\R^{m+1},
\end{equation}
where $m_1:= c_1^m$ and $m_2:=c_2^m.$
\end{corollary}
Notice that the defining properties (F1), (F2) of any Finsler metric
$F=F(x,y)$ imply that $F(x,y)>0$ whenever $y\not=0$ so that Assumption \eqref{D*} 
is automatically satisfied if $F=F(y)$ is
a Minkowski metric since then 
$$
0<\min_{\S^m}F(\cdot)\le F(y)\le \max_{\S^m}F(\cdot)<\infty
$$ 
by continuity
of $F$.

\noindent
{\sc Proof of Corollary \ref{cor:pointwise}.}\,
The homogeneity condition (F1) on $F$ implies
$$
c_1|y|\le F(x,y/|y|)|y|\overset{{\rm (F1)}}{=}
F(x,y)\le c_2|y|,
$$
so that we can apply Lemma \ref{lem:pointwise} to the functions
$F_1(x,y):=|y|$ and $F_2(x,y):=F(x,y)$ and constants $c_i(x):=c_i$ for
$i=1,2,$ to obtain \eqref{AFdefinite} from \eqref{pointwise}.
\qed

One easily checks that $F$ and its $m$-harmonic symmetrization
$F_\sym$ have the same pointwise bounds; hence it does not make
any difference whether one assumes \eqref{D*} for $F$ or for $F_\sym$.
\detail{

\bigskip

$0<c_1\le F(x,y)\le c_2$ for all $(x,y)\in\R^{m+1}\times\S^m$ implies
$$
\frac{1}{c_2^m}\le\frac{1}{F^m(x,y)}\le\frac{1}{c_1^m}\quad\Rightarrow\quad
\frac{2}{c_2^m}\le\frac{1}{F^m(x,y)}+\frac{1}{F^m(x,-y)}\le\frac{2}{c_1^m}
$$
and therefore
$$
\frac{2^{(1/m)}}{c_2}\le\left[ \frac{1}{F^m(x,y)}+\frac{1}{F^m(x,-y)} \right]^{1/m}\le
\frac{2^{(1/m)}}{c_1}
$$
thus
$$
c_1\le \frac{2^{1/m}}{\left[ \frac{1}{F^m(x,y)}+\frac{1}{F^m(x,-y)} \right]^{1/m}}=F_\sym\le c_2.
$$

On the other hand,
$c_1\le F_\sym\le c_2$ implies (with the above calculation reversed)
$$
\frac{2}{c_2^m}\le\frac{1}{F^m(x,y)}+\frac{1}{F^m(x,-y)}\le\frac{2}{c_1^m}
\quad\Foa (x,y)\in\R^{m+1}\times\S^m,
$$
and therefore
\begin{eqnarray*}
\frac{2}{c_2^m}& \le &\inf_{\R^{m+1}\times\S^m}\frac{2}{F^m(x,y)}\le
\inf_{\R^{m+1}\times\S^m}\left(\frac{1}{F^m(x,y)}+\frac{1}{F^m(x,-y)}\right)\\
& \le & \sup_{\R^{m+1}\times\S^m}\left(\frac{1}{F^m(x,y)}+\frac{1}{F^m(x,-y)}\right)
\le \sup_{\R^{m+1}\times\S^m}\frac{2}{F^m(x,y)}\le \frac{2}{c_1^m},
\end{eqnarray*}
which implies
$$
\frac{2}{c_2^m}\le \frac{2}{F^m(x,y)}\le \frac{2}{c_1^m} \quad\Rightarrow\quad
c_1\le F(x,y) \le c_2.
$$

\bigskip

}
The following convexity result was first established by H. Busemann \cite[Theorem II, p. 28]{busemann-convex-brunn-minkowski-1949}
and can also be found in the treatise of Thompson \cite[Theorem 7.1.1]{thompson-1996}.
\begin{theorem}[Busemann]\label{thm:convexity}
If $F$ is a reversible Finsler metric on $\mathscr{N}=\R^{m+1}$ then the 
corresponding expression $\cA^F=\cA^F(x,Z)$ defined in \eqref{AF_space} is convex
in $Z$ for any $x\in\R^{m+1}.$
\end{theorem}
\detail{

\bigskip

{\tt  $F>0$, $1$-homogen und konvex in $y$ wuerde auch reichen}
{\tt deswegen die neue Bemerkung 4 in der Einleitung}

\bigskip

}
We omit Busemann's proof and refer to the literature, but let us point out
that his proof is of geometric nature, and we do not see how to quantify the
convexity directly from his arguments. Nevertheless, Theorem \ref{thm:convexity}
does serve us as a starting point of our quantitative analysis of 
convexity properties via
the spherical Radon transform and its inverse in our treatment of
Finsler-minimal immersion in \cite{overath-vdM-2012b}. At the present stage, however,
Busemann's theorem suffices to solve Plateau's problem in Finsler spaces.

\medskip

\noindent
{\sc Proof of Theorem \ref{thm:plateau}.}\,
Specifying Theorem  \ref{thm:1.1} to the dimension $m=2$, and
taking the closure $\bar{B}$ of the unit disk $B\equiv B_1(0)\subset\R^2$
as the base manifold $\mathscr{M}$ immersed into $\R^3$ via a mapping $Y\in
C^1(\bar{B},\R^3)$ we find  for its Finsler area according to \eqref{area_space}
the expression
\begin{equation}\label{finslerarea}
\area_B^F(Y)=\int_B\cA^F(Y(u),\big(\frac{\partial Y}{\partial u^1}
\wedge\frac{\partial Y}{\partial u^2}\big)(u))\,du^1du^2,
\end{equation}
with an integrand $\cA^F\in C^0(\R^3\times \R^3)$ satisfying \eqref{homo-cartan}
(see \eqref{hom} in Proposition \ref{prop:AF}).
Consequently, the Finsler area \eqref{finslerarea}
can be identified with a {\it Cartan functional} with a 
Cartan integrand $\cA^F$. The Plateau problem
described in the introduction now asks for (a priori possibly branched) 
minimizers  of that functional in the class $\cC(\Gamma),$ where $\Gamma$
is the prescribed rectifiable Jordan curve in $\R^3.$

The growth assumption \eqref{growth} leads by virtue of Corollary 
\ref{cor:pointwise} to
\begin{equation}\label{DD}
\tag{D}
m_1|Z|\le \cA^F(x,Z)\le m_2|Z|\quad\Foa (x,Z)\in\R^3\times\R^3,
\end{equation}
where $0<m_1:=m_F^2\le M_F^2=:m_2<\infty,$
so $\cA^F$  is a {\it definite} Cartan integrand if we speak in the
terminology of \cite[Section 2]{HilvdM-parma}. Moreover,
Theorem \ref{thm:convexity} yields
\begin{equation}\label{CC}
\tag{C}
\textnormal{$Z\mapsto \cA^F(x,Z)$ is convex for any $x\in\R^3,$}
\end{equation}
if $F$ is reversible, but we did not assume that in Theorem \ref{thm:plateau}.
However, the $m$-harmonic symmetrization $F_\sym$ is reversible so that
\eqref{C} is valid for $\cA^{F_\sym}$ if $F_\sym$ itself is a Finsler metric.
But this is exactly our general assumption (GA). Now, with Lemma \ref{lem:AF=AFsym}
we obtain also \eqref{C} for $\cA^F$, so that we can apply
\cite[Theorem 1.4 \& 1.5]{HilvdM-courant} (see also \cite[Theorem 1]{HilvdM-calcvar})
to deduce the existence of a Finsler-area minimizer $X\in\mathcal{C}(\Gamma)$
satisfying the conformality relations \eqref{conf} and the additional
regularity properties stated in Theorem \ref{thm:plateau}.
\qed

\noindent{\sc Proof of Corollary \ref{cor:isop}.}\,
The isoperimetric inequality can be established in a similar way as in the proof of 
Theorem 3 in \cite[p.628]{clarenz-vdM-isop}.
\detail{

\bigskip

{\tt Der Beweis scheint mir schluessig. Nur bin ich mir nicht sicher bei den Saetzen. In \cite[Theorem 1, p. 330]{DHS1} wird die Flaeche als $C^2$ mit beschraenkter $W^{1,2}$-Energie vorausgesetzt. Wie garantieren wir die Existenz einer solchen Vergleichsfl\"ache? Satz?}

{\tt Da in dem Korollar schon $\cC(\Gamma)\not=\emptyset$ vorausgesetzt
ist, weil ja sonst gar kein Minimierer da w\"are, wei{\ss} man nach Minimalfl\"achentheorie, dass man das Dirichlet integral in $\cC(\Gamma)$ minimieren
kann, siehe \cite[Chapter 4]{DHKW1}. Diese Minimierer haben harmonische 
Koordinatenfunktionen und sind deshalb im Innern sogar reell analytisch, also
insbesondere $C^\infty$, und mit diesen Minimalfl\"achen vergleicht man
unsere Finsler-Minimierer...}

\bigskip

}
Let $Y$ be a disk-type minimal surface bounded by
the curve $\Gamma$. Then the  isoperimetric inequality for classic minimal
surfaces \cite[Theorem 1, p. 330]{DHS1} and the growth
condition $m_F |y| \le F(x,y) \le M_F|y|$ imply that for the 
Finslerian minimizer $X$ we can conclude
\begin{equation}
\area^F_B(X) \le \area^F_B(Y ) \le m_2A(Y ) \le \frac{m_2}{4\pi}(\mathscr{L}(\Gamma) )^2 \le \frac{m_2}{4\pi m_1}(\mathscr{L}^F(\Gamma) )^2
\end{equation}
which proves the result.
\qed

\setnumbers
\section{Higher Regularity}\label{sec:3}
\subsection{Radon transform}\label{sec:3.1}

Extending the {\it spherical Radon transform} \cite{radon-1994}
-- for $m=2$ also known as {\it Funk transform}
 \cite{funk-1915} --
to positively
homogenous functions on $\R^{m+1}\setminus\{0\}$
we will formulate 
sufficient conditions
on the Finsler metric $F$
to guarantee the ellipticity of the corresponding
Cartan integrand $\cA^F$, and moreover, the existence of a corresponding
perfect dominance function for $\cA^F$ leading to higher regularity
of minimizers of the Plateau problem; see Section \ref{sec:3.2}.
%In addition, {\it any}  
%Minkowski metric $F=F(y)$ on $\R^{m+1}$ leads
%to an elliptic Cartan integrand $\cA^F$ 
% but in order to prove the latter we 
%need to take the inverse of the spherical Radon transform on the
%Fr\'echet-space $C^\infty(\R^{m+1}\setminus\{0\})$ into account;
%see Section \ref{sec:3.2}. \xx\xx

\begin{definition}\label{def:3.1}
The {\em spherical Radon transform} $\hat{\mathcal{R}}$
defined on the function space $C^0(\S^m)$ of continuous functions
on the unit sphere $\S^m\subset\R^{m+1}$ is given as
$$
\hR[f](\zeta):=\frac{1}{\mathscr{H}^{m-1}(\S^{m-1})}\int_{
\S^m\cap\zeta^\perp}f(\omega)\,d\mathscr{H}^{m-1}(\omega)
\quad\textnormal{for $f\in C^0(\S^m)$ and $\zeta\in\S^m$.}
$$
\end{definition}
This and more general transformations of that kind have been investigated
intensively  within geometric analysis, 
integral geometry, geometric tomography, and convex analysis
by S. Helgason \cite{helgason_radon, helgason_groups}, T.N. 
Bailey et al.
\cite{bailey}, and many others; see e.g. 
\cite[Appendix C]{gardner}, where some useful properties
of the spherical Radon  transform are listed and where explicit
references to the literature is given, in particular to the
book of H. Groemer \cite{groemer}.

It turns out that the Cartan integrand $\cA^F$ defined in \eqref{AF_space}
may be rewritten in terms of the Radon transform after extending
$\hR$ suitably to the space of positively $(-m)$-homogeneous functions
on $\R^{m+1}\setminus\{0\}$; see Corollary \ref{cor:cartan-radon}
below. We set
\begin{equation}\label{R}
\tag{R}
\cR[g](Z):=\frac{1}{|Z|}\hR\big[g|_{\S^m}\big]\Big(\frac{Z}{|Z|}\Big)
\quad\Fo
g\in C^0(\R^{m+1}\setminus\{0\}),\,\,\, Z\in\R^{m+1}\setminus\{0\},
\end{equation}
which by definition is a $(-1)$-homogeneous function on $\R^{m+1}\setminus\{0\}$, and we will prove the following useful representation
formula.
\begin{lemma}\label{lem:general-radon}
For $g\in C^0(\R^{m+1}\setminus\{0\})$ one has
the identity
\begin{equation}\label{general-radon}
\cR\big[g\big](Z)=\frac{1}{|Z|\mathscr{H}^{m-1}(\S^{m-1})}\int_{\S^{m-1}}
g(\theta^\kappa f_\kappa)\,d\mathscr{H}^{m-1}(\theta)\quad\Fo
Z\in\R^{m+1}\setminus\{0\},
\end{equation}
where $\{f_1,\ldots,f_m\}\subset\R^{m+1}$
is an arbitrary orthonormal basis of the
subspace $Z^\perp\subset\R^{m+1}.$
\end{lemma}
This explicit representation of the extended Radon transform $\cR\big[g\big]$ 
can be used to give a direct proof of its continuity and differentiability as long as one inserts continuous
or differentiable homogeneous functions $g$; see \cite{overath-phd_2012}. 
On the other hand, this fact can also be deduced
 from well-established facts on the spherical
Radon transform and more general transformations as in
 \cite[Proposition 2.2, p. 59]{helgason_radon}
(see also \cite[p. 118]{goodey-weil_1991} and the short summary of results
in \cite[Appendix C]{gardner}), 
\detail{

\bigskip

{\tt \yy \cite[Proposition 2.2, p.59]{helgason_radon} macht eine Stetigkeitsaussage fuer die Werte einer verwandte Art der Radontransformation (nicht die Sphärische) mit Verweis auf eine gruppentheoretische Beziehung. Das ist die einzige Aussage, die ich gefunden habe, welche am ehesten beleuchtet, wie Helgason Stetigkeit zeigt.\xx\xx
ALSO SOLLTEN WIR NOCH IN GROEMER ODER GARDNER-ANHANG NACHSEHEN,
UM ETWAS ZITIERBARES ZU FINDEN\xx\xx\yy Stetigkeit mittels ``Groemer-Zustaten'' siehe Details.\yy}

\bigskip

}
so we will omit the proof
here:
 \begin{corollary}\label{cor:Rcontinuous}
The extended Radon transformation $\cR$ is a bounded linear
map from $C^0(\R^{m+1}\setminus\{0\})$ to $C^0(\R^{m+1}\setminus\{0\})$, and if $g\in C^1(\R^{m+1}\setminus\{0\})$ then $\cR\big[g\big]$ is differentiable on $\R^{m+1}\setminus\{0\}.$
\end{corollary}

\noindent
{\sc Proof of Lemma \ref{lem:general-radon}.}
By means of local coordinate charts
$\S^{m-1}\subset\bigcup_{t=1}^MV_t\subset\R^m$ and respective
coordinates $y_t=(y_t^1,\ldots,y_t^{m-1}):V_t\to\Omega_t\subset
\R^{m-1}$  we define the disjoint sets
$A_t:=V_t-\bigcup_{s=1}^{t-1}V_s$ for $t=1,\ldots,M$ and
use the characteristic functions $\chi_{A_t}$ of the sets
$A_t$ to partition the integrand
$$
g(\theta^\kappa f_k)
=\sum_{t=1}^M\chi_{A_t}(\theta)g(\theta^\kappa f_k)=:\sum_{t=1}^Mg_t(\theta)
$$
to find (cf. \cite[p. 142]{baer-buch})
$$
\int_{\S^{m-1}}g(\theta^\kappa f_\kappa)\,d\mathscr{H}^{m-1}(\theta)=\sum_{t=1}^M\int_{\S^{m-1}}g_t(\theta)\,d\mathscr{H}^{m-1}(\theta).
$$
In each term on the right-hand side we apply the area formula
\cite[3.2.3 (2)]{federer} with respect to the (injective)
transformation
$T_t:\Omega_t\to\R^{m+1}$ given by $T_t(y_t):=\theta_t^i(y_t)f_i$
for $t=1\ldots,M$ to obtain
\begin{multline*}
\int_{\S^{m-1}}g_t(\theta)\,d\mathscr{H}^{m-1}(\theta)=
\int_{\Omega_t}g_t(\theta_t(y_t))\sqrt{\det(D\theta_t^T(y_t)D\theta_t(y_t))}\,dy_t^1\cdots dy_t^{m-1}\\
=\int_{\R^{m+1}}g_t(\theta_t(y_t))|_{y_t\in T_t^{-1}(\zeta)}\,d\mathscr{H}^{m-1}(\zeta) 
=\int_{\R^{m+1}}(\chi_{A_t}(\theta_t(y_t))g(
\theta_t^\kappa(y_t)f_\kappa))|_{y_t\in T_t^{-1}(\zeta)}\,
d\mathscr{H}^{m-1}(\zeta)\\
= \int_{\R^{m+1}}\chi_{\S^m\cap Z^\perp}(\zeta)
(\chi_{A_t}(\theta_t(y_t))g(
\theta_t^\kappa(y_t)f_\kappa))|_{y_t\in T_t^{-1}(\zeta)}\,
d\mathscr{H}^{m-1}(\zeta),
\end{multline*}
since $\theta_t^\kappa(y_t)f_\kappa=T_t(y_t)=\zeta$ for $y_t\in T^{-1}_t(\zeta),$ and because $T_t(y_t)\in Z^\perp$ and $|T_t(y_t)|=1$ by
definition of $T_t.$ (Recall that the system
$\{f_1,\ldots,f_m\}$ forms an
orthonormal basis of $Z^\perp.$)
\detail{

\bigskip

With $\partial_\alpha T_t^k(y_t)=\partial_\alpha\theta^i(y_t)f_i^k$
one computes for the Jacobian
\begin{eqnarray*}
(DT_t^{T}(y_t)DT_t(y_t))_{\alpha\beta} & = & \partial_\alpha T_t^k(y_t)\delta_{kl}\partial_\beta T_t^l(y_t) =
\partial_\alpha \theta^i_t(y_t)f_i^k\delta_{kl}\partial_\beta
\theta^j_t(y_t)f^l_j\\
& = &
\partial_\alpha\theta^i_t(y_t)\partial_\beta\theta^j_t(y_t)f_i\cdot
f_j = \partial_\alpha\theta^i_t(y_t)\partial_\beta\theta^j_t(y_t)\delta_{ij}=(D\theta^T_t(y_t)D\theta_t(y_t))_{\alpha\beta}.
\end{eqnarray*}

\bigskip

}
Now, for $y_t\in T_t^{-1}(\zeta)$ one has $T_t(y_t)=\theta^\kappa_t(y_t)f_\kappa =\zeta\in\R^{m+1}$, and therefore
$$
\theta_t(y_t)=(\theta_t^1(y_t),\ldots,\theta^m_t(y_t))=
(f_1\cdot\zeta,\ldots,f_m\cdot \zeta)=:\Phi^T\zeta
$$
for the matrix $\Phi:=(f_1|\cdots|f_m)\in\R^{(m+1)\times m}$
with the orthonormal basis vectors $f_i,$ $i=1,\ldots,m$, as
column vectors. This implies for any set $A\subset\R^m$ that
$\theta=(\theta^1,\ldots,\theta^m)=\Phi^T\zeta\in A$ if and
only if $\zeta\in \Phi A:=\{\xi\in\R^{m+1}:\xi=\Phi a \textnormal{\,for
some $a\in A$}\}$ since  $\Phi^T\Phi=\Id_{\R^{m+1}}.$
Hence the characteristic functions satisfy $\chi_A(\Phi^T\zeta)=
\chi_{\Phi A}(\zeta)$, in particular we find
$$
\chi_{A_t}(\theta_t(y_t))|_{y_t\in T_t^{-1}(\zeta)}=
\chi_{\Phi A_t}(\zeta),
$$
where the sets $\Phi A_t$ are also disjoint, since any
$\xi\in \Phi A_t\cap\Phi A_\sigma$ for $1\le t <\sigma\le M$
has the representations
$\xi=\Phi a_t=\Phi a_\sigma$ for some $a_t\in A_t$ and $a_\sigma\in
A_\sigma,$ which implies $\Phi a_t=
a^i_tf_i=a^i_\sigma f_i=\Phi a_\sigma$, i.e., $a_t=a_\sigma$ as the
$f_i$ are linearly independent. But then $a_t=a_\sigma\in A_t\cap
A_\sigma=\emptyset, $ which is a contradiction.
Summarizing we conclude
\begin{eqnarray*}
\int_{\S^{m-1}}g(\theta^\kappa f_\kappa)\,d\mathscr{H}^{m-1}(\theta)
&=& \int_{\S^{m-1}}\sum_{t=1}^M g_t(\theta)\,d\mathscr{H}^{m-1}(\theta)
 = \sum_{t=1}^M\int_{\Phi A_t}g(\zeta)\,d\mathscr{H}^{m-1}(\zeta)\\
& = &
\int_{\S^m\cap Z^\perp}g(\zeta)\,d\mathscr{H}^{m-1}(\zeta)=
\int_{\S^m\cap (Z/|Z|)^\perp}g(\zeta)\,d\mathscr{H}^{m-1}(\zeta),
\end{eqnarray*}
since we have the disjoint union $ \bigcup_{t=1}^M \Phi A_t = \Phi (\bigcup_{t=1}^M A_t )= \Phi( \S^{m-1}) = \S^{m}\cap Z^\perp $.
\detail{

\bigskip

 Das Einschieben von $\chi_{\S^m\cap Z^\perp}$ unter dem Integral ist ueberfluessig, denn
\begin{equation*}
 \bigcup_{t=1}^M \Phi A_t = \Phi (\bigcup_{t=1}^M A_t )=
 \Phi( \S^{m-1}) = \S^{m}\cap Z^\perp.
\end{equation*}
Es handelt sich dabei nach Konstuktion sogar um eine disjunkte Vereinigung. Also kann man schreiben:
\begin{eqnarray*}
\int_{\S^{m-1}}g(\theta^\kappa f_\kappa)\,d\mathscr{H}^{m-1}(\theta)
&=& \int_{\S^{m-1}}\sum_{t=1}^M g_t(\theta)\,d\mathscr{H}^{m-1}(\theta)\\
& = &
\sum_{t=1}^M\int_{\R^{m+1}}\chi_{
\Phi A_t}(\zeta)g(\zeta)\,d\mathscr{H}^{m-1}(\zeta)\\
& = &
\sum_{t=1}^M\int_{\Phi A_t}g(\zeta)\,d\mathscr{H}^{m-1}(\zeta)\\
& = &
\int_{\cup_{t=1}^M\Phi A_t}g(\zeta)\,d\mathscr{H}^{m-1}(\zeta)\\
& = &
\int_{\S^m\cap Z^\perp}g(\zeta)\,d\mathscr{H}^{m-1}(\zeta)=
\int_{\S^m\cap (Z/|Z|)^\perp}g(\zeta)\,d\mathscr{H}^{m-1}(\zeta),
\end{eqnarray*}

\bigskip

}
and
therefore, by definition \eqref{R} of the extended Radon transform,
$$
\frac{1}{|Z|\mathscr{H}^{m-1}(\S^{m-1})}\int_{\S^{m-1}}
g(\theta^\kappa f_\kappa)\,d\mathscr{H}^{m-1}(\theta)
=\cR\big[g\big](Z)
\quad\Fo
Z\in\R^{m+1}\setminus\{0\}.
$$
\qed

It turns out that the extended Radon transform enjoys a nice
transformation behaviour under the action of the special linear
group $SL(m+1)$ of $(m+1)\times (m+1)$-matrices with determinant
equal to $1$; see Corollary \ref{cor:invariance} below.
More generally, one has the following transformation rule:
\begin{lemma}\label{lem:invariance}
For any   $(-m)$-homogeneous function
$g\in C^0(\R^{m+1}\setminus\{0\})$
one has
\begin{equation}\label{GL-invariance}
\cR\big[g\big]\circ L=
\cR\big[g\circ (\det L)^{1/m}L^{-T}\big]
\end{equation}
for every invertible matrix $L\in\R^{(m+1)\times (m+1)}.$
\end{lemma}

\begin{corollary}\label{cor:invariance}
For all $ L\in SL(m+1)$ and  all $(-m)$-homogeneous  functions $g\in C^0(\R^{m+1}\setminus\{0\})$ one has
\begin{equation}\label{SL-invariance}
\cR\big[g\big]\circ L=
\cR\big[g\circ L^{-T}\big].
\end{equation}
\end{corollary}
\proof
Relation \eqref{SL-invariance} is an
immediate consequence of Lemma \ref{lem:invariance}
since $\det L=1$ for $L\in SL(m+1)$. 
\qed

{\sc Proof of Lemma \ref{lem:invariance}.}\,
By continuity of $\cR\big[\cdot\big]$ it suffices to prove the lemma
for $C^1$-functions.

For an orthonormal basis $\{f_1,\ldots,f_m\}\subset\R^{m+1}$ of an
$m$-dimensional subspace of $\R^{m+1}$ we can form
the
exterior product
$$
f_1\wedge f_2 \wedge \ldots \wedge f_m=
\sum_{i=1}^{m+1}\det(f_1|f_2|\ldots|f_m|e_i)e_i\in\R^{m+1},
$$
where the $e_i$ denote the standard basis vectors of $\R^{m+1}$,
$i=1,\ldots,m+1,$ and we have (see, e.g., \cite[Ch. 2.6, p.14]{flanders}
$$
|f_1\wedge f_2 \wedge \ldots \wedge f_m|^2=
(
f_1\wedge f_2 \wedge \ldots \wedge f_m)\cdot
(f_1\wedge f_2 \wedge \ldots \wedge f_m)=\det(f_i\cdot f_j)=1.
$$
Lemma \ref{lem:general-radon} applied to the $m$-vector
$Z:=f_1\wedge\ldots\wedge f_m$
(so that $\span\{f_1,\ldots,f_m\}=Z^\perp$)  yields
$$
\cR\big[g\big](f_1\wedge\ldots\wedge f_m)=
\frac{1}{\mathscr{H}^{m-1}(\S^{m-1})}\int_{\S^{m-1}}g(\theta^\kappa
f_\kappa)\,d\mathscr{H}^{m-1}(\theta)
$$
for any $g\in C^1(\R^{m+1}\setminus\{0\})$. By means of the
Gau{\ss} map $\nu:\S^{m-1}\to\R^{m}$, which coincides with the
position vector at every point on $\S^{m-1}$, i.e., $\nu(\theta)
=\theta $ for any $\theta=(\theta^1,\ldots,\theta^m)\in\S^{m-1}\subset
\R^m$ we can apply \cite[Satz 3, p. 245]{forster3} to rewrite the Radon
transform in terms of differential forms:
\begin{eqnarray}
\cR\big[g\big](f_1\wedge\ldots\wedge f_m) & = &
\frac{1}{\mathscr{H}^{m-1}(\S^{m-1})}\int_{\S^{m-1}}
g(\theta^\kappa f_\kappa)\theta^s\nu_s(\theta)
\,d\mathscr{H}^{m-1}(\theta)\notag\\
& = & 
\frac{1}{\mathscr{H}^{m-1}(\S^{m-1})}\int_{\S^{m-1}}
g(\theta^\kappa f_\kappa)\theta^s(-1)^{s-1}\,d\theta^1\wedge\ldots
\wedge\widehat{d\theta^s}\wedge\ldots\wedge d\theta^m\notag\\
& =: & \mathcal{I}\big[g\big](F),\label{R=I}
\end{eqnarray}
where $F=(f_1|f_2|\ldots|f_m)\in\R^{(m+1)\times m}$ assembles
the orthonormal basis vectors $f_1,\ldots,f_m$ as columns.

Now we claim that
\begin{equation}\label{B-invariance}
\mathcal{I}\big[g\big](\Xi B)=\frac{1}{\det B}\mathcal{I}\big[g\big](\Xi)
\end{equation}
for any $B=(b^i_j)\in\R^{m\times m} $ with positive determinant,
and $\Xi:=(\xi_1|\xi_2\ldots|\xi_m)\in\R^{(m+1)\times m}$,
where $\{\xi_1,\ldots,\xi_m\}$ is an arbitrary set of $m$
linearly independent vectors in $\R^{m+1}$ replacing the
$f_i,$ $i=1,\ldots,m,$ in the defining integral for $\mathcal{I}\big[
g\big](\cdot)$ in \eqref{R=I}.
Indeed, $B$ represents the linear map $\beta:\R^m\to\R^m$ with
$\beta^i(x)=b^i_jx^j$ for $x=(x^1,\ldots,x^m)\in\R^m$ with
inverse
$\beta^{-1}(y)=a^i_jy^j$ for $y=(y^1,\ldots,y^m)\in\R^m$,
where $A=(a^i_j):=B^{-1}\in\R^{m\times m},$ and we have
$
d\beta^i=b^i_j dx^j$ for $i=1,\ldots,m,$ so that
$d\theta^i=a^i_\tau b^\tau_jd\theta^j=a^i_\tau d\beta^\tau,$
and $\theta^s=a^s_\sigma b^\sigma_j\theta^j=b^\sigma_j\theta^j 
a^s_\sigma =\beta^\sigma(\theta)a^s_\sigma.$
By means of  the matrix
$$
\Xi B=(\xi_1|\ldots|\xi_m)B=(b_1^t\xi_t|\ldots|b_m^t\xi_t)
\in\R^{(m+1)\times m}
$$
we can write the left-hand side of \eqref{B-invariance}
as
\begin{eqnarray*}
\mathcal{I}\big[g\big](\Xi B )& = &
\frac{1}{\mathscr{H}^{m-1}(\S^{m-1})}\int_{\S^{m-1}}
g(\theta^\kappa b^t_\kappa \xi_t)\theta^s(-1)^{s-1}\,d\theta^1\wedge\ldots
\wedge\widehat{d\theta^s}\wedge\ldots\wedge d\theta^m\\
&  &\hspace{-2cm} =  
\frac{1}{\mathscr{H}^{m-1}(\S^{m-1})}\int_{\S^{m-1}}
g(\beta^t(\theta)\xi_t)\beta^\sigma(\theta)a^s_\sigma(-1)^{s-1}
a^1_{\tau_1}d\beta^{\tau_1}\wedge\ldots\wedge
\widehat{a^s_{\tau_s}d\beta^{\tau_s}}\wedge\ldots\wedge
a^m_{\tau_m}d\beta^{\tau_m}.
\end{eqnarray*}
Now, it is a routine matter in computations with determinants
to verify that the last integrand on the right-hand side
equals
$$
\frac{1}{\det B} g(\beta^t(\theta)\xi_t)\beta^s(\theta)(-1)^{s-1}
d\beta^1\wedge\ldots\wedge\widehat{d\beta^s}\wedge\ldots\wedge
d\beta^m,
$$
which is the pull-back $\beta^*\omega$ of the form
$$
\omega(\theta)=\frac{1}{\det B}g(\theta^t\xi_t)\theta^s(-1)^{s-1}\,d\theta^1
\wedge\ldots\wedge\widehat{d\theta^s}\wedge\ldots\wedge
d\theta^m
$$
under the linear mapping $\beta$. Since $\det B=\det D\beta>0$
by assumption  we obtain by the
transformation formula for differential forms (see, e.g.,
\cite[Satz 1, p. 235]{forster3})
\begin{eqnarray*}
\mathcal{I}\big[g\big](\Xi B) = \frac{1}{\mathscr{H}^{m-1}(\S^{m-1})}\int_{\S^{m-1}}
\beta^*\omega & = & \frac{1}{\mathscr{H}^{m-1}(\S^{m-1})}\int_{\beta(\S^{m-1})}\omega\\
&= &\frac{1}{\mathscr{H}^{m-1}(\S^{m-1})}\int_{\S^{m-1}}\omega
=\frac{1}{\det B}\mathcal{I}\big[g\big](\Xi),
\end{eqnarray*}
where we have used the fact that $\omega $ is a closed form and
that the closed surface $\beta(\S^{m-1})$ contains the origin
as the only singularity of the differential form $\omega$ in its
interior, since $\beta$ as a linear map maps $0$ to $0$;
see, e.g., \cite[Corollar, p. 257]{forster3}. (Recall
that $g$ was assumed to be $(-m)$-homogeneous and of class $C^1(\R^{m+1}\setminus\{0\}).$)
Hence the claim \eqref{B-invariance} is proved.

With arguments analogous to \cite[pp. 349, 350]{morrey66} (or
in more detail \cite[pp. 7--11]{vdM-diplomathesis}) one can use
relation \eqref{B-invariance} for fixed $g\in
C^1(\R^{m+1}\setminus\{0\})$ to show that
there is a $(-1)$-homogeneous
function $\mathcal{J}\big[g\big](\cdot):\R^{m+1}\to\R^{m+1}$ such
that
\begin{equation}\label{I=J}
\mathcal{I}\big[g\big](\Xi)=\mathcal{J}\big[g\big](\xi_1\wedge
\ldots\wedge \xi_m)
\Fo\Xi=(\xi_1|\ldots|\xi_m)\in\R^{(m+1)\times m},
\end{equation}
whenever $\xi_1,\ldots,\xi_m\in\R^{m+1}$ are linearly
independent.

For a hyperplane $(\xi_1\wedge\ldots\wedge\xi_m)^\perp\subset
\R^{m+1},$ where $\xi_1,\ldots,\xi_m\in\R^{m+1}$ are linearly
independent vectors, we can now choose an appropriately oriented
orthonormal basis
$\{f_1,\ldots,f_m\}\subset\R^{m+1}$, such that
$$
f_1\wedge\ldots\wedge f_m=\frac{\xi_1\wedge\ldots\wedge\xi_m}{
|\xi_1\wedge\ldots\wedge\xi_m|}.
$$
For the matrix $F=(f_1|\ldots|f_m)\in\R^{(m+1)\times m}$ we
consequently obtain by $(-1)$-homogeneity of $\cR\big[g\big](\cdot)$
and of $\mathcal{J}\big[g\big](\cdot)$
\begin{eqnarray}
\cR\big[g\big](\xi_1\wedge\ldots\wedge \xi_m) & = &
\cR\big[g\big](f_1\wedge\ldots\wedge f_m)\frac{1}{|\xi_1\wedge\ldots
\wedge \xi_m|}\notag\\
& \overset{\eqref{R=I}}{=} &
\mathcal{I}\big[g\big](F)\frac{1}{|\xi_1\wedge\ldots
\wedge \xi_m|}\notag\\
& \overset{\eqref{I=J}}{=} &
\mathcal{J}\big[g\big](f_1\wedge\ldots\wedge f_m)\frac{1}{|\xi_1\wedge\ldots
\wedge \xi_m|}\notag\\
& = &
\mathcal{J}\big[g\big](\xi_1\wedge\ldots\wedge \xi_m) \notag\\
& \overset{\eqref{I=J}}{=} &
\mathcal{I}\big[g\big](\Xi),\label{R=I_general}
\end{eqnarray}
which is relation \eqref{R=I} even for matrices
$\Xi=(\xi_1|\ldots|\xi_m)\in\R^{(m+1)\times m}$ whose
column vectors $\xi_i$, $i=1,\ldots,m,$ are merely
linearly independent.

According to the well-known formula
\begin{eqnarray*}
L(\xi_1\wedge\ldots\wedge\xi_m) & = & (\det L)(L^{-T}\xi_1)\wedge
(L^{-T}\xi_2)\wedge\ldots\wedge (L^{-T}\xi_m)\\
& = & ((\det L)^{\frac 1m}L^{-T}\xi_1)\wedge\ldots
((\det L)^{\frac 1m}L^{-T}\xi_m)
\end{eqnarray*}
\detail{

\bigskip

Proof of this well-known but hardly to be found FORMULA: 

\bigskip
 
  Express the linear mapping $L$ in the canonical basis $(e_i)_{i=1,\ldots,m+1}$ by
  \begin{eqnarray*}
  Le_i =& L_i^je_j
  \end{eqnarray*}
  for $i=1,\ldots,m+1$. Further, in a slight abuse of notation, the linear mapping $L$ will be identified with its transformation matrix $L=(L_i^j)_{i,j}$.
  By the componentwise definition of the wedge product and the Cauchy-Binet formula the components of the wedge-product can be written as
  \begin{eqnarray*}
  ((L\xi_1)\wedge \ldots \wedge (L\xi_m))^j =& (-1)^{j-1}\det (L\Xi)^{(j)} \\
  =& (-1)^{j-1}\det ((L^l_k\xi_\beta^k)^{(j)})\\
  =& (-1)^{j-1} \sum_{l=1}^{m+1}\det L^{(j)}_{(l)}\det \xi^{(l)} \\
  =& \det L\sum_{l=1}^{m+1} \frac{(-1)^{j+l} \det L^{(j)}_{(l)}}{\det L}  (-1)^{l-1}\det \xi^{(l)} \\
  =& ((\det L) L^{-T} \xi_1 \wedge \ldots \wedge \xi_m)^j,
  \end{eqnarray*}
where Cramer's rule was used in the last step.
  \bigskip

}
for any invertible matrix $L\in\R^{(m+1)\times (m+1)}$ we can
now conclude with \eqref{R=I_general} for matrices
$\Xi=(\xi_1|\ldots|\xi_m)\in\R^{(m+1)\times m}$ of maximal
rank $m$,
\begin{eqnarray*}
\cR\big[g\big](L(\xi_1\wedge\ldots\wedge\xi_m)) & = &
\cR\big[g\big](((\det L)^{\frac 1m}L^{-T}\xi_1)\wedge\ldots
((\det L)^{\frac 1m}L^{-T}\xi_m)) \\
& \overset{\eqref{R=I_general}}{=} &
\mathcal{I}\big[g\big]((\det L)^{1/m}L^{-T}\Xi)\\
&  &\hspace{-3cm} =
\frac{1}{\mathscr{H}^{m-1}(\S^{m-1})}\int_{\S^{m-1}}
g(\theta^\kappa(\det L)^{\frac 1m }L^{-T}\xi_\kappa)
\theta^s(-1)^{s-1}\,d\theta^1\wedge\ldots\wedge\widehat{d\theta^s}\wedge\ldots\wedge d\theta^m\\
&  &\hspace{-3cm} =
\cR\big[g\circ (\det L)^{1/m}L^{-T}\big](\xi_1\wedge\ldots\wedge
\xi_m),
\end{eqnarray*}
which proves the lemma, since for $Z\in\R^{m+1}\setminus\{0\}$
and any appropriately oriented
basis $\{\xi_1,\ldots,\xi_m\}\subset\R^{m+1}$ of the
subspace $Z^\perp\subset\R^{m+1}$, we have
$$
Z=|Z|\frac{\xi_1\wedge\ldots\wedge\xi_m}{|
\xi_1\wedge\ldots\wedge\xi_m|},
$$
and therefore by $(-1)$-homogeneity
\begin{eqnarray*}
\cR\big[g\big](LZ) & = &
\cR\big[g\big](L(\xi_1\wedge\ldots\wedge\xi_m))\frac{|
\xi_1\wedge\ldots\wedge\xi_m|}{|Z|}\\
& = & \cR\big[g\circ (\det L)^{1/m}L^{-T}\big](\xi_1\wedge\ldots\wedge
\xi_m)\frac{|
\xi_1\wedge\ldots\wedge\xi_m|}{|Z|}\\
& = & 
\cR\big[g\circ (\det L)^{1/m}L^{-T}\big]\Big(\frac{|Z|\xi_1\wedge\ldots\wedge
\xi_m}{|\xi_1\wedge\ldots\wedge
\xi_m|}\Big)=\cR\big[g\circ (\det L)^{1/m}L^{-T}\big](Z)
\end{eqnarray*}
for any $g\in C^1(\R^{m+1}\setminus\{0\})$ and therefore
also for any
$g\in C^0(\R^{m+1}\setminus\{0\})$ by approximation.
\qed

The transformation behaviour \eqref{SL-invariance} of $\cR$
 under the action of $SL(m+1)$ can be used to
prove valuable differentiation formulas for $\cR$ restricted
to a suitable class of homogeneous functions, since the tangent
space of $SL(m+1)$ seen as a smooth submanifold of
$\R^{(m+1)\times (m+1)}\cong\R^{(m+1)^2 }$ can be characterized
as the set of trace-free matrices; see, e.g.,  \cite[Lemma 8.15 \& Example 8.34]{lee-2003}.
\detail{

\bigskip

Write
$$
SL(m+1)=\{L\in \R^{(m+1)\times (m+1)}: \det L=1\},
$$
so $SL(m+1)$ is the preimage of the regular value zero
of the function $G(L):=\det L-1$. Indeed, by means of
Laplace's theorem on determinants
$$
\det L=\sum_{j=1}^{m+1}(-1)^{i+j}L^i_j\det L^{(i)}_{(j)}
\quad\Foa i=1,\ldots,m+1,
$$
where $L^{(i)}_{(j)}\in\R^{m\times m}$ denotes the matrix
that one obtains from $L$ by removing the $i$-th row and the
$j$-th column, so that
\begin{eqnarray*}
\frac{\partial}{\partial L^k_l}G(L)|_{L=\Id_{\R^{m+1}}}
& = & \frac{\partial}{\partial L^k_l}|_{L=\Id_{\R^{m+1}}}
\left\{\sum_{j=1}^{m+1}(-1)^{k+j}L^k_j\det L^{(k)}_{(j)}
\right\}\\
& = & \sum_{j=1}^{m+1}(-1)^{k+j}\delta_{jl}\det L^{(k)}_{(j)}|_{L=\Id_{\R^{m+1}}}\\
& = & (-1)^{k+l}\det L^{(k)}_{(l)}|_{L=\Id_{\R^{m+1}}}\\
& = & (-1)^{k+l}\delta^k_l,
\end{eqnarray*}
that is, the gradient $\nabla G\in
\R^{(m+1)\times (m+1)}\cong \R^{(m+1)^2}$ at $\Id_{\R^{m+1}}$
does not vanish. Moreover, it defines the tangent space
(\cite[Lemma 8.15]{lee-2003})
$$
T_{\Id_{\R^{m+1}}}SL(m+1)=\textnormal{\,Ker\,}(DG(\Id_{\R^{m+1}})=
(\nabla G(\Id_{\R^{m+1}}))^\perp
$$
of $SL(m+1)$ at the identity via the induced scalar product
$$
L\cdot M:=\sum_{i,j=1}^{m+1}L^i_jM^i_j,
$$
so that we obtain for all tangent vectors $V\in T_{\Id_{\R^{m+1}}}SL(m+1)$
\begin{eqnarray*}
0 & = & V\cdot\nabla G(\Id_{\R^{m+1}})\\
& = &
\sum_{i,j=1}^{m+1}V^i_j(-1)^{i+j}\delta^i_j=
\sum_{i=1}^{m+1}V^i_i(-1)^{i+i}=
\sum_{i=1}^{m+1}V^i_i.
\end{eqnarray*}
Consequently,
$$
T_{\Id_{\R^{m+1}}}SL(m+1)=\{V\in\R^{(m+1)\times (m+1)}:
\trace V=0\}.
$$

\bigskip

}
\begin{theorem}\label{thm:diffrule}
Let $k\in\N$ and $g\in C^k(\R^{m+1}\setminus\{0\})$ be 
positively $(-m)$-homogeneous. Then the Radon transform
$\cR\big[g\big]$ is of class $C^k(\R^{m+1}\setminus\{0\})$,
and one has 
or $Z=(Z^1,\ldots,Z^{m+1}),$
$y=(y^1,\ldots,y^{m+1})\in\R^{m+1}\setminus\{0\}$:
\begin{eqnarray}
&&\hspace{-2cm}Z_{\tau_1}\cdots Z_{\tau_k}\frac{\partial}{\partial Z_{\sigma_1}}
\cdots\frac{\partial}{\partial Z_{\sigma_k}}\cR\big[g\big](Z)
 =  (-1)^k\cR\Big[\frac{\partial}{\partial y^{\tau_1}}\cdots
 \frac{\partial}{\partial y^{\tau_k}}(y^{\sigma_1}\cdots y^{\sigma_k}g)
 \Big](Z),
 \label{Rdiffrule}
 \end{eqnarray}
where we have set $Z_j:=\delta_{jt}Z^t.$
 \end{theorem}
 \proof
We will prove this statement by induction over $k\in\N.$
Notice first, however,  that for a differentiable curve
$\alpha:(-\eps_0,\eps_0)\to SL(m+1)$ with $\alpha(0)=
\Id_{\R^{m+1}}$ and $\alpha'(0)=V\in
T_{\Id_{\R^{m+1}}}SL(m+1)\subset\R^{(m+1)\times (m+1)}$,
i.e., with $\trace V=0$, we can exploit \eqref{SL-invariance}
to find
$$
\cR\big[g\big]\circ \alpha(t)=\cR\big[g\circ (\alpha(t))^{-T}\big]\quad\Foa t\in (-\eps_0,\eps_0).
$$
According to Corollary \ref{cor:Rcontinuous} the  
left-hand side
is differentiable as a function of $t$ on $(-\eps_0,\eps_0)$,
so
that
upon differentiation with respect to $t$ at $t=0$ we obtain
\begin{equation}\label{ddt}
\frac{d}{dt}_{|_{t=0}}\left\{\cR\big[g\big]\circ \alpha(t)(Z)\right\}=
\frac{d}{dt}_{|_{t=0}}
\left\{\cR\big[g\circ (\alpha(t))^{-T}\big](Z)\right\}
\end{equation}
for arbitrary $Z\in \R^{m+1}\setminus\{0\}.$ The left-hand
side of this identity can be computed as
\begin{multline*}
\frac{d}{dt}_{|_{t=0}}\left\{\cR\big[g\big]\circ \alpha(t)(Z)\right\}=
\frac{d}{dt}_{|_{t=0}}\left\{\cR\big[g\big](\alpha(t)Z)\right\}\\
 = \frac{\partial}{\partial Z^i}\cR\big[g\big](Z)\frac{d}{dt}_{|_{t=0}}
(\alpha^i_j(t)Z^j)=\frac{\partial}{\partial Z^i}\cR\big[g\big](Z)
V^i_jZ^j,
\end{multline*}
whereas the right-hand side of \eqref{ddt}
yields (because of the linearity of $\cR[\cdot]$)
\detail{

\bigskip

The linearity of $\cR[\cdot]$ allows us to pull the differentiation
with respect to $t$ into the argument of $\cR$ since one
can pull through difference quotients $\Delta_h$: If $\sigma(t)$ is
the only $t$-dependence then
\begin{eqnarray*}
\Delta_h(\cR[\sigma(\cdot)](Z))(t)& = &
\frac 1h \left\{\cR[\sigma(t+h)](Z)-
\cR[\sigma(t)](Z)\right\}\\
& = &
\frac 1h \left\{\cR[\sigma(t+h)-\sigma(t)](Z)\right\}\\
& = &
\cR[\Delta_h\sigma(t)](Z).
\end{eqnarray*}
{\tt  Im n\"achsten  Schritt wird ein Satz ueber Parameterintegrale oder der Satz von Lebesgue zur majorisierten Konvergenz noetig sein, zur Sicherstellung der Vertauschbarkeit der Grenzprozesse.}
\bigskip

}
\begin{eqnarray*}
\frac{d}{dt}_{|_{t=0}}\left\{\cR\big[g\circ (\alpha(t))^{-T}\big](Z)
\right\}
& = &
\cR\Big[\frac{d}{dt}_{|_{t=0}}\{g\circ (\alpha(t))^{-T}\Big](Z)\\
& = &
\cR\Big[\frac{\partial}{\partial y^i}g(\cdot)((\alpha^{-T})'(0))_j^iy^j\Big](Z)\\
& = &
\cR\Big[\frac{\partial}{\partial y^i}g(\cdot)(-V^T)^i_jy^j\Big](Z)\\
& = &
-\cR\Big[(V^T)^i_jy^j\frac{\partial}{\partial y^i}g(\cdot)\Big](Z),
\end{eqnarray*}
where we have used that
$$
0=\frac{d}{dt}_{|_{t=0}}\Id_{\R^{m+1}}=\frac{d}{dt}_{|_{t=0}}
\left\{\alpha(t)^{-T}\alpha(t)^T\right\}=(\alpha^{-T})'(0)+
\alpha'(0)^{T}.
$$
Setting $W:=V^T$ and recalling that $Z_j=\delta_{jt}Z^t$ we can thus
rewrite \eqref{ddt} as
\begin{equation}\label{WR}
W^j_iZ_j\frac{\partial}{\partial Z_i}\cR\big[g\big](Z)
=-\cR\Big[W^i_jy^j\frac{\partial}{\partial y^i}g(\cdot)
\Big](Z)\quad\Fo g\in C^k(\R^{m+1}\setminus\{0\}).
\end{equation}
This relations holds for {\it any}
trace-free matrix $W\in\R^{(m+1)\times (m+1)}.$

In addition, we will also use the $(-1)$-homogeneity of
$\cR\big[g\big]$ and the  Euler identity
 to obtain
 \begin{eqnarray}
 Z_i\frac{\partial}{\partial Z_i}\cR\big[g\big](Z) & = &
 -\cR\big[g\big](Z)\label{ER}\quad\Fo Z\in\R^{m+1}\setminus\{0\}.
 \end{eqnarray}
%  Furthermore, by homogeneity of $g$ 
% \begin{eqnarray}
% -\frac{\partial}{\partial y^i}(y^ig)(y) & = & -(m+1)g(y)-y^i\frac{\partial}{\partial y^i}g(y)=-g(y),\label{g}NOETIG\xx ???
% \end{eqnarray}
Now we are in the position to prove \eqref{Rdiffrule} for $k=1$.
We choose for fixed $\tau,\sigma
\in\{1,\ldots,m+1\}$ the trace-free matrix
$$
W:=(W^i_j):=(\delta^i_\tau\delta^\sigma_j-\frac{1}{m+1}\delta^i_j
\delta^\sigma_\tau)
$$
\detail{

\bigskip

\begin{eqnarray*}
\trace W&=&\sum_{i}^{m+1}W^i_i=\sum_{i}^{m+1}(\delta^i_\tau\delta^\sigma_i-\frac{1}{m+1}\delta^i_i\delta^\sigma_\tau)\\
& = & (1-\frac{m+1}{m+1})\delta^\sigma_\tau=0.
\end{eqnarray*}

\bigskip

}
to deduce by means of \eqref{ER} for the left-hand side of
formula \eqref{WR}
\begin{eqnarray}
W^j_iZ_j\frac{\partial}{\partial Z_i}\cR\big[g\big](Z) & = &
(\delta^j_\tau\delta^\sigma_i-\frac{1}{m+1}\delta^j_i\delta^\sigma_\tau)
Z_j\frac{\partial}{\partial Z_i}\cR\big[g\big](Z)\notag\\
& = &
Z_\tau\frac{\partial}{\partial Z_\sigma}\cR\big[g\big](Z)
-\frac{1}{m+1}Z_i\frac{\partial}{\partial Z_i}\cR\big[g\big](Z)\delta^\sigma_\tau\notag\\
& \overset{\eqref{ER}}{=} & 
Z_\tau\frac{\partial}{\partial Z_\sigma}\cR\big[g\big](Z)
+\frac{1}{m+1}\cR\big[g\big](Z)
\delta^\sigma_\tau
\label{zwischen},
\end{eqnarray}
whereas for the right-hand side of \eqref{WR} one computes with the
homogeneity of $g$
\begin{eqnarray*}
-\cR\Big[W^i_jy^j\frac{\partial}{\partial y^i}g(\cdot)\Big](Z) & = &
-\cR\Big[(\delta^i_\tau\delta^\sigma_j-\frac{1}{m+1}\delta^i_j\delta^\sigma_\tau)y^j\frac{\partial}{\partial y^i}g(\cdot)\Big](Z)\\
& = &
-\cR\Big[y^\sigma\frac{\partial}{\partial y^\tau}g(\cdot)
-\frac{1}{m+1}y^i\frac{\partial}{\partial y^i}g(\cdot)\delta^\sigma_\tau
\Big](Z)\\
& =  &
-\cR\Big[y^\sigma\frac{\partial}{\partial y^\tau}g(\cdot)
+\frac{m}{m+1}g(\cdot)\delta^\sigma_\tau\Big](Z)\\
& = &
-\cR\Big[\frac{\partial}{\partial y^\tau}\big(y^\sigma g(\cdot)\big)-
\delta^\sigma_\tau g(\cdot)+\frac{m}{m+1}g(\cdot)\delta^\sigma_\tau\Big](Z)\\
& = &
-\cR\Big[\frac{\partial}{\partial y^\tau}\big(y^\sigma g(\cdot)\big)-
\frac{1}{m+1}g(\cdot)\delta^\sigma_\tau\Big](Z)\\
& = &
-\cR\Big[\frac{\partial}{\partial y^\tau}\big(y^\sigma g(\cdot)\big)\Big](Z)+
\frac{\delta^\sigma_\tau}{m+1}\cR\big[g\big](Z),
\end{eqnarray*}
which together with \eqref{zwischen} leads to
$$
Z_\tau\frac{\partial}{\partial Z_\sigma}\cR\big[g\big](Z)=
-\cR\Big[\frac{\partial}{\partial y^\tau}\big(y^\sigma g(\cdot)\big)\Big](Z),
$$
that is, the desired identity \eqref{Rdiffrule} for $k=1$ establishing
the induction hypothesis. 
Let us now assume for the induction step that \eqref{Rdiffrule}
holds true for $l=1,\ldots,k,$ and we shall  prove it also
for $l=k+1.$ Repeatedly applying the product rule and by virtue of
the induction hypothesis for $l=k$ and for $l=1$, we find
\begin{eqnarray*}
&& \hspace{-2cm}Z_{\tau_1}\cdots Z_{\tau_{k+1}}\frac{\partial}{\partial Z_{\sigma_1}}\cdots
\frac{\partial}{\partial Z_{\sigma_{k+1}}}\cR\big[g\big](Z)\\
&& \hspace{-2cm} +\sum_{l=1}^k\delta^{\tau_l}_{\sigma_{k+1}}
Z_{\tau_1}\cdots Z_{\tau_{l-1}}\widehat{Z_{\tau_l}}Z_{\tau_{l+1}}
\cdots Z_{\tau_{k+1}}\frac{\partial}{\partial Z_{\sigma_1}}\cdots
\frac{\partial}{\partial Z_{\sigma_{k}}}\cR\big[g\big](Z)\\
& = & Z_{\tau_{k+1}}\frac{\partial}{\partial Z_{\sigma_{k+1}}}
\Big\{Z_{\tau_1}\cdots Z_{\tau_{k}}
\frac{\partial}{\partial Z_{\sigma_1}}\cdots
\frac{\partial}{\partial Z_{\sigma_{k}}}\cR\big[g\big]\Big\}(Z)\\
& \overset{\eqref{Rdiffrule}\,\,\textnormal{for $l=k$}}{=} &
Z_{\tau_{k+1}}\frac{\partial}{\partial Z_{\sigma_{k+1}}}
\Big\{(-1)^k\cR\Big[\frac{\partial}{\partial y^{\tau_1}}\cdots
\frac{\partial}{\partial y^{\tau_k}}(y^{\sigma_1}\cdots y^{\sigma_k}g)\Big]\Big\}(Z)\\
& \overset{\eqref{Rdiffrule}\,\,\textnormal{for $l=1$}}{=} & (-1)^{k+1}\cR\Big[\frac{\partial}{\partial y^{\tau_{k+1}}}
\Big(y^{\sigma_{k+1}}\frac{\partial}{\partial y^{\tau_1}}\cdots
\frac{\partial}{\partial y^{\tau_k}}(y^{\sigma_1}\cdots y^{\sigma_k}g)
\Big)
\Big](Z).
\end{eqnarray*}
Using the product rule one can carry out the differentiation on the right-hand side
to obtain
\begin{eqnarray*}
&& \hspace{-2cm}Z_{\tau_1}\cdots Z_{\tau_{k+1}}\frac{\partial}{\partial Z_{\sigma_1}}\cdots
\frac{\partial}{\partial Z_{\sigma_{k+1}}}\cR\big[g\big](Z)\\
& = &
(-1)^{k+1}\left\{\cR\Big[\frac{\partial}{\partial y^{\tau_{1}}}
\cdots \frac{\partial}{\partial y^{\tau_{k+1}}}
(y^{\sigma_{1}}\cdots y^{\sigma_k}y^{\sigma_{k+1}}g)\Big](Z)\right.\\
&&\left. -(-1)^{k+1}\cR\left[
\sum_{l=1}^k\delta^{\sigma_{k+1}}_{\tau_l}
\frac{\partial}{\partial y^{\tau_{1}}}
\cdots\frac{\partial}{\partial y^{\tau_{l-1}}}
\widehat{\frac{\partial}{\partial y^{\tau_{l}}}}
\frac{\partial}{\partial y^{\tau_{l+1}}}
\cdots\frac{\partial}{\partial y^{\tau_{k+1}}}
(y^{\sigma_{1}}\cdots y^{\sigma_k}g)\right](Z)\right\}\\
& = &
(-1)^{k+1}\cR\Big[\frac{\partial}{\partial y^{\tau_{1}}}
\cdots \frac{\partial}{\partial y^{\tau_{k+1}}}
(y^{\sigma_{1}}\cdots y^{\sigma_k}y^{\sigma_{k+1}}g)\Big](Z)\\
&&
+(-1)^{k+2}\sum_{l=1}^k\delta^{\sigma_{k+1}}_{\tau_l}
\cR\left[\frac{\partial}{\partial y^{\tau_{1}}}
\cdots\frac{\partial}{\partial y^{\tau_{l-1}}}
\widehat{\frac{\partial}{\partial y^{\tau_{l}}}}
\frac{\partial}{\partial y^{\tau_{l+1}}}
\cdots\frac{\partial}{\partial y^{\tau_{k+1}}}
(y^{\sigma_{1}}\cdots y^{\sigma_k}g)\right](Z)\\
& \overset{\eqref{Rdiffrule}\,\,\textnormal{for $l=k$}}{=} &
(-1)^{k+1}\cR\Big[\frac{\partial}{\partial y^{\tau_{1}}}
\cdots \frac{\partial}{\partial y^{\tau_{k+1}}}
(y^{\sigma_{1}}\cdots y^{\sigma_k}y^{\sigma_{k+1}}g)\Big](Z)\\
&&
+(-1)^{2k+2}\sum_{l=1}^k\delta^{\sigma_{k+1}}_{\tau_l}
Z_{\tau_1}\cdots Z_{\tau_{l-1}}\widehat{Z_{\sigma_{k+1}}}
Z_{\tau_{l+1}}\cdots
Z_{\tau_{k+1}}\frac{\partial}{\partial Z_{\sigma_1}}\cdots
\frac{\partial}{\partial Z_{\sigma_k}}\cR\big[g\big](Z),
\end{eqnarray*}
which proves \eqref{Rdiffrule}. 
\qed

For a function $g\in C^k(\R^{m+1}\setminus\{0\})$ we recall from \eqref{rho_k} the
semi-norms
(here for arbitrary dimension $m\ge 2$)
\begin{equation}\label{rho_k_allgemein}
\rho_l(g):=\max\{|D^\alpha g(\xi)|:\xi\in\S^m,|\alpha|\le l\}
\end{equation}
for $l=0,1,\ldots, k.$
\detail{

\bigskip

The space of $(-m)$-homogeneous $C^\infty(\R^{m+1}\setminus\{0\})$ functions together with the family of semi-norms $\{\rho_k \}_{k=0}^\infty$
 is a so called Fr\'echet space. But we will discuss this in \cite{overath-vdM-2012b} in detail...

BEMERKUNG: Unter dieser Wahl von Seminormen bildet $C^\infty(\R^{m+1}\setminus\{0\})$ ohne die Homogenit\"at 
keinen Fr\'echet space. Any sort of positive homogeneity should do...?? 
 I think so, but due to the $GL$-Invariance result, the chosen homogeneities seem to fit best to the present situation.

}
\begin{corollary}\label{cor:rho_k_Radon}
There is a constant  $C=C(m,k)$ such that for any $(-m)$-homogeneous function
$g\in C^k(\R^{m+1}\setminus\{0\})$
one has
\begin{equation}\label{rho_k_Radon}
\rho_k(\cR\big[g\big])\le C(m,k)\rho_k(g).
\end{equation}
\end{corollary}
\proof
By definition of $\cR$ one has
\begin{equation}\label{eins_rho_k}
|\cR\big[g\big](Z)|\le\frac{1}{|Z|}\max_{\S^m\cap Z^\perp}|g|\le
\frac{1}{|Z|}\max_{\S^m}|g|.
\end{equation}
Contracting \eqref{Rdiffrule} in Theorem \ref{thm:diffrule} 
by multiplication with $Z^{\tau_1}\cdots Z^{\tau_k}$ and summing
over  $\tau_1,\ldots,\tau_k$ from $1$ to $m+1$ we obtain
\begin{equation}\label{zwei_rho_k}
\frac{\partial}{\partial Z_{\sigma_1}}\cdots
\frac{\partial}{\partial Z_{\sigma_k}}\cR\big[g\big](Z)=
(-1)^k\frac{Z^{\tau_1}\cdots Z^{\tau_k}}{|Z|^{2k}}\cR\Big[
\frac{\partial}{\partial y^{\tau_1}}\cdots
\frac{\partial}{\partial y^{\tau_k}}(y^{\sigma_1}\cdots y^{\sigma_k}g(y)\Big](Z).
\end{equation}
Combining \eqref{zwei_rho_k} with \eqref{eins_rho_k} leads to
\begin{eqnarray}
\Big|\frac{\partial}{\partial Z_{\sigma_1}}\cdots
\frac{\partial}{\partial Z_{\sigma_k}}\cR\big[g\big](Z)\Big| & \le &
\frac{1}{|Z|^k}(m+1)^k\hspace{-0.5cm}\max_{(i_1,\ldots,i_k)\atop 1\le i_1,\ldots,i_k\le m+1}
\Big|\cR\Big[
\frac{\partial}{\partial y^{i_1}}\cdots
\frac{\partial}{\partial y^{i_k}}(y^{\sigma_1}\cdots y^{\sigma_k}g(y))\Big](Z)\Big|\notag\\
&& \hspace{-3cm}\overset{\eqref{eins_rho_k}}{\le}  
\frac{1}{|Z|^{k+1}}(m+1)^k\max_{(i_1,\ldots,i_k)\atop 1\le i_1,\ldots,i_k\le m+1}
\max_{\S^m}\Big|\frac{\partial}{\partial y^{i_1}}\cdots
\frac{\partial}{\partial y^{i_k}}(y^{\sigma_1}\cdots y^{\sigma_k}g(y))\Big|.\label{drei_rho_k}
\end{eqnarray}
Now for any choice $i_1,\ldots,i_k\in\{1,\ldots,m+1\}$ one can write by the product rule
\begin{eqnarray*}
\Big|\frac{\partial}{\partial y^{i_1}}\cdots
\frac{\partial}{\partial y^{i_k}}(y^{\sigma_1}\cdots y^{\sigma_k}g(y))\Big|
=\Big|D^\alpha_y(y^\beta g(y))\Big| & = &
\left|\sum_{0\le\gamma\le\alpha}\left(\begin{array}{c}
\alpha\\
\gamma\end{array}\right)D^\gamma_y(y^\beta)D^{\alpha-\gamma}_yg(y)\right|\\
&\le & \left|\sum_{0\le\gamma\le\alpha}\left(\begin{array}{c}
\alpha\\
\gamma\end{array}\right)D^\gamma_y(y^\beta)\rho_{|\alpha-\gamma|}(g)\right|
\end{eqnarray*}
for some multi-indices $\alpha,\beta\in\N^{m+1}$ with $|\alpha|=|\beta|=k$ and
$y\in\S^m.$ Hence \eqref{drei_rho_k} becomes with this notation
$$
\Big|D^\beta_Z\cR\big[g\big](Z)\Big|\le\frac{1}{|Z|^{k+1}}(m+1)^k\rho_k(g)
\max_{y\in\S^m} \left|\sum_{0\le\gamma\le\alpha}\left(\begin{array}{c}
 \alpha\\
 \gamma\end{array}\right)D^\gamma_y(y^\beta)\right|
 =:\frac{\rho_k(g)}{|Z|^{k+1}}C(m,k,\beta),
 $$
 which implies the result with $C(m,k):=\max_{|\beta|\le k} C(m,k,\beta).$
 \qed

\detail{

\bigskip

{\tt BRAUCHEN WIR ERST IN \cite{overath-vdM-2012b}}
\begin{corollary}\label{cor:RCk}
The extended Radon transformation is a bounded linear
operator from the space of positively $(-m)$-homogeneous
functions of class $C^k(\R^{m+1}\setminus\{0\})$ to the
space of positively $(-1)$-homogeneous functions of class
$C^k(\R^{m+1}\setminus\{0\})$.
\end{corollary}

}
\subsection{Existence of a perfect dominance function for the Finsler-area
integrand}\label{sec:3.2}
One further  conclusion from Lemma \ref{lem:general-radon}
is that the Cartan integrand $\cA^F$ defined
in \eqref{AF_space} can be rewritten in terms of the Radon transform:
\begin{corollary}\label{cor:cartan-radon}
Let $F\in C^0(\R^{m+1}\times\R^{m+1})$ satisfy $F(x,y)>0$
for $y\not= 0$, and $F(x,ty)=tF(x,y)$ for all $t>0$,
$(x,y)\in\R^{m+1}\times\R^{m+1}.$ Then
\begin{equation}\label{cartan-radon}
\cA^F(x,Z)=\frac{1}{\cR\big[F^{-m}(x,\cdot)\big](Z)}\quad\Fo
(x,Z)\in\R^{m+1}\times(\R^{m+1}\setminus\{0\}).
\end{equation}
\end{corollary}
\proof
According to Lemma \ref{lem:polar} we can write
$$
\cA^F(x,Z)=\frac{|Z|\mathscr{H}^m(B_1^m(0))}{
\int_{\S^{m-1}}\frac{1}{mF^m(x,\theta^\kappa f_\kappa)}\,d\mathscr{H}^{m-1}(\theta)}
$$
for an orthonormal basis $\{f_1,\ldots,f_m\}$ of the subspace
$Z^\perp\subset\R^{m+1}.$ Now apply Lemma \ref{lem:general-radon}
to the function $g:=F^{-m}(x,\cdot)$ for any fixed $x\in\R^{m+1}$,
and use the identity
 \,$m\mathscr{H}^m(B_1^m(0))=\mathscr{H}^{m-1}(\S^{m-1})$
   to conclude.
      \qed

\begin{lemma}\label{lem:explicit_est}
For every fixed $x\in\R^{m+1}$ there 
is a constant $C=C(m,k,m_2(x),c_1(x))$ depending only on the dimension $m$,
the order of differentiation $k\in\N\cup\{0\}$, the constant $m_2(x)$ from Lemma
\ref{lem:pointwise}, and on the lower bound $c_1(x):=\inf_{\mbbbs^m}F(x,\cdot)$ on $F(x,\cdot)$,  such that
\begin{equation}\label{rho_difference}
\rho_k(|\cdot|-\cA^F(x,\cdot))\le C(m,k,m_2(x),c_1(x))
\rho_k(|\cdot|-F(x,\cdot))\hat{\rho}_k^{2k^2 +mk-1}(F(x,\cdot)),
\end{equation}
where we set $\hat{\rho}_k(f):=\max\{1,\rho_k(f)\}.$
\end{lemma}
\proof
We start with some  general observations for functions $f,g,h\in C^k(\R^{m+1}\setminus
\{0\})$ with $f,g>0$ on the unit sphere $\S^m.$ Henceforth, $C(m,k)$ will denote
generic constants depending on $m$ and $k$ that may change from line to line.

By the product rule we have 
$$
D^\alpha (f g)=\sum_{0\le\beta\le\alpha}\left(\begin{array}{c}\alpha\\
\beta\end{array}\right)D^\beta fD^{\alpha-\beta}g,
$$
so that  for $|\alpha|\le k$
\begin{eqnarray*}
|D^\alpha (fg)| & \le & \sum_{0\le\beta\le\alpha}\left(\begin{array}{c}\alpha\\
\beta\end{array}\right)\rho_{|\beta|}(f)\rho_{|\alpha-\beta|}(g)\\
&\le &
\rho_k(f)\rho_k(g)\sum_{0\le\beta\le\alpha}\left(\begin{array}{c}\alpha\\
\beta\end{array}\right)=:C(m,k)\rho_k(f)\rho_k(g),
\end{eqnarray*}
which implies 
\begin{equation}\label{rho-eins}
\rho_k(fg)\le C(m,k)\rho_k(f)\rho_k(g).
\end{equation}
Inductively we obtain
$$
|D^\alpha (f^m)| \le  C(m,k)\rho_k(f)\rho_k(f^{m-1})
\le  C^2(m,k)\rho_k^2(f)\rho_k(f^{m-2}) \le 
\cdots 
 \le  C^m(m,k)\rho_k^m(f),
$$
whence 
\begin{equation}\label{rho-zwei}
\rho_k(f^m)\le C^m(m,k)\rho_k^m(f).
\end{equation}
One also has
\begin{eqnarray}
D^\alpha\Big(\frac{h}{fg}\Big) & = & \sum_{0\le\beta\le\alpha}
\left(\begin{array}{c}\alpha\\
\beta\end{array}\right)D^\beta hD^{\alpha-\beta}\Big(
\frac{1}{fg}\Big)\notag\\
& = &  \sum_{0\le\beta\le\alpha}\left(\begin{array}{c}\alpha\\
\beta\end{array}\right)D^\beta h \sum_{0\le\gamma\le\alpha-\beta}
\left(\begin{array}{c}\alpha-\beta\\
\gamma\end{array}\right)D^\gamma\Big(\frac 1f \Big)D^{\alpha-\beta-\gamma}\Big(
\frac 1g \Big)\label{rho-drei}.
\end{eqnarray}
To estimate some of the derivative terms we set
$$
f_0:=\min\{1,\min_{\S^m}f\}\quad\textnormal{and recall}\quad \hat{\rho}_k(f):=\max\{1,\rho_k(f)\}.
$$
{\it Claim: For all $k=0,1,2,\ldots $ and $p\ge 1 $ there is a constant $C(m,k,p)$
 such that}
\begin{equation}\label{rho-vier}
\Big|D^\alpha \Big(\frac{ 1}{f^p} \Big)(\xi)\Big|\le
\frac{C(m,k,p)}{f_0^{k+p}}\hat{\rho}_k^k(f)\quad\Foa\xi\in\S^m,\,|\alpha|\le k.
\end{equation}
We prove this  claim by induction over $k$ and notice that for $k=0$ this
is a trivial consequence from the definition of $f_0$ and $\hat{\rho}_k$.
For the induction step we may assume that for all $l=0,\ldots,k$ there
is a constant $C(n,l)$ such that
$$
\Big|D^{\bar{\alpha}} \Big(\frac{1}{f^p} \Big)(\xi)\Big|\le
\frac{C(m,l, p)}{f_0^{l+p}}\hat{\rho}_l^l(f)\quad\Foa |\bar{\alpha}|\le l.
$$
For a multi-index $\alpha$ with $|\alpha|\le k+1$ we find a standard
basis vector $e_l$ and a multi-index $\bar{\alpha}$ with
$|\bar{\alpha}|\le k$ such that $\alpha=\bar{\alpha}+e_l.$ Then we compute
at $\xi\in\S^m$
\begin{eqnarray*}
\Big|D^\alpha \Big(\frac{1}{f^p} \Big) \Big|& = &
\Big|D^{\bar{\alpha}}\partial_l\Big(\frac{1}{f^p} \Big)\Big|=
\Big|D^{\bar{\alpha}}\Big(-\frac{p}{f^{p+1}}\partial_lf\Big)\Big|
 =  \Big|\sum_{0\le\beta\le\bar{\alpha}}\left(\begin{array}{c}\bar{\alpha}\\
\beta\end{array}\right)D^\beta \Big(\frac{p}{f^{p+1}}\Big)D^{\bar{\alpha}-\beta}\Big(\partial_lf\Big)
\Big|\\
& = & \Big|\frac{p}{f^{p+1}}D^{\bar{\alpha}}\partial_lf+\sum_{0\le\beta\le\bar{\alpha}
\atop \beta\not= 0}\left(\begin{array}{c}\bar{\alpha}\\
\beta\end{array}\right)D^\beta \Big(\frac{p}{f^{p+1}}\Big)D^{\bar{\alpha}-\beta}\Big(\partial_lf\Big)
\Big|\\
& = &
\Big|\frac{p}{f^{p+1}}D^{\bar{\alpha}}\partial_lf+\sum_{0\le\beta\le\bar{\alpha}
\atop \beta\not= 0}\left(\begin{array}{c}\bar{\alpha}\\
\beta\end{array}\right)\left[\sum_{0\le\gamma\le\beta}\left(\begin{array}{c}\beta\\
\gamma\end{array}\right)D^\gamma\Big(\frac{p}{f^p} \Big)D^{\beta-\gamma}\Big(\frac 1f \Big)\right]
D^{\bar{\alpha}-\beta}\Big(\partial_l f\Big)\Big|.
\end{eqnarray*}
Using the induction hypothesis in each of the summands and the definition of $f_0$
and $\hat{\rho}_k$ we arrive at
\begin{multline*}
\Big|D^\alpha \Big(\frac{1}{f^p} \Big)(\xi)\Big|\le 
\frac{p}{f_0^{p+1}}\hat{\rho}_{k+1}(f)\\
\qquad +\sum_{0\le\beta\le\bar{\alpha}
\atop \beta\not= 0}\left(\begin{array}{c}\bar{\alpha}\\
\beta\end{array}\right)\left[\sum_{0\le\gamma\le\beta}\left(\begin{array}{c}\beta\\
\gamma\end{array}\right)\frac{C(m,|\gamma|, p)}{f_0^{|\gamma|+p}}\hat{\rho}_{|\gamma|}^{|\gamma|}
(f)\frac{C(m,|\beta-\gamma|, 1)}{f_0^{|\beta-\gamma|+1}}\hat{\rho}_{|\beta-\gamma|}^{|\beta-\gamma|}(f)
\right]\hat{\rho}_{|\alpha-\beta|+1}(f)\\
\le 
\frac{p}{f_0^{p+1}}\hat{\rho}_{k+1}(f)+\sum_{0\le\beta\le\bar{\alpha}
\atop \beta\not= 0}\left(\begin{array}{c}\bar{\alpha}\\
\beta\end{array}\right)\left[\sum_{0\le\gamma\le\beta}\left(\begin{array}{c}\beta\\
\gamma\end{array}\right)\frac{C(m,|\gamma|, p)C(m,|\beta-\gamma|,1)\hat{\rho}_{|\beta|}^{|\beta|}(f)}{
f_0^{|\beta|+p+1}}
\right]\hat{\rho}_{k+1} (f)
\end{multline*}
at $\xi\in\S^m,$
which implies the claim since $|\beta|\le|\bar{\alpha}|\le k$ and $f_0\le 1$, $\hat{\rho}_k(f)
\ge 1.$

As an immediate consequence of 
\eqref{rho-zwei}--\eqref{rho-vier} we estimate 
\begin{eqnarray}
\Big| D^\alpha \Big(\frac{h}{fg}\Big)(\xi)\Big|
& \le & \rho_k(h)\sum_{0\le\beta\le\alpha}\left(\begin{array}{c}
\alpha\\\beta\end{array}\right)\sum_{0\le\gamma\le\alpha-\beta}
\left(\begin{array}{c}
\alpha-\beta\\\gamma\end{array}\right)\frac{C(m,|\gamma|, 1)}{f_0^{|\gamma|+1}}\hat{\rho}_{|\gamma|}^{|\gamma|}(f)\Big|D^{\alpha-\beta-\gamma}\Big(\frac{1}{g(\xi)}\Big)\Big|\notag\\
& \le & 
\frac{\rho_k(h)\hat{\rho}_k^k(f)}{f_0^{k+1}}\sum_{0\le\beta\le\alpha}\left(\begin{array}{c}
\alpha\\\beta\end{array}\right)\sum_{0\le\gamma\le\alpha-\beta}C(m,|\gamma|, 1)
\left(\begin{array}{c}
\alpha-\beta\\\gamma\end{array}\right)\Big|D^{\alpha-\beta-\gamma}\Big(\frac{1}{g(\xi)} \Big)\Big|\notag\\
& =: & C(m,k,g)\frac{\rho_k(h)\hat{\rho}_k^k(f)}{f_0^{k+1}}\quad\Foa |\alpha|\le k,
\AND\xi\in\S^m
\label{rho-sechs}
\end{eqnarray}
so that for $ q\ge 1$
\begin{eqnarray}
\Big|D^\alpha\Big(\frac{1}{f^q}-\frac{1}{g^q}\Big)(\xi)\Big| & = &
\Big|D^\alpha\Big(\frac{g^q-f^q}{f^qg^q}\Big)(\xi)\Big|\notag\\
& \overset{\eqref{rho-sechs}}{\le} & 
C(m,k,g^q)\frac{\rho_k(g^q-f^q)\hat{\rho}_k^k(f^q)}{f_0^{q(k+1)}}
\Foa \xi\in\S^m\label{rho-sieben}
\end{eqnarray}
Using \eqref{rho-zwei} and the identity 
$$
g^q-f^q=(g-f)(g^{q-1}+fg^{q-2}+\cdots + f^{q-1})
$$ 
one obtains (with  new constants $C(q,k,g)$) the inequality
\begin{eqnarray}
\Big|D^\alpha\Big(\frac{1}{f^q}-\frac{1}{g^q}\Big)(\xi)\Big|
& \le &
C(m,k,q,g)\rho_k(g-f)\hat{\rho}_k^{qk}(f)\Big[1+\hat{\rho}_k(f)+\cdots +\hat{\rho}_k^{q-1}(f)\Big] \frac{1}{f_0^{q(k+1)}}
\notag\\
&\le & C(m,k,q,g)\rho_k(g-f)\hat{\rho}_k^{(q+1)k-1 }(f) \frac{1}{f_0^{q(k+1)}}
\Foa\xi\in\S^m.\label{rho-acht}
\end{eqnarray}
After these preparations we are ready to prove \eqref{rho_difference}. 
To estimate $\rho_k(|\cdot|-\cA^F(x,\cdot))$ for fixed $x\in\R^{m+1}$,
where the derivatives are taken with respect to $Z\in\R^{m+1}\setminus
\{0\}$ we first write by means of \eqref{cartan-radon}
\begin{eqnarray*}
D^\alpha\Big(|\cdot|-\cA^F(x,\cdot)\Big)&\overset{\eqref{cartan-radon}}{=}&
D^\alpha\left(
\frac{1}{\cR\big[|\cdot|^{-m}\big]}-
\frac{1}{\cR\big[F^{-m}(x,\cdot)\big]}\right)\\
& = &
D^\alpha\left(
\frac{\cR\big[F^{-m}(x,\cdot)-|\cdot|^{-m}\big]}{\cR\big[|\cdot|^{-m}\big]
\cR\big[F^{-m}(x,\cdot)\big]}\right).
\end{eqnarray*}
(Here we used linearity of the Radon transform $\cR\big[\cdot\big].$)

According to Corollary \ref{cor:cartan-radon} and Corollary
\ref{cor:pointwise} one has
\begin{equation}\label{rho-neun}
\cR\big[F^{-m}(x,\cdot)\big](Z)=\frac{1}{\cA^F(x,Z)}\ge
\frac{1}{m_2(x)|Z|}\quad\Foa Z\in\R^{m+1}\setminus\{0\},
\end{equation}
so that we can use \eqref{rho-sechs} for $h:=\cR\big[
F^{-m}(x,\cdot)-|\cdot|^{-m}\big],$ $f:=\cR\big[F^{-m}(x,\cdot)\big]$,
and $g:=\cR\big[|\cdot|^{-m}\big]$ for fixed $x\in\R^{m+1}$
to obtain
\begin{eqnarray*}
\rho_k(|\cdot|-\cA^F(x,\cdot)) &\overset{\eqref{rho-acht}}{\le}&
C(m,k,1,m_2(x))\rho_k(\cR\big[F^{-m}(x,\cdot)-|\cdot|^{-m}\big])
\hat{\rho}_k^{2k-1}(\cR\big[F^{-m}(x,\cdot)\big])\\
&\overset{\eqref{rho_k_Radon}}{\le} & C'(m,k,1,m_2(x)) \rho_k(F^{-m}(x,\cdot)-
|\cdot|^{-m})\hat{\rho}_k^{2k-1}(F^{-m}(x,\cdot));
\end{eqnarray*}
the last inequality follows from \eqref{rho_k_Radon}.

We estimate further by means of \eqref{rho-acht} and
\eqref{rho-vier} for $f:=F(x,\cdot)$ and $g:=|\cdot|
$ and $q:=m$ to find for fixed $x\in\R^{m+1}$
$$
\rho_k(|\cdot|-\cA^F(x,\cdot))\le C(m,k,m_2(x),c_1(x))
\rho_k(|\cdot|-F(x,\cdot))\hat{\rho}_k^{(m+1)k-1}(F(x,\cdot))
\hat{\rho}_k^{k(2k-1)}(F(x,\cdot)),
$$
where now the constant depends also on  $c_1(x)=\inf_{\mbbbs^m}F(x,\cdot),$
which is \eqref{rho_difference}.
\qed

In order to apply the existing  
regularity theory for (possibly
branched) minimizers of
Cartan functionals to the Finslerian area functional
$\area^F$
we need to prove  that under suitable conditions on the
underlying Finsler metric $F$ the Cartan integrand $\cA^F$ satisfies
the following {\it parametric
ellipticity} condition (formulated in the general target dimension $n$ with
$N:=n(n-1)/2$, cf. \cite[p. 298]{HilvdM-dominance}):
\begin{definition}[(Parametric) ellipticity]\label{def:ellipticity}
A Cartan integrand 
$\cC=\cC(x,Z)\in C^2(\R^{n}\times\R^{N}\setminus\{0\})$ (satisfying \eqref{homo-cartan}
on p. \pageref{homo-cartan})
is called
{\em
elliptic} if and only if for every $R_0>0$ there is some number
$\lambda_{\cC}(R_0)>0$ such that the Hessian
$\cC_{ZZ}(x,Z)-\lambda_{\cC}(R_0)\cA^E_{ZZ}(x,Z)$
is positive semi-definite\footnote{The stronger form of uniform ellipticity, i.e., 
a positive definite Hessian $\cC_{ZZ}(x,Z)$ cannot be expected because of the
homogeneity \eqref{homo-cartan}, which implies $\cC_{ZZ}(x,Z)Z=0.$}
for all $(x,Z)\in\overline{B_{R_0}(0)}
\times (\R^{N}\setminus\{0\}).$
\end{definition}
(Recall from the introduction that $\cA^E(Z)=|Z|$ 
denotes the classic area integrand generated by the Euclidean 
metric $E(y):=|y|$ in place of a general Finsler metric $F(x,y)$.)

The concept of a dominance function was introduced  in \cite{HilvdM-dominance}
for Cartan functionals on two-dimensional domains (but for surfaces in
any co-dimension, i.e. with $n\ge 2$). Denote by
$$
\mathpzc{c}(x,p):=\cC(x,p_1\wedge p_2)\quad\Fo p=(p_1,p_2)\in\R^n\times\R^n\simeq
\R^{2n},\,x\in\R^n,
$$
the {\it associated Lagrangian} of $\cC$.
\begin{definition}[Perfect dominance function
\protect{\cite[Definition 1.2]{HilvdM-dominance}}]\label{def:perfectdom}
A {\em perfect dominance function} for the Cartan integrand
$\cC$ with associated Lagrangian $\mathpzc{c}$
is a function $G\in C^0(\R^n\times\R^{2n})
\cap C^2(\R^n\times (\R^{2n}\setminus\{0\}))$ 
satisfying the following conditions for $x\in\R^n$ and $p=(p_1,p_2)\in\R^{2n}$:
\begin{enumerate}
\item[\rm (D1)]
$\mathpzc{c}(x,p)\le G(x,p)$ with 
\item[\rm (D2)]
$\mathpzc{c}(x,p)= G(x,p)$ if and only if 
 $|p_1|^2=|p_2|^2$ and $p_1\cdot p_2=0$;
\item[\rm (D3)]
$G(x,tp)=t^2G(x,p)$ for all $t>0$;
\item[\rm (D4)]
there are constants $0<\mu_1\le\mu_2$ such that
$\mu_1|p|^2\le G(x,p)\le\mu_2|p|^2$;
\item[\rm (E)\,\,\,]
for any $R_0>0$ there is a constant $\lambda_G(R_0)>0$ such that
$$
\xi\cdot G_{pp}(x,p)\xi\ge\lambda_G(R_0)|\xi|^2 
\quad\textnormal{for $|x|\le R_0$, $p\not=0$, $\xi\in\R^{2n}$.}
$$
\end{enumerate}
\end{definition}
We quote from  \cite{HilvdM-dominance} the following quantitative sufficient
criterion for the existence of a perfect dominance function.
\begin{theorem}[Perfect dominance function, Thm. 1.3 in \cite{HilvdM-dominance}]
\label{thm:perfect_dominance}
Let $\cC^*\in C^0(\R^{n}\times\R^{N})\cap C^2(\R^{n}\times (\R^{N}\setminus\{0\}))$
be a Cartan integrand satisfying conditions {\rm \eqref{homo-cartan}, 
\eqref{D}} (see pages \pageref{homo-cartan}, \pageref{D}) with constants $m_1(\cC^*),
m_2(\cC^*).$ In addition,
let $\cC^*$ be elliptic in the sense of Definition \ref{def:ellipticity} with
\begin{equation}\label{uniform_ell}
\lambda(\cC^*):=\inf_{R_0\in (0,\infty]}\lambda_{\cC^*}(R_0)>0.
\end{equation}
Then for 
\begin{equation}\label{k_cond}
k>k_0(\cC^*):=2[m_2(\cC^*)-\min\{\lambda(\cC^*),m_1(\cC^*)/2\}]
\end{equation}
the Cartan integrand $\cC$ defined by
\begin{equation}\label{area_perturbation}
\cC(x,Z):=k|Z|+\cC^*(x,Z)
\end{equation}
possesses a perfect dominance function. 
\end{theorem}
We can use this result and
the scale invariance in Definition \ref{def:perfectdom} to
quantify the $C^2$-deviation of a general Cartan integrand $\cC(x,Z)$
from the classic area integrand $\cA(Z):=|Z|$ that is tolerable
for the existence of a perfect dominance function for $\cC$.
\begin{corollary}\label{cor:delta}
If 
\begin{equation}\label{delta}
\delta:=\sup_{x\in\R^n}\{\rho_2(\cC(x,\cdot)-\cA(\cdot))\}<\frac 15,
\end{equation}
then $\cC$ possesses a perfect dominance function.
\end{corollary}
\proof
For $Z\in\R^N\setminus \{0\}$ and $x\in\R^n$ one estimates
\begin{eqnarray*}
\frac{1}{|Z|}\cC(x,Z)\overset{\eqref{homo-cartan}}{=}
\cC(x,Z/|Z|)& \ge & 
1-\Big|\cC(x,Z/|Z|)-|Z/|Z||\Big|\\
& \ge &
1-\rho_0(\cC(x,\cdot)-\cA(\cdot))\\
& \ge & 1-\delta,
\end{eqnarray*}
which implies
$R\cC(x,Z)\ge R(1-\delta)|Z|$ for any scaling factor $R>0$.
Thus if we take $R>(1-\delta)^{-1}$ we obtain 
$$
R\cC(x,Z)-\cA(Z)\ge [R(1-\delta)-1]|Z|>0\quad\Foa Z\not=0,
$$
and similarly,
$$
R\cC(x,Z)-\cA(Z)\le [R(1+\delta)-1]|Z|\quad\Foa Z\in\R^N.
$$
Hence for each $R>(1-\delta)^{-1}$
we obtain a new Cartan integrand $\cC_R(x,Z):=R\cC(x,Z)-\cA(Z)$ (satisfying
the homogeneity condition \eqref{homo-cartan}) and the growth
condition \eqref{D} with constants
\begin{equation}\label{growth_constants}
0<m_1(\cC_R):=R(1-\delta)-1\le m_2(\cC_R):=R(1+\delta)-1.
\end{equation}
Regarding the parametric ellipticity we estimate for fixed $x\in\R^n$ and
$Z\in\R^N\setminus\{0\}$
\begin{eqnarray*}
|Z|\xi\cdot\cC_{ZZ}(x,Z)\xi & = &
\xi\cdot\cC_{ZZ}(x,Z/|Z|)\xi \ge \xi\cdot \cA_{ZZ}(Z/|Z|)\xi-\Big|\xi
\cdot \Big[\cC_{ZZ}(x,Z/|Z|)-\cA_{ZZ}(Z/|Z|)\Big]\xi\Big|\\
&\ge &
\xi\cdot \cA_{ZZ}(Z/|Z|)\xi-|\Pi_{Z^\perp}\xi|^2\rho_2(\cC(x,\cdot)-
\cA(\cdot))=|\Pi_{Z^\perp}\xi|^2(1-\delta),
\end{eqnarray*}
which implies for any scaling factor $R>0$
$$
|Z|\xi\cdot R\cC_{ZZ}(x,Z)\xi \ge R(1-\delta)|\Pi_{Z^\perp}\xi|^2,
$$
where $\Pi_{Z^\perp}$ denotes the orthogonal projection onto the $(N-1)$-dimensional
subspace $Z^\perp.$
Hence the Cartan integrand $\cC_R(x,Z)=R\cC(x,Z)-\cA(Z)$ is an elliptic
parametric integrand in the sense of Definition \ref{def:ellipticity}, even
with the uniform estimate
\begin{equation}\label{uniform_ellipticity}
\lambda(\cC_R-\cA)\ge R(1-\delta)-1>0
\end{equation}
as long as $R>(1-\delta)^{-1}.$ 

Now we write the scaled Cartan integrand $\cC_R$ as
$$
\cC_R(x,Z)=\cA(Z)+(\cC_R(x,Z)-\cA(Z))=:\cA(Z)+\cC_R^*(x,Z),
$$
which is of the form \eqref{area_perturbation} with $k=1$. 
By virtue of \eqref{growth_constants} and \eqref{uniform_ellipticity} one
can calculate the quantity $k_0(\cC_R^*)$ in \eqref{k_cond} of
Theorem \ref{thm:perfect_dominance} as
\begin{eqnarray}
k_0(\cC_R^*) & = & 2[m_2(\cC_R^*)-\frac 12 m_1(\cC_R^*)]\notag\\
&\le & 2[R(1+\delta)-1-\frac{R(1-\delta)-1}{2}]=R+3R\delta-1\label{k_0}.
\end{eqnarray}
As $R$ tends to $(1-\delta)^{-1}$ from above the 
right hand side of the last estimate tends to 
$
4\delta/(1-\delta)$, which is less than one, since $\delta <1/5$ by
assumption \eqref{delta}. Hence we can find a scaling factor $R_0$
greater but sufficiently close to $(1-\delta)^{-1}$ such that 
$k_0(\cC_{R_0}^*) < 1=k$ so that according to Theorem \ref{thm:perfect_dominance}
the  scaled Cartan integrand $\cC_{R_0}=R_0\cC$ possesses a
 perfect dominance function. All defining properties of a perfect
 dominance function in Definition \ref{def:perfectdom} are scale
 invariant which implies that also the original Cartan integrand 
 $\cC$ possesses a perfect dominance function.
 \qed

{\sc Proof of Theorem \ref{thm:higher_reg}.}\,
Assume at first that 
\begin{equation}\label{initial_assumption}
\rho_2(F(x,\cdot)-|\cdot|)<1/2.
\end{equation}
Then 
$$
F(x,y)\ge |y|-|F(x,y)-|y||>|y|-\frac 12 =\frac 12\quad\Foa x\in\R^3,\,y\in\S^2
$$
so that the quantity
$c_1(x)=\inf_{S^2}F(x,\cdot)$ appearing in Lemma \ref{lem:explicit_est}
is bounded from below by $1/2$. Analogously, one finds $c_2(x)< 3/2, $
which implies $m_1(x)>1/4 $ and $m_2(x)<9/4$; see Lemma \ref{lem:pointwise}.
Thus the constant $C$  in \eqref{rho_difference} of 
Lemma \ref{lem:explicit_est} depends  only on $k$, since we have fixed the
dimension $m=2.$ Moreover, again by our initial assumption \eqref{initial_assumption},
one finds for any multi-index $\alpha\in\N^3$ with $|\alpha|\le 2$ 
and any $x\in\R^3$
$$
|D^\alpha_yF(x,y)|\le\rho_2(|\cdot|)+\rho_2(|\cdot|-F(x,y))\le C+\frac 12\quad\Foa
y\in\S^2,
$$
so that $\rho_2(F(x,\cdot))\le C+1/2.$
These observations under the initial assumption \eqref{initial_assumption}
reduce \eqref{rho_difference} in Lemma 
\ref{lem:explicit_est} for $m=2$ and $k=2$ to
the estimate
$$
\rho_2(\cA(\cdot)-\cA^F(x,\cdot))\le C\rho_2(|\cdot|-F(x,\cdot))\quad\Foa x\in\R^3
$$
with a universal and uniform constant $C$. Choosing now $\delta_0<1/(5C)$
in \eqref{suff_reg_condition} of Theorem \ref{thm:higher_reg} one finds
according to Corollary \ref{cor:delta} a perfect dominance function for
the Cartan integrand $\cA^F$, and we conclude with Theorem 1.9 in
\cite{HilvdM-courant} and Theorem 1.1 in \cite{HilvdM-crelle}.
\qed

\setnumbers
\section{Proof of Theorem \ref{thm:sufficient}}\label{sec:4}
We start this section with an auxiliary lemma involving binomial coefficients
$$
\left(\begin{array}{c}
n\\
k\end{array}\right)\quad\Fo n\in\N,\, k\in\Z,
$$
where we set 
$$
\left(\begin{array}{c}
n\\
k\end{array}\right)=0\quad\textnormal{if $k>n$ or if $k<0$.}
$$
\begin{lemma}\label{lem:4.1}
Let $m\in\N$ and $a\in (0,1/\sqrt{m-1})$ if $m>1$,  then
\begin{equation}\label{4.00}
f(a,m):=\sum_{k=0}^{\lfloor\frac m2 \rfloor}\left\{
\left(\begin{array}{c}
m\\
2k+1\end{array}\right)-\left(\begin{array}{c}
m\\
2k\end{array}\right)\right\}a^{2k}\ge 0.
\end{equation}
If $m$ is odd or if $m=2$, it suffices to have $a\in (0,1).$
\end{lemma}
\proof
We distinguish the cases $m=2q+1, $ $m=2(2q+1)$, and $m=4q,$ for some
$q\in\N\cup\{0\},$ and we can assume that $m>1$ since for $m=1$ the
statement is trivially true.

{\it Case I. $m=2q+1$ for some $q\in\N$.}\,
We write
\begin{eqnarray*}
2f(a,m) & =& 2\sum_{k=0}^{q}\left\{
\left(\begin{array}{c}
m\\
2k+1\end{array}\right)-\left(\begin{array}{c}
m\\
2k\end{array}\right)\right\}a^{2k}
-
\sum_{k=0}^{q}\left\{
\left(\begin{array}{c}
m\\
2k+1\end{array}\right)-\left(\begin{array}{c}
m\\
2k\end{array}\right)\right\}a^{2k}\\
&& +
\sum_{k=0}^{q}\left\{
\left(\begin{array}{c}
m\\
m-(2k+1)\end{array}\right)-\left(\begin{array}{c}
m\\
m-2k\end{array}\right)\right\}a^{2k},
\end{eqnarray*}
where we used the well-known identity
\begin{equation}\label{4.0}
\left(\begin{array}{c}
n\\
k
\end{array}\right)=
\left(\begin{array}{c}
n\\
n-k\end{array}\right)
\end{equation}
in the last sum.
Inserting $m=2q+1$ and substituting $l:=q-k$ we can rewrite the last sum
as
$$
\sum_{l=0}^{q}\left\{
\left(\begin{array}{c}
m\\
2l\end{array}\right)-\left(\begin{array}{c}
m\\
2l+1\end{array}\right)\right\}a^{2(q-l)}
$$
to obtain
$$
2f(a,m)=\sum_{k=0}^{q}\left\{
\left(\begin{array}{c}
m\\
2k+1\end{array}\right)-\left(\begin{array}{c}
m\\
2k\end{array}\right)\right\}(a^{2k}-a^{2(q-k)}).
$$
Since $0<a<1$ we realize that the second factor is nonnegative if and only if
$2(q-k)\ge 2k\Leftrightarrow q\ge 2k$, which is exactly the inequality that
ensures that the first factor is nonnegative by means of the general identity
\begin{equation}\label{4.1}
\left(\begin{array}{c}
n\\
k\end{array}\right)-\left(\begin{array}{c}
n\\k-1\end{array}\right)=
\frac{n+1-2k}{n+1}
\left(\begin{array}{c}
n+1\\
k\end{array}\right).
\end{equation}
If $2(q-k)<2k\Leftrightarrow q<2k,$ both factors in the $k$-th term of the sum
are negative, which proves $2f(a,m)\ge 0$ for odd $m\in\N$, if $a\in (0,1).$

\medskip

{\it Case II. $m=2(2q+1)$ for some $q\in\N\cup\{0\}$.}\,
We extract the last term of the sum and write
\begin{eqnarray*}
f(a,m) &= &\sum_{k=0}^{2q}\left\{
\left(\begin{array}{c}
m\\
2k+1\end{array}\right)-\left(\begin{array}{c}
m\\
2k\end{array}\right)\right\} a^{2k}-a^m\\
& = &
\sum_{k=0}^{2q}\left\{
\frac{m}{m-(2k+1)}
\left(\begin{array}{c}
m-1\\
2k+1\end{array}\right)-
\frac{m}{m-2k}
\left(\begin{array}{c}
m-1\\
2k\end{array}\right)\right\}-a^m,
\end{eqnarray*}
where we used the general identity
\begin{equation}\label{4.1A}
\left(\begin{array}{c}
n\\
k\end{array}\right)=
\frac{n}{n-k}
\left(\begin{array}{c}
n-1\\
k\end{array}\right)
\end{equation}
for all binomial terms. Separating the first term ($k=0$) and using the
trivial inequality
\begin{equation}\label{4.1B}
\frac{m}{m-(2k+1)}
\ge
\frac{m}{m-2k}
\quad\Fo k=1,\ldots,2q,
\end{equation}
we obtain
\begin{equation}\label{4.2}
f(a,m)\ge 1+\sum_{k=0}^{2q}\frac{m}{m-2k}
\left\{
\left(\begin{array}{c}
m-1\\
2k+1\end{array}\right)-
\left(\begin{array}{c}
m-1\\
2k\end{array}\right)\right\}a^{2k}-a^m,
\end{equation}
since
\begin{multline*}
\frac{m}{m-(2\cdot 0+1)}
\left(\begin{array}{c}
m-1\\
2\cdot 0+1\end{array}\right)-
\frac{m}{m-2\cdot 0}
\left(\begin{array}{c}
m-1\\
2\cdot 0\end{array}\right)=m-1\\
=1+\frac{m}{m-2\cdot 0}
\left(\begin{array}{c}
m-1\\
2\cdot 0+1\end{array}\right)-\frac{m}{m-2\cdot 0}
\left(\begin{array}{c}
m-1\\
2\cdot 0\end{array}\right).
\end{multline*}
According to \eqref{4.1} the terms in the sum in \eqref{4.2}
are nonnegative if and only if $k\le q$ and negative for $k>q$, so that
we can split the sum in two: one summing over $k$ from $0$ to $q$, and the
other from $q+1$ to $2q$.
Rewriting the second sum by means of \eqref{4.0} as
$$
\sum_{k=q+1}^{2q}\frac{m}{m-2k}
\left\{
\left(\begin{array}{c}
m-1\\
m-1-(2k+1)\end{array}\right)-
\left(\begin{array}{c}
m-1\\
m-1-2k\end{array}\right)\right\}a^{2k},
$$
which upon substituting $l:=2q-k$ yields
$$
\sum_{l=0}^{q-1}\frac{m}{m-2(2q-l)}
\left\{
\left(\begin{array}{c}
m-1\\
2l\end{array}\right)-
\left(\begin{array}{c}
m-1\\
2l+1\end{array}\right)\right\}a^{2(2q-l)},
$$
so that \eqref{4.2} becomes
\begin{multline}
f(a,m) \ge  1-a^m\label{4.3}\\
%&&\hspace{-3cm} 
+\sum_{k=0}^{q-1}\left\{
\left(\begin{array}{c}
m-1\\
2k+1\end{array}\right)-
\left(\begin{array}{c}
m-1\\
2k\end{array}\right)\right\}\left[
\frac{m}{m-2k}a^{2k}-\frac{m}{m-2(2q-k)}
a^{2(2q-k)}\right],
\end{multline}
since the isolated term for $k=q$ vanishes according to \eqref{4.1} in this case:
$$
\left(\begin{array}{c}
m-1\\
2q+1\end{array}\right)-
\left(\begin{array}{c}
m-1\\
2q\end{array}\right)\overset{\eqref{4.1}}{=}
\frac{m-2(2q+1)}{m}
\left(\begin{array}{c}
m\\
2q+1\end{array}\right)=0.
$$
Also, by \eqref{4.1}, all binomial differences in \eqref{4.3}
are positive\footnote{If $q=0\Leftrightarrow m=2$, the sum in
\eqref{4.3} vanishes altogether, so that $f(a,m)\ge 1-a^m>0$
for every $a\in (0,1).$}, since $0\le k<q.$ For the same range
$k=0,1\ldots, q-1$ one also estimates
\begin{eqnarray*}
\frac{m}{m-2k}a^{2k}-
\frac{m}{m-2(2q-k)}
a^{2(2q-k)}
&\ge &
\frac{m}{m-2k}a^{2k}-
\frac m2 a^{2(2q-k)}\\
&>& \frac{m}{m-2k}a^{2k}-
\frac{m}{2(m-1)}a^{4q-2k-2}\\
&>& \left[
\frac{m}{m-2k}-a^{4q-4k-2}\right]a^{2k}>\frac{2k}{m-2k}a^{2k}\ge 0,
\end{eqnarray*}
since $m>1$. Here
we used the assumption $(m-1)a^2<1$ for the first time. 
Consequently, $f(a,m)>1-a^m>0.$

\medskip

{\it Case III. $m=2(2q)=4q$ for some $q\in\N.$}\,
In this case we isolate the terms for $k=0$, $k=\lfloor m/2\rfloor=2q,$
and $k=2q-1$ from the remaining sum in the expression for $f(a,m)$, and use
\eqref{4.1A} and \eqref{4.1B} as in Case II to deduce
\begin{eqnarray}
f(a,m) &\ge & (m-1)a^0+\sum_{k=1}^{2q-2}
\frac{m}{m-2k}\left\{
\left(\begin{array}{c}
m-1\\
2k+1\end{array}\right)-
\left(\begin{array}{c}
m-1\\
2k\end{array}\right)\right\}a^{2k}\notag\\
&& \quad + \left\{
\left(\begin{array}{c}
m\\
m-1\end{array}\right)-
\left(\begin{array}{c}
m\\
m-2\end{array}\right)\right\}a^{m-2} - a^m.
\label{4.4}
\end{eqnarray}
Now  we split the remaining sum in \eqref{4.4} in half, and in the second
sum from $k=q$ to $k=2q-2$ we use \eqref{4.0} and the
substitution $l:=2q-k-1$ to obtain
\begin{multline*}
\sum_{k=q}^{2q-2}
\frac{m}{ m-2k}\left\{
\left(\begin{array}{c}
m-1\\
m-1-(2k+1)\end{array}\right)-
\left(\begin{array}{c}
m-1\\
m-1-2k\end{array}\right)\right\}a^{2k}\\
=\sum_{l=1}^{q-1}
\frac{m}{m-2(2q-l-1)}\left\{
\left(\begin{array}{c}
m-1\\
2l\end{array}\right)-
\left(\begin{array}{c}
m-1\\
2l+1\end{array}\right)\right\}a^{2(2q-l-1)}.
\end{multline*}
Inserting this into \eqref{4.4} we arrive at
\begin{multline}
f(a,m)  \ge  (m-1)+\left[m-\frac{m(m-1)}{2}\right]a^{m-2}-a^m\\
%&& \hspace{-2.5cm}
+\sum_{k=1}^{q-1}
\left\{\left(\begin{array}{c}
m-1\\
2k+1\end{array}\right)-
\left(\begin{array}{c}
m-1\\
2k\end{array}\right)\right\}\left[\frac{m}{m-2k}a^{2k}-
\frac{m}{m-2(2q-k-1)}a^{2(2q-k-1)}\right].\label{4.5}
\end{multline}
As in Case II the remaining binomial differences are positive for $k<q$, and 
for the last
term we estimate
$$
\frac{m}{m-2(2q-k-1)}\le\frac{m}{m-2(2q-2)}=\frac m2 \Fo k=1,\ldots,q-1,
$$
so that by means of the assumption $(m-1)a^2<1$ 
\begin{eqnarray*}
\frac{m}{m-2k}a^{2k}-\frac{m}{m-2(2q-k-1)}a^{2(2q-k-1)} &\ge &
\frac{m}{m-2k}a^{2k}-\frac m2 a^{m-2k-2}\\
& > & \frac{m}{m-2k}a^{2k}-\frac{m}{2(m-1)}a^{m-2k-4}\\
 &  &\hspace{-7cm} > \left[\frac{m}{m-2k}- a^{m-4k-4}\right]a^{2k}\ge
 \left[\frac{m}{m-2k}- 1 
 \right] a^{2k} >0\Fo k=1,\ldots,q-1.
 \end{eqnarray*}
 Consequently, we have the final estimate
 \begin{eqnarray*}
 f(a,m)  & > & (m-1)+\frac{3m-m^2}{2}a^{m-2}-a^m>(m-1)+\frac{(3-m)m}{2(m-1)}a^{m-4}-a^m\\
& > & (m-1)+\frac{3-m}{2}a^{m-4}-a^m
  >  \frac{m+1}{2}-a^m>0,
 \end{eqnarray*}
 since  $(m-1)a^2<1$, and because $m\ge 4$ in this case.
 \qed

\medskip

{\sc Proof of Theorem \ref{thm:sufficient}.}\, 
Since $f:=F_{\textnormal{sym}}$ is automatically as smooth as $F$ and satisfies the
homogeneity condition (F1) it is enough to show that its fundamental tensor
$g^f_{ij}:=(f^2/2)_{ij}=(F_{\textnormal{sym}}^2)_{ij}$ is positive definite
on the slit tangent bundle $\R^{m+1}\times (\R^{m+1}\setminus\{0\})$; see
condition (F2) in the introduction. For that purpose we fix $(x,y)
\in\R^{m+1}\times (\R^{m+1}\setminus\{0\})$, and by scaling we can assume
without loss of generality that the symmetric part $F_s$ of $F$ satisfies
$F_s(x,y)=1$, which we will use later to apply Lemma \ref{lem:4.1}.

For any $w\in\R^{m+1}$ there exist $\alpha,\beta\in\R$ such that $w=\alpha y+
\beta \xi$ for some vector $\xi$ satisfying $\xi\cdot (F_s)_y(x,y)=0$, since
one easily checks that $F_s>0$ and that $F_s$ is positively $1$-homogeneous,
which implies that the $m$-dimensional subspace $(F_s)_y(x,y)^\perp$ together with $y$
span all of $\R^{m+1}.$ One can also show that $F_s$ 
itself is a Finsler structure, a fact which we will use later on in the proof.
\detail{

\bigskip

NEW REASONING:

Fact 1: $\R^{m+1}=\span\{y,(F_s)_y(y)^\perp\}$, since $y\cdot (F_s)_y(y)=F_s(y)>0$
by homogeneity.

Fact 2: For all $w\in\R^{m+1}$ there are $\alpha,\beta\in\R$ and $\xi_w\in
(F_s)_y(y)^\perp\backslash \lbrace 0 \rbrace  $ such that $w=\alpha y+\beta\xi_w.$

Fact 3: Since $y\not\in(F_s)_y(y)^\perp$ and $y\not\in F_y(y)^\perp$ one finds for $\xi_w\in (F_s)_y(y)^\perp
\backslash \lbrace 0 \rbrace  $ a number $\gamma\in\R\setminus\{0\} $ and $\sigma\in\R$ and 
$\eta\in F_y(y)^\perp\backslash \lbrace 0 \rbrace  $ such that $\xi_w=\sigma y+\gamma\eta$, which immediately
implies
$$
F_{y^iy^j}(y)\xi_w^i\xi_w^j=\gamma^2\Big(\frac 12 F^2\Big)_{y^iy^j}\eta^i\eta^j>0,
$$
since $F$ is a Finsler metric.

Fact 4: Using all this we calculate for $w$ as in Fact 2 and $\xi_w$ as in Facts 2--4:
\begin{eqnarray}
 w^i (F_s^2/2)_{y^iy^j}w^j & = &\alpha^2 F_s^2 + \beta^2 F_s 
(F_s)_{y^iy^j}\xi_w^i\xi_w^j  > 0\label{F_s_pos_definite}
\end{eqnarray}
 for $w\neq 0$, i.e., $\alpha^2+\beta^2\not=0$ in Fact 2. 

WAHRSCHEINLICH kann man auf die ''$w$-Konstruktion`` verzichten, wenn man Ergebnisse zu parametrischen Integranden aus Giaquinta-Hildebrandt zu Hilfe nimmt.

\bigskip

}

In order to evaluate the quadratic form
\begin{equation}\label{4.5A}
g^f_{ij}w^iw^j=\alpha^2g^f_{ij}y^iy^j+2\alpha\beta g^f_{ij}y^i\xi^j+\beta^2g^f_{ij}\xi^i\xi^j
\end{equation}
at $(x,y)$ we look at the pure and mixed terms separately. Wherever we can we will
omit the fixed argument $(x,y)$.

By virtue of (F1) for $f=F_{\textnormal{sym}}$ we immediately obtain
\begin{equation}\label{4.6}
g^f_{ij}y^iy^j=y^i(f_{y^i}f_{y^j}+ff_{y^iy^j})y^j\overset{\textnormal{(F1)}}{=}f^2.
\end{equation}

Before handling the mixed terms in \eqref{4.5A} let us compute convenient formulas
for the $m$-harmonic symmetrization $f$ of $F$. Differentiating the defining
formula \eqref{m-harmonic} with respect to $y^j$ we deduce
\begin{eqnarray*}
(f^m)_{y^j} = mf^{m-1}f_{y^j} & = & -\frac{2}{(F^{-m}(x,y)+F^{-m}(x,-y))^2}\left[
\frac{-mF_{y^j}(x,y)}{F^{m+1}(x,y)}+
\frac{mF_{y^j}(x,-y)}{F^{m+1}(x,-y)}\right]\\
& = & -\frac m2 f^{2m} \left[
\frac{F_{y^j}(x,-y)}{F^{m+1}(x,-y)}-
\frac{F_{y^j}(x,y)}{F^{m+1}(x,y)}\right],
\end{eqnarray*}
which implies
\begin{equation}\label{4.7}
f_{y^j}=\frac 12 f^{m+1}\left[
\frac{F_{y^j}(x,y)}{F^{m+1}(x,y)}-
\frac{F_{y^j}(x,-y)}{F^{m+1}(x,-y)}\right].
\end{equation}
Differentiating \eqref{4.7} with respect to $y^i$ leads to the following formula
for the Hessian of $f$ at $(x,y)$.
\begin{eqnarray}
f_{y^iy^j} & = & \frac{m+1}{2}f^mf_{y^i}\left[
\frac{F_{y^j}(x,y)}{F^{m+1}(x,y)}-
\frac{F_{y^j}(x,-y)}{F^{m+1}(x,-y)}\right]
+\frac{f^{m+1}}{2}\left(
\frac{F_{y^iy^j}(x,y)}{F^{m+1}(x,y)}+\frac{F_{y^iy^j}(x,-y)}{F^{m+1}(x,-y)}\right.\notag\\
&&\qquad\qquad\qquad\left.
-\frac{(m+1)F_{y^j}(x,y)F_{y^i}(x,y)}{F^{m+2}(x,y)}
-\frac{(m+1)F_{y^j}(x,-y)F_{y^i}(x,-y)}{F^{m+2}(x,-y)}\right)\notag\\
&&\hspace{-2.0cm}\overset{\eqref{4.7}}{=}
\frac{m+1}{4} f^{2m+1}
\left[
\frac{F_{y^j}(x,y)}{F^{m+1}(x,y)}-
\frac{F_{y^j}(x,-y)}{F^{m+1}(x,-y)}\right]\left[
\frac{F_{y^i}(x,y)}{F^{m+1}(x,y)}-
\frac{F_{y^i}(x,-y)}{F^{m+1}(x,-y)}\right]\label{4.8}\\
&& \hspace{-2.0cm}+ \frac{f^{m+1}}{2}\left(\frac{F_{y^iy^j}(x,y)}{F^{m+1}(x,y)}+\frac{F_{y^iy^j}(x,-y)}{F^{m+1}(x,-y)}-(m+1)\left[\frac{F_{y^j}(x,y)F_{y^i}(x,y)}{F^{m+2}(x,y)}+
\frac{F_{y^j}(x,-y)F_{y^i}(x,-y)}{F^{m+2}(x,-y)}\right]\right)\notag
\end{eqnarray}
Concerning the mixed term in \eqref{4.5A} we use \eqref{4.7}, the identities
$y^if_{y^i}=f$, $F_{y^iy^j}(x,y)y^i=0$ both due to (F1) for $f$ and $F$, respectively, 
and $\xi\perp (F_s)_y(x,y)=
-(F_s)_y(x,-y)$
to find
\begin{eqnarray}
g^f_{ij}y^i\xi^j & \overset{\eqref{4.7}}{=} &  \frac 12 f^{m+2}\left[
\frac{F_{y^j}(x,y)}{F^{m+1}(x,y)}-
\frac{F_{y^j}(x,-y)}{F^{m+1}(x,-y)}\right]\xi^j\notag\\
& = & \frac 12 f^{m+2}\left[
\frac{1}{F^{m+1}(x,y)}-
\frac{1}{F^{m+1}(x,-y)}\right](F_a)_{y}(x,y)\cdot\xi,\label{4.9}
\end{eqnarray}
where we also used that the gradient $(F_a)_y$ of the antisymmetric part
is symmetric with respect to its
second entry.

For the last term in \eqref{4.5A} we use \eqref{4.7} and \eqref{4.8} to compute 
\begin{eqnarray}
g_{ij}^f\xi^i\xi^j & = & \frac{m+2}{4}f^{2m+2}\left[\frac{1}{F^{m+1}(x,y)}-
\frac{1}{F^{m+1}(x,-y)}\right]^2\Big((F_a)_{y}(x,y)\cdot\xi\Big)^2\notag\\
%&& + \frac{m+1}{4}f^{2m+2}\left[\frac{1}{F^{m+1}(x,y)}-
%\frac{1}{F^{m+1}(x,-y)}\right]^2\Big((F_a)_{y^l}(x,y)\xi^l\Big)^2\notag\\
&&
%\hspace{-2.5cm}
+\frac 12 f^{m+2}\left\{\frac{F_{y^iy^j}(x,y)\xi^i\xi^j}{F^{m+1}(x,y)}+\frac{F_{y^iy^j}(x,-y)
\xi^i\xi^j}{F^{m+1}(x,-y)}\right.\notag\\
&&\left.\qquad -(m+1)\left[\frac{1}{F^{m+2}(x,y)}+
\frac{1}{F^{m+2}(x,-y)}\right]\Big((F_a)_{y}(x,y)\cdot\xi\Big)^2\right\}\label{4.10}.
\end{eqnarray}
Inserting \eqref{4.6}, \eqref{4.9}, and \eqref{4.10} into \eqref{4.5A} we can write
for any $\epsilon>0$
\begin{eqnarray}
g_{ij}^fw^iw^j & = &
\left\{\alpha\epsilon f+\frac{\beta}{2\epsilon}f^{m+1}\left[\frac{1}{F^{m+1}(x,y)}-
\frac{1}{F^{m+1}(x,-y)}\right](F_a)_{y}(x,y)\cdot\xi\right\}^2+(1-\epsilon^2)\alpha^2f^2\notag\\
&& + \frac{\beta^2}{4}\left\{\Big(m+2-\frac{1}{\epsilon^2}\Big)f^{2m+2}
\left[\frac{1}{F^{m+1}(x,y)} -
\frac{1}{F^{m+1}(x,-y)}\right]^2\Big((F_a)_{y}(x,y)\cdot\xi\Big)^2\right\}\notag\\
&&+\frac{\beta^2f^{m+2}}{2}\left\{\frac{F_{y^iy^j}(x,y)\xi^i\xi^j}{F^{m+1}(x,y)}+\frac{F_{y^iy^j}(x,-y)
\xi^i\xi^j}{F^{m+1}(x,-y)}\right.\notag\\
&&\left.\qquad -(m+1)\left[\frac{1}{F^{m+2}(x,y)}+
\frac{1}{F^{m+2}(x,-y)}\right]\Big((F_a)_{y}(x,y)\cdot\xi\Big)^2\right\}\label{4.11}.
\end{eqnarray}
We should mention at this stage that  the now obvious condition 
$$
\Big((F_a)_y(x,y)\cdot\xi\Big)^2<\frac{1}{m+1} \xi\cdot F(x,y)F_{yy}(x,y)\xi\Foa\xi\in
(F_s)_y(x,y)^\perp
$$
to guarantee a positive right-hand side in \eqref{4.11} (for $
1\ge\epsilon^2\ge\frac{1}{m+2}$)
would be too restrictive as one can easily check in case of the Minkowski-Randers
metric $F(x,y)=|y|+b_iy^i$ for $m=2.$

Now we focus on the last three lines of the expression \eqref{4.11} for $g^f_{ij}w^iw^j$,
with the common factor $\frac{\beta^2}{2}f^{m+2}$, and define for 
\begin{equation}\label{deltaB}
2\delta\equiv 2\delta(m,\epsilon):=\frac{m+2}{2}-\frac{1}{2\epsilon^{2}}\quad\AND\quad B:=\Big((F_a)_{y}(x,y)\cdot\xi\Big)^2\
\end{equation}
the term
\begin{eqnarray}
P(y,\delta,m)&:=& 2\delta f^m\left[\frac{1}{F^{m+1}(x,y)} -
\frac{1}{F^{m+1}(x,-y)}\right]^2B\label{Pdelta}\\
&& +\frac{1}{F^{m+2}(x,y)}\{F(x,y)F_{y^iy^j}(x,y)\xi^i\xi^j-(m+1)B\}\notag\\
&& +\frac{1}{F^{m+2}(x,-y)}\{F(x,-y)F_{y^iy^j}(x,-y)\xi^i\xi^j-(m+1)B\}.\notag
\end{eqnarray}
To prove the theorem it will be sufficient in view of \eqref{4.11} 
to show that a suitably rescaled 
variant of $P(y,\delta,m)$ for some choice of $\epsilon$ (which determines
$\delta=\delta(m,\epsilon)$ according to \eqref{deltaB}) is strictly positive. 
For a more detailed analysis of this  expression we need to use the splitting $F=F_s+F_a$ in
the definition of the $m$-harmonic symmetrization $f=F_{\textnormal{sym}}$ to compute
\begin{equation}\label{fm}
f^m=\frac{2}{\frac{1}{F^m(x,y)}+
\frac{1}{F^m(x,-y)}}
=\frac{2F^m(x,y)F^m(x,-y)}{F^m(x,-y)+F^m(x,y)},
\end{equation}
where 
$$
F^m(x,y)=\sum_{k=0}^m
\left(\begin{array}{c}
m\\
k\end{array}\right)F_a^k(x,y)F_s^{m-k}(x,y),
$$
and by symmetry of $F_s$ and asymmetry of $F_a$ in $y$
$$
F^m(x,-y)=((-1)F_a(x,y)+F_s(x,y))^m=
\sum_{k=0}^m
\left(\begin{array}{c}
m\\
k\end{array}\right)(-1)^kF_a^k(x,y)F_s^{m-k}(x,y),
$$
such that
\begin{eqnarray*}
0<F^m(x,-y)+F^m(x,y)& = & \sum_{k=0}^m
\left(\begin{array}{c}
m\\
k\end{array}\right)F_a^k(x,y)F_s^{m-k}(x,y)\Big((-1)^k+1^k\Big)\\
& = &
2\sum_{l=0}^{\lfloor\frac m2 \rfloor}\left(\begin{array}{c}
m\\
2l\end{array}\right)F_a^{2l}(x,y)F_s^{m-2l}(x,y).
\end{eqnarray*}
Inserting this last expression into \eqref{fm} and the resulting term into 
\eqref{Pdelta} we find for
$$
Q(y,\delta,m):=F^{m+2}(x,y)F^{m+2}(x,-y)\sum_{l=0}^{\lfloor\frac m2 \rfloor}\left(\begin{array}{c}
m\\
2l\end{array}\right)F_a^{2l}(x,y)F_s^{m-2l}(x,y)P(y,\delta,m)
$$
the formula
\begin{eqnarray}
Q(y,\delta,m) & = & 2\delta \Big(F^{m+1}(x,y)-F^{m+1}(x,-y)\Big)^2B
\label{Qdelta}\\
&& \hspace{-1.2cm}+ \sum_{l=0}^{\lfloor\frac m2 \rfloor}\left(\begin{array}{c}
m\\
2l\end{array}\right)F_a^{2l}(x,y)F_s^{m-2l}\Big\{
F^{m+2}(x,-y)\Big[F(x,y)F_{y^iy^j}(x,y)\xi^i\xi^j-(m+1)B
\Big]\notag\\ 
&& \qquad\qquad +F^{m+2}(x,y)\Big[F(x,-y)F_{y^iy^j}(x,-y)\xi^i\xi^j-(m+1)B
\Big]\Big\} \notag
\end{eqnarray}
With the same splitting $F=F_s+F_a$ as before we can express the square
$$
\Big(F^{m+1}(x,y)-F^{m+1}(x,-y)\Big)^2=4\left[\sum_{l=0}^{\lfloor \frac m2 \rfloor}
\left(\begin{array}{c}
m+1\\
2l+1\end{array}\right)F_a^{2l+1}(x,y)F_s^{m-2l}\right]^2,
$$
and the powers
$$
F^{m+2}(x,-y)=\sum_{k=0}^{m+2}
\left(\begin{array}{c}
m+2\\
k\end{array}\right)(-1)^kF_a^{k}(x,y)F_s^{m+2-k}(x,y),
$$
and 
$$
F^{m+2}(x,y)=\sum_{k=0}^{m+2}
\left(\begin{array}{c}
m+2\\
k\end{array}\right)F_a^{k}(x,y)F_s^{m+2-k}(x,y),
$$
%whose sum and difference turn therefore out to be
%\begin{eqnarray*}
%F^{m+2}(x,y)+F^{m+2}(x,-y) & = & 2\sum_{n=0}^{\lfloor \frac m2 \rfloor +1}
%\left(\begin{array}{c}
%m+2\\
%2nend{array}\right)F_a^{2n}(x,y)F_s^{m+2-2n}(x,y),
%\\
%F^{m+2}(x,y)-F^{m+2}(x,-y) & = & 2\sum_{n=0}^{
as well as the Hessian expressions
\begin{eqnarray*}
F_{y^iy^j}(x,y)F(x,y) &= &(F_s(x,y)+F_a(x,y))\Big[(F_s)_{y^iy^j}(x,y)+
(F_a)_{y^iy^j}(x,y)\Big],\\
F_{y^iy^j}(x,-y)F(x,-y) &= &(F_s(x,y)-F_a(x,y))\Big[(F_s)_{y^iy^j}(x,y)-
(F_a)_{y^iy^j}(x,y)\Big]
\end{eqnarray*}
to rewrite \eqref{Qdelta} as
\begin{eqnarray}
Q(y,\delta,m) & = & H+B\left\{8\delta\left[\sum_{l=0}^{\lfloor \frac m2 \rfloor}
\left(\begin{array}{c}
m+1\\
2l+1\end{array}\right)F_a^{2l+1}F_s^{m-2l}\right]^2\right.\label{Qlang}\\
&& \hspace{-3cm}\left.-(m+1)
\sum_{l=0}^{\lfloor \frac m2 \rfloor}
\left(\begin{array}{c}
m\\
2l\end{array}\right)F_a^{2l}F_s^{m-2l}
\sum_{k=0}^{m+2}
\left(\begin{array}{c}
m+2\\
k\end{array}\right)F_a^kF_s^{m+2-k}\Big((-1)^k+1^k\Big)\right\}\notag,
\end{eqnarray}
where the fixed argument $(x,y)$ is suppressed from now on, and where
we have abbreviated all terms involving the Hessians $(F_s)_{y^iy^j}$ and
$(F_a)_{y^iy^j}$ by $H$, which may be written as
\begin{eqnarray}
H & = & 
\sum_{l=0}^{\lfloor \frac m2 \rfloor}
\left(\begin{array}{c}
m\\
2l\end{array}\right)F_a^{2l}F_s^{m-2l}
\left\{\sum_{k=0}^{m+2}
\left(\begin{array}{c}
m+2\\
k\end{array}\right)F_a^{k}F_s^{m+2-k}\right.\label{Hessian}\\
&&
%\hspace{-0.3cm}
\cdot
\Big(\Big[F_s(F_s)_{y^iy^j}+
F_a(F_a)_{y^iy^j}\Big]((-1)^k+1^k)-
\Big[F_s(F_a)_{y^iy^j}+F_a(F_s)_{y^iy^j}\Big](1^k-(-1)^k)\Big)
\Bigg\}.\notag
\end{eqnarray}
With the identities
\begin{equation}\label{firstbinomial}
\sum_{k=0}^{m+2}
\left(\begin{array}{c}
m+2\\
k\end{array}\right)F_a^{k}F_s^{m+2-k}((-1)^k+1^k)=2
\sum_{n=0}^{\lfloor \frac m2 \rfloor +1}
\left(\begin{array}{c}
m+2\\
2n\end{array}\right)F_a^{2n}F_s^{m+2-2n}
\end{equation}
and
\begin{eqnarray*}
\sum_{k=0}^{m+2}
\left(\begin{array}{c}
m+2\\
k\end{array}\right)F_a^{k}F_s^{m+2-k}(1^k-(-1)^k)&=& 2
\sum_{l=0}^{\lfloor \frac{m-1}{2}\rfloor +1}\left(\begin{array}{c}
m+2\\
2l+1\end{array}\right)F_a^{2l+1}F_s^{m+2-(2l+1)}\\
&=& 2
\sum_{n=0}^{\lfloor \frac{m+1}{2}\rfloor +1}\left(\begin{array}{c}
m+2\\
2n-1\end{array}\right)F_a^{2n-1}F_s^{m+2-(2n-1)}
\end{eqnarray*}
(recall that the binomial for $n=0$ vanishes in the last sum),
we can regroup 
$$
F_aF_a^{2n-1}F_s^{m+2-(2n-1)}=F_a^{2n}F_s^{m-2n}F_s
$$
to summarize all terms within the braces in \eqref{Hessian} involving the symmetric
Hessian $(F_s)_{yy}$ as
$$
2\sum_{n=0}^{\lfloor \frac{m+1}{2}\rfloor +1}
\left\{
\left(\begin{array}{c}
m+2\\
2n\end{array}\right)-
\left(\begin{array}{c}
m+2\\
2n-1\end{array}\right)\right\}F_a^{2n}F_s^{m+2-2n}\Big(F_s(F_s)_{y^iy^j}\xi^i\xi^j\Big),
$$
where we also used the fact that 
$$
\left(\begin{array}{c}
m+2\\
2n\end{array}\right)=0\quad\Fo n=\lfloor \frac{m+1}{2}\rfloor +1.
$$
In an analogous fashion we can summarize all terms involving the
asymmetric Hessian $(F_a)_{yy}$ in \eqref{Hessian} to obtain 
\begin{eqnarray}
H & = & 
2\sum_{l=0}^{\lfloor \frac m2 \rfloor}
\left(\begin{array}{c}
m\\
2l\end{array}\right)F_a^{2l}F_s^{m-2l}
\Bigg\{
\Big(F_s(F_s)_{y^iy^j}\xi^i\xi^j\Big)\label{Hessianbetter}\\
&& \cdot
\sum_{n=0}^{\lfloor \frac{m+1}{2}\rfloor +1}
\left\{
\left(\begin{array}{c}
m+2\\
2n\end{array}\right)-
\left(\begin{array}{c}
m+2\\
2n-1\end{array}\right)\right\}F_a^{2n}F_s^{m+2-2n}\notag\\
&& +\Big(F_a(F_a)_{y^iy^j}\xi^i\xi^j\Big)
\sum_{n=0}^{\lfloor \frac m2 \rfloor +1}
\left\{
\left(\begin{array}{c}
m+2\\
2n\end{array}\right)-
\left(\begin{array}{c}
m+2\\
2n+1\end{array}\right)\right\}F_a^{2n}F_s^{m+2-2n}\Bigg\}.\notag
\end{eqnarray}
Recall our initial scaling
$F_s=F_s(x,y)=1$, which implies $|F_a|=|F_a(x,y)|<1$ since  $F(x,y)>0$ and $F(x,-y)>0$ lead 
to $|F_a|<F_s$ by definition. But the assumption of Theorem \ref{thm:sufficient} implies more:
If one takes $w:=y$ in \eqref{suff_ineq}, one can use homogeneity to find
\begin{eqnarray}
F_a^2(x,y)=\Big((F_a)_y(x,y)\cdot y\Big)^2 &\overset{\eqref{suff_ineq}}{<} &
\frac{1}{m+1} (g_{F_s})_{ij}(x,y)y^iy^j\notag\\
& = & \frac{1}{m+1}\Big((F_s)_y\cdot y\Big)^2=F_s^2(x,y)=\frac{1}{m+1},\label{Faineq}
\end{eqnarray}
so that one can apply Lemma \ref{lem:4.1} for $a:=|F_a(x,y)|<(m+1)^{-1/2}$
replacing the $m$ in that lemma by 
$m+2$ to find that the last line in \eqref{Hessianbetter} is non-negative since the matrix
$F_a(F_a)_{yy}$ is negative semi-definite by assumption\footnote{If $F_a(x,y)$ happens to vanish
then Lemma \ref{lem:4.1} is not applicable but the last line in \eqref{Hessianbetter} vanishes 
anyway}.

We use the resulting inequality for $H$ and 
fix $\epsilon:=1 $ in \eqref{4.11}
such that  
$8\delta=2(m+1)$ by means of \eqref{deltaB} to obtain comparable terms in \eqref{Qlang}
to estimate
\begin{eqnarray}
\frac 12 Q\equiv \frac 12 Q(y,\frac{m+1}{4},m) & \ge & 
(m+1)B\left\{\left[\sum_{l=0}^{\lfloor \frac m2 \rfloor}
\left(\begin{array}{c}
m+1\\
2l+1\end{array}\right)F_a^{2l+1}\right]^2\notag\right.\\
&&-\left.
\sum_{l=0}^{\lfloor \frac m2 \rfloor}
\left(\begin{array}{c}
m\\
2l\end{array}\right)F_a^{2l}\sum_{n=0}^{\lfloor \frac m2 \rfloor+1}
\left(\begin{array}{c}
m+2\\
2n\end{array}\right)F_a^{2n}\right\}\label{Qbetter}\\
&& \hspace{-3cm}+(F_s)_{y^iy^j}\xi^i\xi^j\sum_{l=0}^{\lfloor \frac m2 \rfloor}
\left(\begin{array}{c}
m\\
2l\end{array}\right)F_a^{2l}
\sum_{n=0}^{\lfloor \frac{m+1}{2}\rfloor +1}\left\{
\left(\begin{array}{c}
m+2\\
2n\end{array}\right)-
\left(\begin{array}{c}
m+2\\
2n-1\end{array}\right)\right\}F_a^{2n}.\notag
\end{eqnarray}
One can check directly by virtue of \eqref{suff_ineq}
that $Q>0$ if $F_a(y)$ happens to vanish, since then 
$$
Q/2=-(m+1)B+(F_s)_{y^iy^j}\xi^i\xi^j=-(m+1)B+g^{F_s}_{ij}\xi^i\xi^j\overset{\eqref{suff_ineq}}{>} 0,
$$
(recall that $\xi\in (F_s)_y^\perp$)
so that we may assume from now on that $|F_a|\in (0,(m+1)^{-1/2}).$

To produce comparable terms in the first term on the right-hand side
of \eqref{Qbetter} we use the well-known binomial identity
\begin{equation}\label{binomial_three}
\left(\begin{array}{c}
n+1\\
k+1\end{array}\right)=
\left(\begin{array}{c}
n\\
k+1\end{array}\right)+
\left(\begin{array}{c}
n\\
k\end{array}\right)
\end{equation}
to write for the first sum on the right-hand side of \eqref{Qbetter}
\begin{multline}
\left[\sum_{l=0}^{\lfloor \frac m2 \rfloor}
\left(\begin{array}{c}
m+1\\
2l+1\end{array}\right)F_a^{2l+1}\right]^2
=\left[\sum_{l=0}^{\lfloor \frac m2 \rfloor}
\left(\begin{array}{c}
m\\
2l+1\end{array}\right)F_a^{2l+1}\right]^2\label{squareterm}\\
+2\sum_{l=0}^{\lfloor \frac m2 \rfloor}
\left(\begin{array}{c}
m\\
2l\end{array}\right)F_a^{2l+1}
\sum_{n=0}^{\lfloor \frac m2 \rfloor}
\left(\begin{array}{c}
m\\
2n+1\end{array}\right)F_a^{2n+1}
+
\left[\sum_{l=0}^{\lfloor \frac m2 \rfloor}
\left(\begin{array}{c}
m\\
2l\end{array}\right)F_a^{2l+1}\right]^2.
\end{multline}
Substituting $2n+1=2k-1$ in the product of sums in the second line and regrouping
the powers of $F_a$ we can
rewrite this product as
$$
2\sum_{l=0}^{\lfloor \frac m2 \rfloor}
\left(\begin{array}{c}
m\\
2l\end{array}\right)F_a^{2l}
\sum_{k=0}^{\lfloor \frac m2 \rfloor+1}
\left(\begin{array}{c}
m\\
2k-1\end{array}\right)F_a^{2k}.
$$
Similarly, the substitution $2l=2k-2$ for the last square of sums  in \eqref{squareterm}
leads to
\begin{multline}
\left[\sum_{l=0}^{\lfloor \frac m2 \rfloor}
\left(\begin{array}{c}
m+1\\
2l+1\end{array}\right)F_a^{2l+1}\right]^2-\sum_{l=0}^{\lfloor \frac m2 \rfloor}
\left(\begin{array}{c}
m\\
2l\end{array}\right)F_a^{2l}
\sum_{k=0}^{\lfloor \frac m2 \rfloor+1}
\left(\begin{array}{c}
m+2\\
2k\end{array}\right)F_a^{2k}
=\left[\sum_{l=0}^{\lfloor \frac m2 \rfloor}
\left(\begin{array}{c}
m\\
2l+1\end{array}\right)F_a^{2l+1}\right]^2\label{squaretermbetter}\\
-\sum_{l=0}^{\lfloor \frac m2 \rfloor}
\left(\begin{array}{c}
m\\
2l\end{array}\right)F_a^{2l}\left\{
\sum_{k=0}^{\lfloor \frac m2 \rfloor+1}
\left\{
\left(\begin{array}{c}
m+2\\
2k\end{array}\right)-2
\left(\begin{array}{c}
m\\
2k-1\end{array}\right)-
\left(\begin{array}{c}
m\\
2k-2\end{array}\right)\right\}F_a^{2k}\right\},
\end{multline}
where one can use successively the binomial identities
\begin{eqnarray*}
2\left(\begin{array}{c}
m\\
2k-1\end{array}\right)+\left(\begin{array}{c}
m\\
2k-2\end{array}\right)&= &
\left(\begin{array}{c}
m\\
2k-1\end{array}\right)+\left(\begin{array}{c}
m\\
2k-1\end{array}\right)+\left(\begin{array}{c}
m\\
2k-2\end{array}\right)\\
&    = & \left(\begin{array}{c}
m\\
2k-1\end{array}\right)+\left(\begin{array}{c}
m+1\\
2k-1\end{array}\right),
\end{eqnarray*}
and then
$$
\left(\begin{array}{c}
m+2\\
2k\end{array}\right)-
\left(\begin{array}{c}
m+1\\
2k-1\end{array}\right)=
\left(\begin{array}{c}
m+1\\
2k\end{array}\right),
$$
and finally 
$$
\left(\begin{array}{c}
m+1\\
2k\end{array}\right)-
\left(\begin{array}{c}
m\\
2k-1\end{array}\right)=
\left(\begin{array}{c}
m\\
2k\end{array}\right)
$$
to deduce
\begin{multline}\label{Qbest}
\left[\sum_{l=0}^{\lfloor \frac m2 \rfloor}
\left(\begin{array}{c}
m+1\\
2l+1\end{array}\right)F_a^{2l+1}\right]^2
-\sum_{l=0}^{\lfloor \frac m2 \rfloor}
\left(\begin{array}{c}
m\\
2l\end{array}\right)F_a^{2l}
\sum_{k=0}^{\lfloor \frac m2 \rfloor+1}
\left(\begin{array}{c}
m+2\\
2k\end{array}\right)F_a^{2k}\\
=\left[\sum_{l=0}^{\lfloor \frac m2 \rfloor}
\left(\begin{array}{c}
m\\
2l+1\end{array}\right)F_a^{2l+1}\right]^2
-
\left[\sum_{l=0}^{\lfloor \frac m2 \rfloor}
\left(\begin{array}{c}
m\\
2l\end{array}\right)F_a^{2l}\right]^2,
\end{multline}
the right-hand side of which is negative since  $F_a^2<(m+1)^{-1}$ which
can be used in each term of the left sum to obtain
$$
\left(\begin{array}{c}
m\\2l+1\end{array}\right)F_a^2<
\left(\begin{array}{c}
m\\2l\end{array}\right)
$$ 
for each $l=0,\ldots,\lfloor m/2 \rfloor.$
We have used before that our central assumption \eqref{suff_ineq} leads
to
$(m+1)B<(F_s)_{y^iy^j}\xi^i\xi^j$ which in combination with 
\eqref{Qbetter} and  \eqref{Qbest} leads to the strict inequality
\begin{multline}\label{Qalmostfinal}
\frac 12 Q  >  (F_s)_{y^iy^j}\xi^i\xi^j\left\{
\left[\sum_{l=0}^{\lfloor \frac m2 \rfloor}
\left(\begin{array}{c}
m\\
2l+1\end{array}\right)F_a^{2l+1}\right]^2\right.\\
\left.
+\sum_{l=0}^{\lfloor \frac m2 \rfloor}
\left(\begin{array}{c}
m\\
2l\end{array}\right)F_a^{2l}
\sum_{k=0}^{\lfloor \frac{m+1}{2}+1\rfloor}
\left\{
\left(\begin{array}{c}
m+2\\
2k\end{array}\right)-
\left(\begin{array}{c}
m+2\\
2k-1\end{array}\right)-
\left(\begin{array}{c}
m\\
2k\end{array}\right)\right\}F_a^{2k}\right\},
\end{multline}
where we have also used that for  $n>\lfloor m/2 \rfloor$ the binomial
$\left(\begin{array}{c}
m\\2n\end{array}\right)$ vanishes.
Repeatedly using \eqref{binomial_three} we can reduce the
difference of three binomials in \eqref{Qalmostfinal} to
$$
\left(\begin{array}{c}
m\\
2k-1\end{array}\right)-
\left(\begin{array}{c}
m+1\\
2k-2\end{array}\right),
$$
and the substitution $2k-1=2l+1$ then leads to 
$$
\sum_{l=0}^{\lfloor \frac{m+1}{2} \rfloor}\left\{
\left(\begin{array}{c}
m\\2l+1\end{array}\right)-
\left(\begin{array}{c}
m+1\\2l\end{array}\right)
\right\}F_a^{2l+2}
$$
for the last sum in \eqref{Qalmostfinal}. Now adding and subtracting
equal products of sums we arrive at
\begin{multline*}
\frac 12 Q > (F_s)_{y^iy^j}\xi^i\xi^j\left\{
\sum_{l=0}^{\lfloor \frac m2 \rfloor}
\left(\begin{array}{c}
m\\2l+1\end{array}\right)F_a^{2l+1}\left[
\sum_{k=0}^{\lfloor \frac m2 \rfloor}
\left\{
\left(\begin{array}{c}
m\\2k+1\end{array}\right)-
\left(\begin{array}{c}
m\\2k\end{array}\right)\right\}F_a^{2k+1}\right]\right.\\
\left. +
\sum_{l=0}^{\lfloor \frac m2 \rfloor}
\left(\begin{array}{c}
m\\2l\end{array}\right)F_a^{2l+1}\left[
\sum_{k=0}^{\lfloor \frac{m+1}{2} \rfloor}
\left\{
\left(\begin{array}{c}
m\\2k+1\end{array}\right)+\left(\begin{array}{c}
m\\2k\end{array}\right)\right.\right.\right.\\
\left.\left.\left. -
\left(\begin{array}{c}
m+1\\2k\end{array}\right)+
\left(\begin{array}{c}
m\\2k+1\end{array}\right)-
\left(\begin{array}{c}
m\\2k\end{array}\right)
\right\}F_a^{2k+2}\right]
\right\}.
\end{multline*}
Using \eqref{binomial_three} for the first two binomials in the last
sum, and a shift of indices as before allows us to take out
a factor 
$$
\sum_{k=0}^{\lfloor \frac{m}{2} \rfloor}
\left\{
\left(\begin{array}{c}
m\\2k+1\end{array}\right)-
\left(\begin{array}{c}
m\\2k\end{array}\right)
\right\}F_a^{2k+1}
$$
and regroup powers of $F_a$ to obtain 
\begin{multline*}
\frac 12 Q > (F_s)_{y^iy^j}\xi^i\xi^j\left\{
\sum_{l=0}^{\lfloor \frac m2 \rfloor}
\left(\begin{array}{c}
m+1\\
2l+1\end{array}\right)F_a^{2l}
\left[
\sum_{k=0}^{\lfloor \frac m2 \rfloor}
\left\{
\left(\begin{array}{c}
m\\2k+1\end{array}\right)-
\left(\begin{array}{c}
m\\2k\end{array}\right)
\right\}F_a^{2k+2}
\right]\right.\\
+\left. \sum_{l=0}^{\lfloor \frac m2 \rfloor}
\left(\begin{array}{c}
 m\\
2l\end{array}\right)F_a^{2l}
\left[
\sum_{k=0}^{\lfloor \frac{m+1}{2} \rfloor}
\left\{
\left(\begin{array}{c}
m+1\\2k+1\end{array}\right)-
\left(\begin{array}{c}
m+1\\
2k\end{array}\right)
\right\}F_a^{2k+2}
\right]\right\}.
\end{multline*}
Both sums over $k$ on the right-hand side are non-negative
according to Lemma \ref{lem:4.1}, which finally proves
Theorem \ref{thm:sufficient}
\qed

\addcontentsline{toc}{section}{References}

\bibliography{finsler-plateau}{}
\bibliographystyle{plain}

%%%%%%%%%%%%%%%%%%%%%%%%%%%%%%%%%%%%%%%%%%%%%
%%%%%%%%%%%%%%%   addresses  %%%%%%%%%%%%%%%%
%%%%%%%%%%%%%%%%%%%%%%%%%%%%%%%%%%%%%%%%%%%%%%

\small
\vspace{1cm}
\begin{minipage}{56mm}
{\sc Patrick Overath}\\
Institut f\"ur Mathematik\\
RWTH Aachen University\\
Templergraben 55\\
D-52062 Aachen\\
GERMANY\\
E-mail: {\tt overath@}\\
{\tt instmath.rwth-aachen.de}
\end{minipage}
\hfill
\begin{minipage}{56mm}
{\sc Heiko von der Mosel}\\
Institut f\"ur Mathematik\\
RWTH Aachen University\\
Templergraben 55\\
D-52062 Aachen\\
GERMANY\\
Email: {\tt heiko@}\\{\tt instmath.rwth-aachen.de}
\end{minipage}

\end{document}